\documentclass[a4paper,twoside]{article}

\usepackage{amssymb,amsmath,amsthm,dsfont,amsfonts,color,latexsym}

\numberwithin{equation}{section}
\theoremstyle{plain}
\newtheorem{thm}{Theorem}[section]
\newtheorem{theorem}{Theorem}[section]      
\newtheorem{lemma}[theorem]{Lemma}          
\newtheorem{proposition}[theorem]{Proposition}

\theoremstyle{definition}

\theoremstyle{definition}
\newtheorem{remark}[theorem]{Remark}

\usepackage{hyperref}
\usepackage{mathrsfs} 

\definecolor{blue}{rgb}{0.00,0.00,1.00}
\definecolor{red}{rgb}{1.00,0.00,0.00}

\renewcommand{\baselinestretch}{1.2}
\begin{document}


\title{Global Stability and Energy Growth in the Sheared Vlasov--Poisson--Boltzmann System}
\author{
  Zhida Chang$^1$,\, Mingying Zhong$^{1,2}$\\[2mm]
  {\small\it $^1$School of Mathematics, Guangxi University, China.}\\
  {\small\it E-mail: czdzhidachang@163.com}\\
  {\small\it $^2$Center for Applied Mathematical of Guangxi (Guangxi University), Guangxi University, China.}\\
  {\small\it E-mail: zhongmingying@gxu.edu.cn}\\[5mm]
}
\date{ }

\pagestyle{myheadings}
\markboth{Sheared Vlasov--Poisson--Boltzmann System}{Z.-D.~Chang, M.-Y.~Zhong}

\maketitle

\thispagestyle{empty}

\begin{abstract}\noindent{We study the Vlasov--Poisson--Boltzmann system on the three-dimensional torus under uniform shear flow. For Maxwell molecules with Grad's angular cutoff and sufficiently small shear rate, we consider perturbations around the spatially homogeneous self-similar profile of the sheared Boltzmann equation. In self-similar variables, we prove the global stability of this profile, including global existence and uniqueness, the exponential decay of nonzero-order spatial derivatives in weighted $L^\infty$ spaces, and the uniform boundedness of the renormalized perturbation. The analysis combines Caflisch's decomposition, Guo's $L^\infty$--$L^2$ estimates, macro--micro analysis, and a spectral study of the zero-frequency mode. We further derive a closed zero-frequency system for a renormalized total energy and suitable second-order moments, which yields a precise large-time description of the shear-induced energy growth.}


\end{abstract}


\tableofcontents  
\section{Introduction}
\setcounter{equation}{0}

In this paper, we study the dynamics of a dilute gas of a single species of charged particles, such as electrons, driven by a self-consistent electric field. The corresponding one-species Vlasov--Poisson--Boltzmann (VPB) system on the three-dimensional torus is given by
\begin{equation}\label{VPB-0}
\left\{
\begin{aligned}
& \partial_t F + v\cdot\nabla_x F + \nabla_x\phi\cdot\nabla_v F = Q(F,F),\\
& \Delta_x \phi = \int_{\mathbb R^3} F(t,x,v)\,dv - 1,\quad \int_{\mathbb T^3}\phi\,dx=0.
\end{aligned}
\right.
\end{equation}
Here $(t,x,v)\in \mathbb R_+\times\mathbb T^3\times\mathbb R^3$, with
$\mathbb T^3=(\mathbb R/\mathbb Z)^3$, $F=F(t,x,v)\ge 0$ is the distribution function of the charged particles, and $\phi=\phi(t,x)$ is the electric potential. We have normalized the constant background charge density as $\bar\rho=1$.

The Boltzmann collision operator is defined by
\begin{equation}\label{Q-def}
Q(F_1,F_2)(v):=Q_+(F_1,F_2)(v)-Q_-(F_1,F_2)(v),
\end{equation}
where the gain and loss parts are given by
\begin{equation}\label{Q-def-dist}
\left\{
\begin{aligned}
Q_+(F_1,F_2)(v)
&=
\int_{\mathbb R^3}\int_{\mathbb S^2}
B(|v-v_*|,\cos\theta)F_1(v')F_2(v_*')
\,d\omega dv_*,\\
Q_-(F_1,F_2)(v)
&=
\int_{\mathbb R^3}\int_{\mathbb S^2}
B(|v-v_*|,\cos\theta)F_1(v)F_2(v_*)
\,d\omega dv_*.
\end{aligned}
\right.
\end{equation}
Here $\omega\in\mathbb S^2$ denotes the collision direction, and the scattering angle $\theta$ is defined by
$$
\cos\theta=\frac{v-v_*}{|v-v_*|}\cdot\omega.
$$
The post-collisional velocities $v'$ and $v_*'$ are given by
\begin{equation}\label{post-pre-v}
v'=v-\big[(v-v_*)\cdot\omega\big]\omega,\quad
v_*'=v_*+\big[(v-v_*)\cdot\omega\big]\omega.
\end{equation}
In particular, momentum and kinetic energy are conserved
\[
v+v_*=v'+v_*',\quad
|v|^2+|v_*|^2=|v'|^2+|v_*'|^2.
\]

Throughout this paper, we consider the Maxwell molecule model
\begin{equation}\label{B-Maxwell}
B(|v-v_*|,\cos\theta)=B_0(\cos\theta),
\end{equation}
and assume that $B_0$ satisfies Grad's angular cutoff condition
\begin{equation}\label{Grad-cutoff}
0\le B_0(\cos\theta)\le C|\cos\theta|.
\end{equation}

In many physical situations, the gas evolves under the influence of a
nonstationary background flow. Motivated by the homoenergetic formulation of
uniform shear flow (USF) developed in \cite{JamesNotaVelazquez19b}, in which
the simple-shear deformation gives rise to the velocity-space operator
$-\alpha v_2\partial_{v_1}$, we consider the following one-species USF--VPB
system on the three-dimensional torus:
\begin{equation}
\label{VPB-3D-Cauchy}
\left\{
\begin{aligned}
&\partial_tF+v\cdot\nabla_xF
+\nabla_x\phi\cdot\nabla_vF
-\alpha v_2\partial_{v_1}F
=Q(F,F),\\
&\Delta_x\phi
=\int_{\mathbb R^3}F(t,x,v)\,dv-1,
\quad
\int_{\mathbb T^3}\phi\,dx=0,\\
&F(0,x,v)=F_0(x,v).
\end{aligned}
\right.
\end{equation}
Here $\alpha>0$ denotes the shear rate.

In the spatially homogeneous case, the charge-neutrality condition and the
normalization of $\phi$ imply $\phi=0$. Consequently,
\eqref{VPB-3D-Cauchy} reduces to the homogeneous USF Boltzmann equation
\begin{equation}
\label{USF-BE}
\partial_tF-\alpha v_2\partial_{v_1}F=Q(F,F).
\end{equation}
Let $F(t,v)$ be a solution to \eqref{USF-BE}, and assume that it decays sufficiently fast with respect to $v$. Multiplying \eqref{USF-BE} by $1$, $v_i$, and $|v|^2$, respectively, and then integrating in $v$, we obtain the moment equations
\begin{equation}\label{energy-usf}
\left\{
\begin{aligned}
&\frac{d}{dt}\int_{\mathbb R^3} F\, dv = 0,\\
&\frac{d}{dt}\int_{\mathbb R^3} v_1F\, dv + \alpha\int_{\mathbb R^3} v_2F\, dv = 0,\\
&\frac{d}{dt}\int_{\mathbb R^3} v_iF\, dv = 0,\quad i=2,3,\\
&\frac{d}{ dt}\int_{\mathbb R^3} |v|^2F\, dv + 2\alpha\int_{\mathbb R^3} v_1v_2F\, dv = 0.
\end{aligned}
\right.
\end{equation}
It follows from \eqref{energy-usf} that, without loss of generality, we may assume that the mass and momentum satisfy
\begin{equation}\label{mass-momentum-normal}
\int_{\mathbb R^3} F(t,v)\, dv = 1,\quad
\int_{\mathbb R^3} v_iF(t,v)\, dv = 0,\quad i=1,2,3,\quad \forall\, t\ge 0.
\end{equation}
The last identity in \eqref{energy-usf} shows that the shear term introduces a viscous heating mechanism through the stress moment $\displaystyle\int v_1v_2F\,{\rm d}v$, so that the energy scale of the system changes in time. Consequently, the large-time behavior of the solution cannot simply be understood as convergence to a fixed Maxwellian. Precisely because of this nonequilibrium effect continuously driven by shear, the introduction of a self-similar scaling is a natural way to characterize the long-time dynamics.

In the Maxwell molecule case, we consider exponentially self-similar solutions (see, for example, \cite[Chapter~2]{GarzoSantos03} and \cite[Section~5.1]{JamesNotaVelazquez19a})
\begin{equation}\label{self-similar-F}
F(t,v)=e^{-3\beta t}G\Big(\frac{v}{e^{\beta t}}\Big),
\end{equation}
where $\beta>0$ is a constant and $G=G(v)$ is a time-independent profile. Substituting \eqref{self-similar-F} into \eqref{USF-BE}, we obtain the profile equation
\begin{equation}\label{self-similar-G}
-\beta\nabla_v\cdot(vG)-\alpha v_2\partial_{v_1}G
=Q(G,G).
\end{equation}
Moreover, $G$ satisfies the normalization conditions
\begin{equation}\label{G-normal}
\int_{\mathbb R^3}G(v)\,dv=1,\quad
\int_{\mathbb R^3}vG(v)\,dv=0,\quad
\int_{\mathbb R^3}|v|^2G(v)\,dv=3.
\end{equation}
Under the constraint \eqref{G-normal}, the scaling rate $\beta$ can be equivalently determined by
\begin{equation}\label{beta-def}
\beta
=-\alpha\frac{\int_{\mathbb R^3}v_1v_2G(v)\,dv}{\int_{\mathbb R^3}|v|^2G(v)\,dv}
=-\frac{\alpha}{3}\int_{\mathbb R^3}v_1v_2G(v)\, dv.
\end{equation}

When the shear rate $\alpha$ is sufficiently small, \cite{DuanLiu21} shows that there exists a unique smooth profile $G$ solving \eqref{self-similar-G}--\eqref{G-normal}. More precisely, $G$ admits the first-order expansion
\begin{equation}\label{def-G}
G(v)=\mu(v)-\frac{\alpha}{2b_0}v_1v_2\mu(v)+O(\alpha^2),
\end{equation}
and therefore
\begin{equation}\label{beta}
\beta=O(\alpha^2),
\end{equation}
where $\mu(v)$ denotes the global Maxwellian with unit density, zero mean velocity, and unit temperature
\[
\mu(v)=(2\pi)^{-\frac32}e^{-\frac{|v|^2}{2}},
\]
and the constant $b_0>0$ is defined by
\begin{equation}\label{def-b0}
b_0 = 3\pi \int_{-1}^{1} B_0(z) z^2(1-z^2)\, dz >0.
\end{equation}

The study of homoenergetic flows at the level of moment equations goes
back to the pioneering works of Galkin and Truesdell
\cite{Galkin58,Truesdell56}, with a systematic account given in
\cite{TruesdellMuncaster80}. Cercignani later proved the existence of
homoenergetic affine flows and studied granular shear flows
\cite{Cercignani89,Cercignani01,Cercignani02}, while Bobylev, Caraffini,
and Spiga investigated the corresponding group-invariant solutions
\cite{BobylevCaraffiniSpiga96}. Monte Carlo simulations of USF were
carried out in \cite{MontaneroSantosGarzo96}, while anisotropic algebraic
high-velocity tails were derived for a two-dimensional grazing-collision
model in \cite{Acedo02}; a systematic treatment of kinetic shear flows
and nonlinear transport is given in \cite{GarzoSantos03}.

Self-similar solutions of spatially homogeneous Maxwell-type Boltzmann
equations were studied by Bobylev and Cercignani
\cite{BobylevCercignani02}. In the homoenergetic setting, James, Nota,
and Vel\'azquez constructed self-similar profiles and analyzed their
velocity distributions and entropy properties
\cite{JamesNotaVelazquez19a}. They later carried out formal large-time
analyses in the collision-dominated and hyperbolic-dominated regimes
\cite{JamesNotaVelazquez19b,JamesNotaVelazquez20}. Based on Bobylev's
Fourier formulation for Maxwell molecules \cite{Bobylev75}, Bobylev,
Nota, and Vel\'azquez proved convergence toward self-similar profiles
for a modified Maxwell--Boltzmann equation with a sufficiently small
deformation matrix \cite{BobylevNotaVelazquez20}. For non-cutoff Maxwell
molecules, Kepka established the existence, uniqueness, stability, and
regularity of the corresponding self-similar profiles \cite{Kepka23}.

For the Boltzmann equation under USF, Duan and Liu proved the existence,
uniqueness, regularity, nonnegativity, and exponential stability of
self-similar profiles for Maxwell molecules at small shear rates
\cite{DuanLiu21}. They later established global existence for the
spatially homogeneous Cauchy problem with cutoff hard potentials and
justified a uniform-in-time asymptotic expansion under the homoenergetic
scaling \cite{DuanLiu25}. The diffusive limit of three-dimensional kinetic
Couette flow was studied in \cite{DuanLiuStrainYang24}, while the global
stability of the spatially inhomogeneous sheared Boltzmann equation on the
torus was established in \cite{DuanLiuShen25}.

The VPB system has been extensively studied in the absence of shear.
Global perturbative theories near Maxwellians were established in the
periodic and whole-space settings
\cite{Guo02,YangYuZhao06,YangZhao06}. For the one-species system with a
nonconstant background density, Duan and Yang constructed a nontrivial
stationary solution and proved its nonlinear stability
\cite{DuanYang10}, while optimal time-decay rates in the whole space were
obtained in \cite{DuanStrain11}. For the bipolar VPB system, spectral
properties and optimal decay rates were established in
\cite{LiYangZhong16}, and the stability of viscous shock profiles and
rarefaction waves was proved in \cite{LiWangYangZhong18}. The effect of
the Poisson coupling on the spectral structure and optimal decay of the
linearized VPB system was further analyzed in \cite{LiYangZhong21}.

Despite these developments, a rigorous stability theory for the VPB
system under uniform shear remains largely open. This is the problem
addressed in the present work.
\bigskip

\noindent{\bf Notation.}
Throughout this paper, $m=(m_1,m_2,m_3)$ and $n=(n_1,n_2,n_3)$ denote multi-indices in $\mathbb N^3$, and we set
\[
\partial_x^m
=\partial_{x_1}^{m_1}\partial_{x_2}^{m_2}\partial_{x_3}^{m_3},
\quad
\partial_v^n
=\partial_{v_1}^{n_1}\partial_{v_2}^{n_2}\partial_{v_3}^{n_3},
\quad
\partial_n^m:=\partial_x^m\partial_v^n.
\]
For any $h=h(x)$ on $\mathbb T^3$, we decompose $h$ into its zero-frequency and nonzero-frequency modes by setting
\[
P_0^x h(t):=\int_{\mathbb T^3}h(t,x)\,dx,\quad
P_{\neq0}^x h(t):=h(t)-P_0^x h(t).
\]

We denote by $\|\cdot\|_{L^\infty_{x,v}}$, $\|\cdot\|_{L^2_{x,v}}$,
$\|\cdot\|_{L^\infty_x}$, $\|\cdot\|_{L^2_x}$,
$\|\cdot\|_{L^\infty_v}$, and $\|\cdot\|_{L^2_v}$ the usual norms on
$L^\infty(\mathbb T_x^3\times\mathbb R_v^3)$,
$L^2(\mathbb T_x^3\times\mathbb R_v^3)$,
$L^\infty(\mathbb T_x^3)$, $L^2(\mathbb T_x^3)$,
$L^\infty(\mathbb R_v^3)$, and $L^2(\mathbb R_v^3)$, respectively. For inner products, the subscripts indicate the variables of integration:
\[
\langle f,g\rangle_{x,v}
=
\int_{\mathbb T^3}\int_{\mathbb R^3}
f(x,v)g(x,v)\,dv\,dx,
\]
\[
\langle f,g\rangle_x
=
\int_{\mathbb T^3}f(x)g(x)\,dx,
\quad
\langle f,g\rangle_v
=
\int_{\mathbb R^3}f(v)g(v)\,dv.
\]
For a matrix $A=(a_{ij})\in\mathbb R^{3\times3}$, we define $\|A\|_{\max}=\max_{1\le i,j\le3}|a_{ij}|$. For any $l\ge0$, define a weighted function as $$w_l(v):=(1+|v|^2)^l.$$

Throughout this paper, $C>0$ denotes a generic constant which may vary from line to line, and $\varepsilon>0$ denotes a sufficiently small constant whose value may be chosen differently at different occurrences. Constants with subscripts, such as $C_l$, may depend on the indicated parameters and may increase as these parameters increase. For multi-indices $m'\le m$ and $n'\le n$, we denote $C_{m,m'}=\prod_{i=1}^3\binom{m_i}{m_i'}$ and $C_{n,n'}=\prod_{i=1}^3\binom{n_i}{n_i'}$.
\bigskip

We now present our main results.

\begin{thm}\label{thm1.1}
Let $N\ge1$, let $G(v)$ be the self-similar profile, and let $\beta$ be given by \eqref{beta-def}. There exists a large constant $l_0>0$ such that for any $l\ge l_0$, we choose $\alpha_0=\alpha_0(N,l)>0$ sufficiently small. Then there exist constants $\varepsilon_0>0$, $\lambda>0$, and $C>0$, all independent of $\alpha_0$, such that for any $0<\alpha\le\alpha_0$, if the initial data $F_0(x,v)\ge0$ satisfies
\[
\sum_{|m|+|n|\le N}
\Big\|w_l\partial_n^m\big[F_0(x,v)-G(v)\big]\Big\|_{L^\infty_{x,v}}
\le\varepsilon_0,
\]
and
\begin{equation}\label{initial-con}
\int_{\mathbb T^3}\int_{\mathbb R^3}\big[F_0(x,v)-G(v)\big]\,dv\,dx=0,
\quad
\int_{\mathbb T^3}\int_{\mathbb R^3}v\big[F_0(x,v)-G(v)\big]\,dv\,dx=0,
\end{equation}
then the USF--VPB system \eqref{VPB-3D-Cauchy} admits a unique global solution $F(t,x,v)\ge0$ satisfying
\begin{equation}\label{thm1.1-est-1}
\Big\|w_l\big[e^{3\beta t}F(t,x,e^{\beta t}v)-G(v)\big]\Big\|_{L^\infty_{x,v}}
\le
C\sum_{|m|\le1}
\Big\|w_l\partial_x^m\big[F_0(x,v)-G(v)\big]\Big\|_{L^\infty_{x,v}},
\end{equation}
\begin{equation}\label{thm1.1-est-2}
\sum_{0<|n|\le N}
\Big\|w_l\partial_v^n\big[e^{3\beta t}F(t,x,e^{\beta t}v)-G(v)\big]\Big\|_{L^\infty_{x,v}}
\le
C\sum_{\substack{|m|+|n|\le N\\ |m|\le1}}
\Big\|w_l\partial_n^m\big[F_0(x,v)-G(v)\big]\Big\|_{L^\infty_{x,v}},
\end{equation}
and
\begin{equation}\label{thm1.1-est-3}
\sum_{\substack{|m|+|n|\le N\\ |m|>0}}
\Big\|w_l\partial_n^m\big[e^{3\beta t}F(t,x,e^{\beta t}v)\big]\Big\|_{L^\infty_{x,v}}
\le
Ce^{-[\lambda-(N-1)\beta]t}
\sum_{\substack{|m|+|n|\le N\\ |m|>0}}
\Big\|w_l\partial_n^mF_0(x,v)\Big\|_{L^\infty_{x,v}}.
\end{equation}
In particular, the self-consistent electric field satisfies
\begin{equation}\label{thm1.1-est-4}
\sum_{|m|\le N}
\Big\|\partial_x^m\nabla_x\phi(t)\Big\|_{L^\infty_x}
\le
Ce^{-\lambda t}
\sum_{0<|m|\le N}
\Big\|w_l\partial_x^m F_0(x,v)\Big\|_{L^\infty_{x,v}}.
\end{equation}
\end{thm}

\begin{thm}\label{thm1.2}
Under the conditions of Theorem~\ref{thm1.1}, define the renormalized total energy by
\[
\mathcal E_\alpha(t)
:=
e^{-2\beta t}\bigg(
\int_{\mathbb T^3}\int_{\mathbb R^3}|v|^2F(t,x,v)\,dv\,dx
+\|\nabla_x\phi(t)\|_{L_x^2}^2
\bigg)
-\int_{\mathbb R^3}|v|^2G(v)\,dv.
\]
Then there exist a limiting energy state $\mathcal E_\alpha(\infty)\in\mathbb R$ and a constant
\[
0<\tilde\lambda<\min\{2b_0+3\beta,\lambda\}
\]
such that
\[
|\mathcal E_\alpha(t)-\mathcal E_\alpha(\infty)|
\le
Ce^{-\tilde\lambda t},
\quad t\ge0.
\]
Equivalently,
\begin{equation}\label{energy-gro}
\int_{\mathbb T^3}\int_{\mathbb R^3}|v|^2F(t,x,v)\,dvdx
+\|\nabla_x\phi(t)\|_{L_x^2}^2
=
e^{2\beta t}\Big(
\int_{\mathbb R^3}|v|^2G(v)\,dv+\mathcal E_\alpha(\infty)
\Big)
+O(e^{(2\beta-\tilde\lambda)t}).
\end{equation}
Moreover, the limiting energy state admits the refined representation
\begin{equation}\label{Einf-refined}
\mathcal E_\alpha(\infty)
=
\frac{b_0+\beta}{b_0+3\beta}
\int_{\mathbb T^3}\int_{\mathbb R^3}
\big[F_0(x,v)-G(v)\big]
\bigg[
|v|^2
-\frac{\alpha}{b_0+\beta}v_1v_2
+\frac{\alpha^2}{2(b_0+\beta)^2}v_2^2
\bigg]\,dvdx
+O(\mathfrak I_0^2),
\end{equation}
where
\[
\mathfrak I_0:=
\sum_{|m|+|n|\le N}
\big\|w_l\partial_n^m\big[F_0(x,v)-G(v)\big]\big\|_{L^\infty_{x,v}}.
\]
\end{thm}

\begin{remark}
We record the following remarks on Theorems~\ref{thm1.1} and \ref{thm1.2}.

\noindent{\rm (i)}
Theorem~\ref{thm1.1} gives the global stability of the self-similar profile $G$ in a velocity-weighted $L^\infty$ framework. More precisely, the renormalized perturbation and its pure velocity derivatives remain uniformly bounded for all time, while all components involving at least one spatial derivative decay exponentially. Since $P_{\neq0}^xG=0$, Poincar\'e inequality on $\mathbb T^3$ implies
\[
\sum_{0\le |n|\le N-1}
\Big\|w_l\partial_v^nP_{\neq0}^x\big[e^{3\beta t}F(t,x,e^{\beta t}v)\big]\Big\|_{L^\infty_{x,v}}
\le
Ce^{-[\lambda-(N-1)\beta]t}
\sum_{\substack{|m|+|n|\le N\\ |m|>0}}
\Big\|w_l\partial_x^m\partial_v^nF_0(x,v)\Big\|_{L^\infty_{x,v}}.
\]
Thus the nonzero-frequency modes decay exponentially in the self-similar variables. Moreover, by \eqref{beta},
we can choose $\alpha_0>0$ sufficiently small to  ensure that $\lambda-(N-1)\beta>0$.

\noindent{\rm (ii)}
Theorem~\ref{thm1.2} identifies the zero-frequency energy dynamics. The renormalized total energy $\mathcal E_\alpha(t)$, which contains both the kinetic energy and the electric field energy, converges exponentially to a limiting state $\mathcal E_\alpha(\infty)$. Equivalently, the total energy grows at the precise rate $e^{2\beta t}$, with prefactor
\[
\int_{\mathbb R^3}|v|^2G(v)\,dv+\mathcal E_\alpha(\infty).
\]
The refined representation \eqref{Einf-refined} further shows that, up to the quadratic remainder $O(\mathfrak I_0^2)$, the limiting energy state is determined by a shear-corrected second-order moment of the initial perturbation. This is consistent with the moment identity \eqref{energy-usf}: the evolution of the energy moment is driven by the stress moment $\displaystyle \int v_1v_2F\,dv$. Consequently, the limiting energy state $\mathcal E_\alpha(\infty)$ is not determined by the initial energy moment alone. Using the expansion above, the shear-corrected weight in \eqref{Einf-refined} satisfies
\[
\frac{b_0+\beta}{b_0+3\beta}
\Big[
|v|^2
-\frac{\alpha}{b_0+\beta}v_1v_2
+\frac{\alpha^2}{2(b_0+\beta)^2}v_2^2
\Big]
=
\Big(1-\frac{\alpha^2}{3b_0^2}\Big)|v|^2
-\frac{\alpha}{b_0}v_1v_2
+\frac{\alpha^2}{2b_0^2}v_2^2
+O(\alpha^3)|v|^2.
\]
Thus the $v_1v_2$-moment gives the leading $O(\alpha)$ shear correction, while the $v_2^2$-moment and the correction to the $|v|^2$-weight give $O(\alpha^2)$ contributions.

\noindent{\rm (iii)}
To the best of our knowledge, Theorems~\ref{thm1.1} and \ref{thm1.2} provide the first global-in-time stability and large-time asymptotic results for the VPB system under uniform shear flow around a nonequilibrium self-similar profile in a velocity-weighted $L^\infty$ framework. If the Poisson coupling is removed, the system formally reduces to the inhomogeneous sheared Boltzmann equation, for which related global stability and asymptotic results were recently obtained in \cite{DuanLiuShen25}. When $\alpha=0$, the shear vanishes, $\beta=0$, and $G$ reduces to the global Maxwellian equilibrium, so that the model becomes the classical VPB equation near equilibrium. In this limiting case, the exponential growth factor in Theorem~\ref{thm1.2} becomes trivial, and \eqref{Einf-refined} reduces, at the leading perturbative order, to the initial kinetic energy perturbation. Since the initial field energy is of quadratic order and is included in the remainder, this is consistent with the kinetic-field energy structure of the classical VPB dynamics.
\end{remark}

We now turn to the main ideas of the proof. The analytical framework of this
paper builds on the self-similar theory for the USF--Boltzmann equation
developed in \cite{DuanLiu21,JamesNotaVelazquez19a}. It is also inspired by a
zero-frequency mode analysis method in \cite{DuanLiuShen25}. USF generates
shear-induced heating, and hence the velocity scale of the solution changes in
time. We use the self-similar scaling \eqref{scaled-f}, under which the leading
background becomes the stationary nonequilibrium profile $G$, defined by
\eqref{self-similar-G}--\eqref{G-normal}. In this way, the VPB system is
reformulated as a perturbation problem around $G$, as shown in
\eqref{tildeg-eq}. To handle the high-velocity growth caused by the
self-similar scaling and the shear term, we use Caflisch's decomposition
\[
\mu^{\frac12}\tilde g=g_1+\mu^{\frac12}g_2,
\]
following \cite{Caflisch80a,DuanLiu21}. Since Caflisch's decomposition
weakens the direct $L^2$ control, we further combine it with Guo's
$L^\infty$--$L^2$ method \cite{Guo10}. A key ingredient in this step is the characteristic change-of-variables estimate in Lemma~\ref{lem-cov-3d}. The weighted $L^\infty$ estimates are then closed with the help of the refined $L^2$ estimates developed in Section~\ref{sec4}.

We next explain the main idea underlying the weighted $L^\infty$ estimates. The exponential decay of nonzero-order spatial derivatives is a principal objective of this paper and is also essential for closing these estimates, since it compensates for the growth in time of the transport factor. The main difficulty comes from the time-dependent transport terms
\[
e^{\beta t}v\cdot\nabla_x g_1,
\quad
e^{\beta t}v\cdot\nabla_x g_2.
\]
Although these terms are incorporated into the characteristics, velocity derivatives do not commute with $v\cdot\nabla_x$. After applying mixed derivatives to the equations, this produces commutator source terms which enter the Duhamel representations \eqref{I1-decomp}--\eqref{I2-decomp} in the form
\[
e^{\beta s}w_l\partial_{n-e_i}^{m+e_i}g_1,
\quad
e^{\beta s}w_l\partial_{n-e_i}^{m+e_i}g_2,
\]
where one velocity derivative is converted into one spatial derivative. This motivates the introduction of the key layered time weights $e^{\lambda_n^mt}$ in the $L^\infty$ closure, where $\lambda_n^m$ is defined in \eqref{def-lambda-mn} by
\[
\lambda_n^m=
\begin{cases}
0,& |m|=0,\\
\lambda-|n|\beta,& |m|>0.
\end{cases}
\]
For $|m|>0$ and $n_i>0$, the time weights are exactly matched with the derivative shift generated by the commutator, since $\lambda_{n-e_i}^{m+e_i}=\lambda_n^m+\beta$. Thus, the additional factor $e^{\beta s}$ is absorbed by the time weight associated with the mixed derivative $\partial_{n-e_i}^{m+e_i}g_j$, $j=1,2$, which is precisely the mechanism underlying the estimates \eqref{weighted-I61}--\eqref{weighted-I42}. The case $|m|=0$ is left unweighted in accordance with the $L^2$ theory below: the zero-order macroscopic temperature mode is only uniformly controlled in time, whereas exponential decay is available only in the nonzero-order spatial derivative hierarchy. When $|m|=0$, the commutator produces a nonzero-order spatial derivative, so the resulting term is controlled in the $|m|>0$ hierarchy. Since $\beta=O(\alpha^2)$, taking the shear rate sufficiently small keeps $\lambda_n^m\ge0$ for all $|m|+|n|\le N$. This layered choice of time weights plays a key role in closing the weighted $L^\infty$ estimates for mixed $x$-$v$ derivatives, and the decay-rate loss $(N-1)\beta$ in \eqref{thm1.1-est-3} of Theorem~\ref{thm1.1} reflects this mechanism.

In the $L^2$ analysis, the main difficulty arises from the temperature mode
$c$, defined later in \eqref{def-abc}. Under the USF background, the
shear-induced heating destroys the usual energy conservation, so that one
cannot naturally impose $P_0^x c=0$. Although the $c$-equation
\eqref{eq-macro-abc} contains the damping term $2\beta c$, this damping is
too weak because $\beta=O(\alpha^2)$. In particular, it does not yield
sufficiently strong zero-order dissipation or exponential decay for $c$.
Consequently, the zero-order estimates must be treated separately from the
decay estimates for nonzero-order spatial derivatives. Otherwise, zero-order
macroscopic or microscopic quantities without the required decay would enter
the nonzero-order spatial derivative estimates and prevent their closure.

A key step in deriving the nonzero-order spatial derivative estimates is to test the
$c$-equation in the macroscopic system \eqref{eq-macro-abc} against
$P_{\neq0}^x\partial_x^m c$. This is needed only in the borderline case
 $m=0$, in order to remove the zero-frequency mode. When $|m| \ge1$, one has
\[
P_{\neq0}^x\partial_x^m c=\partial_x^m c .
\]
The projection also preserves the heat dissipation:
\[
-\big\langle \Delta_x\partial_x^m c,
P_{\neq0}^x\partial_x^m c\big\rangle_x
=
\|\nabla_x\partial_x^m c\|_{L_x^2}^2 .
\]
In the case $m=0$, the resulting lower-order quantities are replaced by their
nonzero-frequency modes, so Poincar\'e inequality gives
\[
\|P_{\neq0}^x d_{12}\|_{L_x^2}
\le
C\|\nabla_x d_{12}\|_{L_x^2}
\le
C\|\nabla_x\mathbf P_1\tilde g\|_{L^2_{x,v}},
\quad
\|P_{\neq0}^x c\|_{L_x^2}
\le
C\|\nabla_x c\|_{L_x^2}.
\]
We refer to \eqref{eq-Pneq0-c}--\eqref{c-step} and \eqref{c-dis} for the implementation of this idea.

More precisely, the principal obstruction in the zero-order estimates is the
zero-frequency temperature mode $P_0^x c$. Inspired by \cite{DuanLiuShen25}, we introduce a renormalized total energy
\[
\mathcal E_\alpha(t)=\sqrt6\,P_0^x c(t)+e^{-2\beta t}\|\nabla_x\phi(t)\|_{L_x^2}^2,
\]
and couple it with suitable second-order moments. This gives a three-component zero-frequency ODE system
\[
\frac{d}{dt}U(t)+A_\alpha U(t)=R(t),
\quad
U(t)=
\begin{pmatrix}
\mathcal E_\alpha(t)\\
P_0^x d_{12}(t)\\
P_0^x d_{22}(t)
\end{pmatrix},
\]
with $A_\alpha$ and $R(t)$ specified in \eqref{def-U-Aa}--\eqref{def-Rt}. A spectral analysis of this system yields uniform-in-time control of
$P_0^x c$, which in turn provides the required zero-order control of $c$;
see Lemma~\ref{lem-Uc-L2}. This supplies the missing zero-order estimate
needed to close the weighted $L^\infty$ argument. A further analysis of the
same ODE system yields Theorem~\ref{thm1.2}, which characterizes the
shear-induced growth of the total energy through \eqref{energy-gro} and
provides the refined expansion of the limiting renormalized energy in
\eqref{Einf-refined}.

The rest of the paper is organized as follows. In Section~\ref{sec2}, we
derive the perturbation equations in the self-similar and Caflisch-decomposed
variables, introduce the relevant operators, and collect several basic
estimates. In Section~\ref{Sec3}, we establish the weighted $L^\infty$
estimates under a suitable time-weighted a priori assumption and identify the
$L^2$ controls required for their closure. In Section~\ref{sec4}, we derive
the $L^2$ energy estimates needed to close the a priori argument. The
nonzero-order spatial derivative estimates and the zero-order estimates are treated
separately, with the latter relying on the spectral analysis of the
zero-frequency system. In Section~\ref{Sec5}, we combine the local
well-posedness theory with the a priori estimates established in the previous
sections to prove global existence, uniqueness, nonnegativity, and large-time
behavior, including the shear-induced energy growth. Appendices~\ref{app-B} and~\ref{app-A} contain the proof of the local existence result and the proof of a moment identity
for $L$, respectively.

\section{Preliminaries}\label{sec2}

We begin with the self-similar scaling associated with the homogeneous USF profile. This scaling is chosen to match the time-dependent velocity scale generated by the shear-induced heating. More precisely, we define
\begin{equation}\label{scaled-f}
f(t,x,v)=e^{3\beta t}F(t,x,e^{\beta t}v),
\quad
F(t,x,v)=e^{-3\beta t}f(t,x,e^{-\beta t}v).
\end{equation}
Substituting \eqref{scaled-f} into \eqref{VPB-3D-Cauchy}, we obtain the scaled USF--VPB system
\begin{equation}\label{scaled-USF-VPB}
\left\{
\begin{aligned}
&\partial_t f + e^{\beta t}v\cdot\nabla_x f
- \beta\nabla_v\cdot(vf)
- \alpha v_2\partial_{v_1}f
+ e^{-\beta t}\nabla_x\phi\cdot\nabla_v f
= Q(f,f),\\
& \Delta_x\phi
= \int_{\mathbb R^3}f(t,x,v)\,dv - 1,\quad\int_{\mathbb T^3}\phi\,dx=0,\\
&f(0,x,v)=F_0(x,v).
\end{aligned}
\right.
\end{equation}
Define the perturbation function $\tilde f$ by
\begin{equation}\label{def-tildef}
f(t,x,v)=G(v)+\tilde f(t,x,v),
\end{equation}
where $G$ is the homogeneous self-similar profile satisfying \eqref{self-similar-G}--\eqref{G-normal}. Since $\displaystyle \int_{\mathbb R^3}G(v)\,dv=1$, it follows that $\tilde f$ satisfies
\begin{equation}\label{perturbed-USF-VPB}
\left\{
\begin{aligned}
&\partial_t\tilde f
+ e^{\beta t}v\cdot\nabla_x\tilde f
- \beta\nabla_v\cdot(v\tilde f)
- \alpha v_2\partial_{v_1}\tilde f
+ e^{-\beta t}\nabla_x\phi\cdot\nabla_v\tilde f\\
& \quad
= Q(G,\tilde f)+Q(\tilde f,G)+Q(\tilde f,\tilde f)
- e^{-\beta t}\nabla_x\phi\cdot\nabla_v G,\\
& \Delta_x\phi
= \int_{\mathbb R^3}\tilde f(t,x,v)\,dv,\quad\int_{\mathbb T^3}\phi\,dx=0,\\
& \tilde f(0,x,v)=F_0(x,v)-G(v).
\end{aligned}
\right.
\end{equation}

According to \eqref{def-G}, we write
\begin{equation}\label{def-G1}
G(v):=\mu(v)+\alpha\mu^{\frac12}(v)G_1(v),\quad
\mu^{\frac12}(v)G_1(v)\sim -\frac{1}{2b_0}v_1v_2\mu(v).
\end{equation}
Furthermore, letting $\tilde f=\mu^{\frac12}\tilde g$ and substituting it into \eqref{perturbed-USF-VPB}, we obtain
\begin{equation}\label{tildeg-eq}
\left\{
\begin{aligned}
& \partial_t\tilde g
+ e^{\beta t}v\cdot\nabla_x\tilde g
- \beta\nabla_v\cdot(v\tilde g)
+ \frac{\beta}{2}|v|^2\tilde g
- \alpha v_2\partial_{v_1}\tilde g
+ \frac{\alpha}{2}v_1v_2\tilde g \\
& \quad+ e^{-\beta t}\nabla_x\phi\cdot\nabla_v\tilde g
- \frac12 e^{-\beta t}(v\cdot\nabla_x\phi)\tilde g
+ L\tilde g \\
& \quad
= \alpha\Gamma(G_1,\tilde g)
+ \alpha\Gamma(\tilde g,G_1)
+ \Gamma(\tilde g,\tilde g)
+ e^{-\beta t}\mu^{\frac12}v\cdot\nabla_x\phi
- \alpha e^{-\beta t}\nabla_x\phi\cdot\big(\nabla_v G_1-\frac{v}{2}G_1\big), \\
& \Delta_x\phi=\int_{\mathbb R^3}\mu^{\frac12}(v)\tilde g(t,x,v)\, dv,\quad\int_{\mathbb T^3}\phi\,dx=0, \\
& \tilde g(0,x,v)=\mu^{-\frac12}(v)\big[F_0(x,v)-G(v)\big].
\end{aligned}
\right.
\end{equation}
Here the bilinear operator $\Gamma$ and the linearized collision operator $L$ are defined by
\[
\left\{
\begin{aligned}
&\Gamma(f,g)(v)=
\mu^{-\frac12}(v)
Q\big(\mu^{\frac12}f,\mu^{\frac12}g\big)(v),\\
&Lf(v)=
-\mu^{-\frac12}(v)
\Big(
Q\big(\mu,\mu^{\frac12}f\big)(v)
+
Q\big(\mu^{\frac12}f,\mu\big)(v)
\Big)
=
\nu_0 f(v)-Kf(v),\\
&\nu_0=
\int_{\mathbb R^3}\int_{\mathbb S^2}
B_0(\cos\theta)\mu(v_*)\,d\omega\,dv_*,\\
&Kf(v)=
\mu^{-\frac12}(v)
\Big(
Q\big(\mu,\mu^{\frac12}f\big)(v)
+
Q_+\big(\mu^{\frac12}f,\mu\big)(v)
\Big).
\end{aligned}
\right.
\]

To control the rapidly growing high-velocity terms
$\displaystyle \frac{\beta}{2}|v|^2\tilde g+\frac{\alpha}{2}v_1v_2\tilde g$
in \eqref{tildeg-eq}, we employ Caflisch's decomposition
\begin{equation}\label{Caflisch-decomp}
\mu^{\frac12}\tilde g=g_1+\mu^{\frac12}g_2,
\end{equation}
and substitute it into \eqref{tildeg-eq}. We then take $[g_1,g_2,\phi]$ to
satisfy the following coupled system:
\begin{equation}\label{g1-eq}
\left\{
\begin{aligned}
& \partial_t g_1
+e^{\beta t}v\cdot\nabla_x g_1
-\beta\nabla_v\cdot(vg_1)
-\alpha v_2\partial_{v_1}g_1
+e^{-\beta t}\nabla_x\phi\cdot\nabla_v g_1
+\nu_0 g_1\\
& \quad
= \chi_{M} \mathcal K g_1
-\frac{\beta}{2}|v|^2\mu^{\frac12}g_2
-\frac{\alpha}{2}\mu^{\frac12}v_1v_2 g_2
+\frac{e^{-\beta t}}{2}(v\cdot\nabla_x\phi)\mu^{\frac12}g_2 \\
& \quad\quad
+\widetilde H(g_1,g_2)
-\alpha e^{-\beta t}\nabla_x\phi\cdot\nabla_v(\mu^{\frac12}G_1), \\
& g_1(0,x,v)=F_0(x,v)-G(v)=\tilde f_0(x,v),
\end{aligned}
\right.
\end{equation}
\begin{equation}\label{g2-eq}
\left\{
\begin{aligned}
& \partial_t g_2
+e^{\beta t}v\cdot\nabla_x g_2
-\beta\nabla_v\cdot(vg_2)
-\alpha v_2\partial_{v_1}g_2
+e^{-\beta t}\nabla_x\phi\cdot\nabla_v g_2
+Lg_2 \\
& \quad
=\mu^{-\frac12}(1-\chi_M)\mathcal K g_1
+e^{-\beta t}\mu^{\frac12}v\cdot\nabla_x\phi, \\
& g_2(0,x,v)=0,
\end{aligned}
\right.
\end{equation}
and
\begin{equation}\label{Poisson-g1g2}
\Delta_x\phi(t,x)
=\int_{\mathbb R^3}\big(g_1+\mu^{\frac12}g_2\big)(t,x,v)\,dv,
\quad
\int_{\mathbb T^3}\phi\,dx=0.
\end{equation}
Here $\mathcal K$ is defined by
\[
\mathcal K f(v)
:=
\mu^{\frac12}(v)K\big(\mu^{-\frac12}f\big)(v).
\]
The nonlinear term $\widetilde H(g_1,g_2)$ is given by
\begin{align}
\widetilde H(g_1,g_2)
&:=Q(g_1,g_1)
+Q(g_1,\mu^{\frac12}g_2)
+Q(\mu^{\frac12}g_2,g_1)
+Q(\mu^{\frac12}g_2,\mu^{\frac12}g_2) \nonumber\\
&\quad+\alpha Q(\mu^{\frac12}G_1,g_1)
+\alpha Q(\mu^{\frac12}G_1,\mu^{\frac12}g_2)
+\alpha Q(g_1,\mu^{\frac12}G_1)
+\alpha Q(\mu^{\frac12}g_2,\mu^{\frac12}G_1).
\label{H-tilde-def}
\end{align}
The smooth cutoff function $\chi_M=\chi_M(v)$ is defined by
\[
\chi_M(v)=
\left\{
\begin{aligned}
&1,\quad |v|\ge M+1,\\
&0,\quad |v|\le M,
\end{aligned}
\right.
\]
where $M=M(l)\gg1$ will be chosen later according to \eqref{weighted-I11}.

For later use, we also introduce the projection associated with $L$. The null
space of the operator $L$, denoted by $\ker L$, is spanned by the orthonormal basis $\{\chi_j, \ 0\le j\le4\}$ given by
\begin{equation}\label{def-chi}
\chi_0=\mu^{\frac12},\quad
\chi_j=v_j\mu^{\frac12}\ \ (1\le j\le3),\quad
\chi_4=\frac{(|v|^2-3)\mu^{\frac12}}{\sqrt6}.
\end{equation}
Let $\mathbf P_0$ be the projection operator from $L^2(\mathbb R_v^3)$ onto
$\ker L$, namely
\[
\mathbf P_0 h=\sum_{j=0}^4\langle h,\chi_j\rangle_v\chi_j,
\quad
\mathbf P_1=I-\mathbf P_0.
\]

We next collect several estimates for the collision operators which will be used repeatedly in the sequel.

\begin{lemma}[{\cite[Lemmas 3.2 and 3.3]{Guo06}}]\label{lem-est-L}
In the case of Maxwell molecules, the linearized collision operator $L$ is
nonnegative and has a spectral gap on the microscopic space. More precisely,
there exists a constant $\delta_0>0$ such that, for each fixed $x$,
\begin{equation}\label{coercivity-L}
\langle Lh,h\rangle_v
=
\langle L\mathbf P_1 h, \mathbf P_1 h\rangle_v
\ge \delta_0\|\mathbf P_1 h\|^2_{L^2_v}.
\end{equation}
Moreover, for any nonzero-order multi-index $n$ and any $l\ge0$, we have
\begin{equation}\label{wl-dv-coercivity-L}
\big\langle w_l^{2}\partial_v^{n}L h,\partial_v^{n}h\big\rangle_v
\ge
\delta_0\big\|w_l\partial_v^{n} h\big\|^2_{L^2_v}
-C\|h\|^2_{L^2_v}.
\end{equation}
\end{lemma}

\begin{lemma}\label{lem-Gamma}
In the case of Maxwell molecules, for any multi-index $n$ and any $l\ge0$, there exists a constant $C>0$ such that
\begin{equation}\label{eq-Gamma-L2}
\|w_l\partial_v^n\Gamma(F_1,F_2)\|_{L^2_v}
\le
C\sum_{n'\le n}
\|w_l\partial_v^{n'}F_1\|_{L^2_v}
\|w_l\partial_v^{n-n'}F_2\|_{L^2_v},
\end{equation}
and
\begin{equation}\label{eq-Gamma-Linf}
\|w_l\partial_v^n\Gamma(F_1,F_2)\|_{L^\infty_v}
\le
C\sum_{n'\le n}
\|w_l\partial_v^{n'}F_1\|_{L^\infty_v}
\|w_l\partial_v^{n-n'}F_2\|_{L^\infty_v}.
\end{equation}
\end{lemma}
\begin{proof}
The proof of \eqref{eq-Gamma-L2} follows the same argument as in \cite[Lemma 2.3]{Guo02}, while the proof of \eqref{eq-Gamma-Linf} is analogous to \cite[Lemma 5]{Guo10}. For brevity, we omit the details.
\end{proof}

\begin{lemma}[{\cite[Proposition 3.1]{Arkeryd87}}]\label{lem-Q}
In the case of Maxwell molecules, if $\displaystyle l>\frac32$, then for any multi-index $n$, there exists a constant $C>0$ such that
\begin{equation}\label{eq-Q-Linf}
\|w_l\partial_v^n Q(F_1,F_2)\|_{L^\infty_v}
\le
C\sum_{n'\le n}
\|w_l\partial_v^{n'}F_1\|_{L^\infty_v}
\|w_l\partial_v^{n-n'}F_2\|_{L^\infty_v}.
\end{equation}
\end{lemma}

\begin{lemma}[{\cite[Lemmas 2.1 and 2.2]{DuanYangZhao12}}]\label{lem-K}
In the case of Maxwell molecules, let $K_q(v,v_*)$ be defined by
\begin{equation}\label{def-Kq}
K_q(v,v_*)
:=
\big(|v-v_*|+|v-v_*|^{-1}\big)
\exp\!\bigg[
-\frac{q}{8}|v-v_*|^2
-\frac{q}{8}\frac{(|v|^2-|v_*|^2)^2}{|v-v_*|^2}
\bigg],
\quad q\in(0,1).
\end{equation}
Then the following estimates hold.

\begin{enumerate}
\item[\rm(i)] Let $n$ be any multi-index. For any $q\in(0,1)$, there exists a constant $C_{|n|,q}>0$ such that, for all $v\in\mathbb R^3$,
\begin{equation}\label{est-Kq-C1}
\big|\partial_v^n(Kh)(v)\big|
\le
C_{|n|,q}
\int_{\mathbb R^3}K_q(v,v_*)
\sum_{|n'|\le |n|}
\big|\partial_v^{n'}h(v_*)\big|\,dv_* .
\end{equation}

\item[\rm(ii)] For any $\varepsilon>0$, there exists a constant $C_{|n|,\varepsilon}>0$ such that
\begin{equation}\label{est-K-L2-C}
\big\|\partial_v^n(Kh)\big\|_{L^2_v}
\le
\varepsilon\sum_{|n'|=|n|}\big\|\partial_v^{n'}h\big\|_{L^2_v}
+
C_{|n|,\varepsilon}\|h\|_{L^2_v}.
\end{equation}

\item[\rm(iii)] There exists a constant $C_{q,l}>0$ such that, for any $v\in\mathbb R^3$,
\begin{equation}\label{est-Kq-C2}
\int_{\mathbb R^3}K_q(v,v_*)\frac{w_l(v)}{w_l(v_*)}\,dv_*
\le
\frac{C_{q,l}}{1+|v|}.
\end{equation}
\end{enumerate}
\end{lemma}

\begin{lemma}[{\cite[Proposition 2.1]{DuanLiu21}}]\label{lem-DL-K}
In the case of Maxwell molecules, for any integer $k\ge0$, there exists a constant $C=C(k)>0$ such that, for any sufficiently large weight index $l>0$, one can choose a cutoff radius
\[
M=M(l)\quad(\text{for instance }M=l^2)
\]
so that the following smallness estimate in the high-velocity region holds:
\begin{equation}
\sup_{|v|\ge M}w_l(v)\big|\nabla_v^k(\mathcal K h)(v)\big|
\le
\frac{C}{l}
\sum_{k'=0}^{k}
\big\|w_l\nabla_v^{k'}h\big\|_{L^\infty_v}.
\end{equation}
\end{lemma}

\begin{lemma}\label{lem-Lmoments}
For Maxwell molecules, the linearized operator $L$ satisfies, for any $1\le i,j\le3$,
\begin{equation}\label{L-moments}
\left\{
\begin{aligned}
L (v_iv_j\mu^{\frac12})
&=
2b_0\Big(v_iv_j-\frac{\delta_{ij}}{3}|v|^2\Big)\mu^{\frac12},\\
L (|v|^2v_i\mu^{\frac12})
&=
\frac43b_0(|v|^2-5)v_i\mu^{\frac12}.
\end{aligned}
\right.
\end{equation}
\end{lemma}

\begin{proof}
The first identity in \eqref{L-moments} is exactly \cite[Lemma 6.6]{DuanLiu21}. For brevity, we postpone the proof of the second identity to Appendix~\ref{app-A}.
\end{proof}

We also recall the existence and expansion of the self-similar profile $G$.

\begin{lemma}[{\cite[Theorem 1.1]{DuanLiu21}}]\label{lem-DL-G}
In the case of Maxwell molecules, let $b_0>0$ be the constant defined by
\eqref{def-b0}. Then there exists a constant $l_0>0$ such that for any
$l\ge l_0$, there exists $\alpha_0=\alpha_0(l)>0$ with the following
property: for all $\alpha\in(0,\alpha_0)$, the steady Boltzmann equation
\eqref{self-similar-G} admits a unique smooth solution
$G=G(v)\in C^\infty(\mathbb R^3)$ satisfying
\[
\int_{\mathbb R^3}[1,v,|v|^2]G(v)\,dv=[1,0,3].
\]
Moreover, for any integer $k\ge0$, there exists a constant
$C_{k,l}>0$, independent of $\alpha$, such that
\begin{equation}\label{DL-G}
\Big\|w_l(v)\nabla_v^k
\Big[
G-\Big(1-\frac{\alpha}{2b_0}v_1v_2\Big)\mu
\Big]\Big\|_{L^\infty_v}
\le C_{k,l}\alpha^2.
\end{equation}
In particular, if $G_1$ is defined by \eqref{def-G1}, then
\begin{equation}\label{G1-bound}
\Big\|
w_l\nabla_v^k\big(\mu^{\frac12}G_1\big)
\Big\|_{L^\infty_v}
\le C_{k,l}.
\end{equation}
\end{lemma}

We shall also use the following standard elliptic estimate on the torus.

\begin{lemma}[Elliptic estimate on $\mathbb T^3$, {\cite[Lemma~3.2]{ModenaSattig20}}]\label{lem-elliptic-T3}
Let $1<p<\infty$ and let $k\ge0$ be an integer. Suppose that
$\phi$ is periodic on $\mathbb T^3$, satisfies
\[
\int_{\mathbb T^3}\phi(x)\,dx=0,\quad
\rho:=\Delta_x\phi\in W^{k,p}(\mathbb T^3).
\]
Then, there exists a constant $C_{k,p}>0$, depending only on $k$ and
$p$, such that
\begin{equation}\label{elliptic-Wkp}
\|\phi\|_{W^{k+2,p}(\mathbb T^3)}
\le
C_{k,p}\|\rho\|_{W^{k,p}(\mathbb T^3)}.
\end{equation}
\end{lemma}

\section{\texorpdfstring{$L^\infty$}{L-infinity} Estimates}\label{Sec3}

In this section, we establish the $L^\infty$ estimates. We first define the Banach space
\[
\widetilde{\mathcal Y}_T
:=
\Big\{
[\mathcal{G}_1,\mathcal{G}_2]\in L^\infty(0,T;W^{N,\infty}_{x,v})\, \Big| \,
\|[\mathcal{G}_1,\mathcal{G}_2]\|_{\widetilde{\mathcal Y}_T}<\infty,\
\big\langle \mathcal G_1,[1,v_i]\big\rangle_{x,v}+\big\langle \mathcal G_2,[1,v_i]\mu^{\frac 12}\big\rangle_{x,v}=0,\ i=1,2,3
\Big\},
\]
where
\[
\big\|[\mathcal G_1,\mathcal G_2]\big\|_{\widetilde{\mathcal Y}_T}
:=
\sum_{|m|+|n|\le N}\sup_{0\le t\le T}
\left(
\big\|w_l\partial_n^m \mathcal G_1(t)\big\|_{L^\infty_{x,v}}
+
\big\|w_l\partial_n^m \mathcal G_2(t)\big\|_{L^\infty_{x,v}}
\right).
\]
Let $T\in(0,\infty]$ be the maximal existence time of the local solution constructed in Lemma \ref{lem-local-existence}. By the standard continuation principle, it is enough to establish a uniform bound on $\|[g_1,g_2]\|_{\widetilde{\mathcal Y}_T}$ in order to extend the solution globally in time. To this end, we introduce time weights, which will also lead to time-decay estimates. We impose the following a priori assumption: for some constant $0<\eta\ll1$,
\begin{equation}\label{assa}
\sum _{|m|+|n|\le N}\sup_{0\le t\le T}
e^{\lambda_n^m t}
\Big(
\|w_l\partial_n^m g_1(t)\|_{L^\infty_{x,v}}
+
\|w_l\partial_n^m g_2(t)\|_{L^\infty_{x,v}}
\Big)
\le \eta.
\end{equation}
Here
\begin{equation}\label{def-lambda-mn}
\lambda_n^m=
\left\{
\begin{aligned}
&0,\quad\quad |m|=0,\\
&\lambda-|n|\beta,\quad |m|>0,
\end{aligned}
\right.
\end{equation}
where
\begin{equation}\label{choice-lambda}
0<\lambda<\frac12\min\{\nu_0,\tilde\lambda_0\},
\end{equation}
and $\tilde\lambda_0>0$ is a constant specified in
\eqref{EN1g2-absorb}. Since $\beta=O(\alpha^2)$, we take
$\alpha>0$ sufficiently small so that
\[
\lambda_n^m\ge0,\quad |m|+|n|\le N.
\]

We now derive the weighted differentiated system. For any multi-indices $m,n\in\mathbb N^3$ satisfying $|m|+|n|\le N$, set
\begin{equation}\label{eq-def-h12mn}
h_{1}^{m,n}:=w_l\partial_n^m g_1,
\quad
h_{2}^{m,n}:=w_l\partial_n^m g_2.
\end{equation}
Applying $w_l\partial_n^m$ to \eqref{g1-eq} and \eqref{g2-eq}, we obtain the coupled system for $[h_{1}^{m,n},h_{2}^{m,n}]$:
\begin{equation}\label{hmn-g1}
\left\{
\begin{aligned}
& \partial_t h_{1}^{m,n}
+e^{\beta t}v\cdot\nabla_x h_{1}^{m,n}
-\beta v\cdot\nabla_v h_{1}^{m,n}
-\alpha v_2\partial_{v_1}h_{1}^{m,n}
+e^{-\beta t}\nabla_x\phi\cdot\nabla_v h_{1}^{m,n}
+\nu_0 h_{1}^{m,n}
\\
& \quad
-\bigg[
\beta (|n|+3)
-2l\beta\frac{|v|^2}{1+|v|^2}
-2l\alpha\frac{v_1v_2}{1+|v|^2}
+2l e^{-\beta t}\frac{v\cdot\nabla_x\phi}{1+|v|^2}
\bigg]h_{1}^{m,n}
=\mathcal R_{1}^{m,n},
\\
& h_{1}^{m,n}(0,x,v)=w_l\partial_x^m\partial_v^n\tilde f_0(x,v),
\end{aligned}
\right.
\end{equation}
and
\begin{equation}\label{hmn-g2}
\left\{
\begin{aligned}
& \partial_t h_{2}^{m,n}
+e^{\beta t}v\cdot\nabla_x h_{2}^{m,n}
-\beta v\cdot\nabla_v h_{2}^{m,n}
-\alpha v_2\partial_{v_1}h_{2}^{m,n}
+e^{-\beta t}\nabla_x\phi\cdot\nabla_v h_{2}^{m,n}
+\nu_0 h_{2}^{m,n}
\\
& \quad
-\bigg[
\beta (|n|+3)
-2l\beta\frac{|v|^2}{1+|v|^2}
-2l\alpha\frac{v_1v_2}{1+|v|^2}
+2l e^{-\beta t}\frac{v\cdot\nabla_x\phi}{1+|v|^2}
\bigg]h_{2}^{m,n}
=\mathcal R_{2}^{m,n},
\\
& h_{2}^{m,n}(0,x,v)=0,
\end{aligned}
\right.
\end{equation}
where $\mathcal R_{1}^{m,n}$ and $\mathcal R_{2}^{m,n}$ are given by
\begin{equation}\label{def-Rmn12}
\left\{
\begin{aligned}
\mathcal R_{1}^{m,n}
:={}&
w_l\partial_n^m(\chi_M\mathcal K g_1)
-\frac{\beta}{2}w_l\partial_n^m(|v|^2\mu^{\frac12}g_2)
-\frac{\alpha}{2}w_l\partial_n^m(\mu^{\frac12}v_1v_2 g_2)
\\
&+
\frac{e^{-\beta t}}{2}
\sum_{0\le m'\le m} C_{m,m'}
(\partial_x^{m'}\nabla_x\phi)\cdot
w_l\partial_n^{m-m'}(\mu^{\frac12}v g_2)
+w_l\partial_n^m\widetilde H(g_1,g_2)
\\
&-
\alpha e^{-\beta t}w_l\partial_n^m\big[\nabla_x\phi\cdot\nabla_v(\mu^{\frac12}G_1)\big]
-e^{\beta t}\sum_{i=1}^{3}\mathbf{1}_{\{n_i>0\}}n_i
w_l\partial_{n-e_i}^{m+e_i}g_1
\\
&-
\mathbf{1}_{\{n_2>0\}}\alpha n_2
w_l\partial_{n-e_2+e_1}^{m}g_1
-e^{-\beta t}
\sum_{0<m'\le m} C_{m,m'}
(\partial_x^{m'}\nabla_x\phi)\cdot
w_l\nabla_v\partial_n^{m-m'}g_1,
\\[0.4em]
\mathcal R_{2}^{m,n}
:={}&
w_l\partial_n^m(Kg_2)
+w_l\partial_n^m\big[\mu^{-\frac12}(1-\chi_M)\mathcal K g_1\big]
+e^{-\beta t}w_l\partial_n^m(\mu^{\frac12}v\cdot\nabla_x\phi)
\\
&-
e^{\beta t}\sum_{i=1}^{3}\mathbf{1}_{\{n_i>0\}}n_i
w_l\partial_{n-e_i}^{m+e_i}g_2
-\mathbf{1}_{\{n_2>0\}}\alpha n_2
w_l\partial_{n-e_2+e_1}^{m}g_2
\\
&-
e^{-\beta t}
\sum_{0<m'\le m} C_{m,m'}
(\partial_x^{m'}\nabla_x\phi)\cdot
w_l\nabla_v\partial_n^{m-m'}g_2.
\end{aligned}
\right.
\end{equation}

Accordingly, we define the characteristic curves $[s,X(s;t,x,v),V(s;t,x,v)]$ associated with \eqref{hmn-g1}--\eqref{hmn-g2} by
\begin{equation}\label{char-h}
\left\{
\begin{aligned}
&\frac{d}{ds}X(s;t,x,v)=e^{\beta s}V(s;t,x,v),\\
&\frac{d}{ds}V(s;t,x,v)
=
e^{-\beta s}\nabla_x\phi\big(s,X(s;t,x,v)\big)
-\alpha V_2(s;t,x,v)e_1
-\beta V(s;t,x,v),\\
&X(t;t,x,v)=x,\quad V(t;t,x,v)=v.
\end{aligned}
\right.
\end{equation}
The equivalent integral formulation is
\begin{equation}\label{char-h-int}
\left\{
\begin{aligned}
X(s;t,x,v)
&=
x-\int_s^t e^{\beta \tau}V(\tau;t,x,v)\,d\tau,
\\
V_1(s;t,x,v)
&=
e^{\beta(t-s)}v_1
+\alpha e^{\beta(t-s)}(t-s)v_2
-e^{-\beta s}\int_s^t \partial_{x_1}\phi(\tau,X(\tau;t,x,v))\,d\tau
\\
&\quad
-\alpha e^{-\beta s}\int_s^t (\tau-s)\,
\partial_{x_2}\phi(\tau,X(\tau;t,x,v))\,d\tau,
\\
V_2(s;t,x,v)
&=
e^{\beta(t-s)}v_2
-e^{-\beta s}\int_s^t \partial_{x_2}\phi(\tau,X(\tau;t,x,v))\,d\tau,
\\
V_3(s;t,x,v)
&=
e^{\beta(t-s)}v_3
-e^{-\beta s}\int_s^t \partial_{x_3}\phi(\tau,X(\tau;t,x,v))\,d\tau.
\end{aligned}
\right.
\end{equation}

\begin{lemma}\label{lemma-poisson-3D}
Let $\phi=\phi(t,x)$ satisfy \eqref{Poisson-g1g2}. If $l\ge 3$, then for any multi-index $m\in\mathbb N^3$, we have
\begin{equation}\label{poisson-est-3D}
\|\partial_x^m\nabla_x\phi(t)\|_{L_x^\infty}
\le
Cl^{-\frac32}
\Big(
\|w_l\partial_x^m g_1\|_{L_{x,v}^\infty}
+\|w_l\partial_x^m g_2\|_{L_{x,v}^\infty}
\Big),
\end{equation}
where the constant $C>0$ is independent of $l$.
\end{lemma}

\begin{proof}
Since $\displaystyle\int_{\mathbb T^3}\phi\,dx=0$, the Sobolev embedding
$W^{1,4}(\mathbb T^3)\hookrightarrow L^\infty(\mathbb T^3)$
(see \cite[(2.1) and Subsection~2.4]{BenyiOh13}), together with
\eqref{elliptic-Wkp} in Lemma~\ref{lem-elliptic-T3}, yields
\begin{equation}\label{Sobolev}
\|\partial_x^m\nabla_x\phi(t)\|_{L_x^\infty}
\le
C\|\partial_x^m\nabla_x\phi(t)\|_{W_x^{1,4}}
\le
C\|\partial_x^m\Delta_x\phi(t)\|_{L_x^4}.
\end{equation}
On the other hand,
\[
\|\partial_x^m\Delta_x\phi(t)\|_{L_x^4}
\le
\left(
\|w_l\partial_x^m g_1\|_{L_{x,v}^\infty}
+\|w_l\partial_x^m g_2\|_{L_{x,v}^\infty}
\right)
\int_{\mathbb R^3}w_{-l}\,dv.
\]
Moreover,
\[
\int_{\mathbb R^3}w_{-l}\,dv
=
\int_{\mathbb R^3}(1+|v|^2)^{-l}\,dv
=
4\pi\int_0^\infty r^2(1+r^2)^{-l}\,dr
=
\pi^{\frac32}\frac{\Gamma(l-\frac32)}{\Gamma(l)}
\le
Cl^{-\frac32},
\quad l\ge 3.
\]
Substituting the above estimate back yields \eqref{poisson-est-3D}. This completes the proof.
\end{proof}

\begin{lemma}\label{lem-char-3D}
Let $[g_1,g_2]\in\widetilde{\mathcal Y}_T$, and let $\phi=\phi(t,x)$ satisfy \eqref{Poisson-g1g2}. Assume that $\nabla_x\phi$ is continuous in $t\in[0,T]$. Then, for any given $t\in[0,T]$ and any $(x,v)\in\mathbb T^3\times\mathbb R^3$, the characteristic system \eqref{char-h} admits a unique solution $(X(s),V(s))=(X,V)(s;t,x,v)$ on $[0,t]$. Moreover, $(X(s),V(s))$ satisfies the integral formulation \eqref{char-h-int}.
\end{lemma}

\begin{proof}
Set
\[
\mathcal F(s,X,V)
:=
\left(e^{\beta s}V,\ e^{-\beta s}\nabla_x\phi(s,X)-\alpha V_2e_1-\beta V\right).
\]

By Lemma~\ref{lemma-poisson-3D} and $[g_1,g_2]\in\widetilde{\mathcal Y}_T$, we have
\[
\sup_{0\le s\le T}
\left(
\|\nabla_x\phi(s)\|_{L_x^\infty}
+
\|\nabla_x^2\phi(s)\|_{L_x^\infty}
\right)
<\infty.
\]
Hence $\nabla_x\phi$ is uniformly bounded and uniformly Lipschitz in $x$ on
$[0,T]\times\mathbb T^3$. Indeed, for any $X_1,X_2\in\mathbb T^3$,
\[
\big|
\nabla_x\phi(s,X_1)-\nabla_x\phi(s,X_2)
\big|
\le
\|\nabla_x^2\phi(s)\|_{L_x^\infty}|X_1-X_2|.
\]
Therefore, for $Y_i=(X_i,V_i)$, $i=1,2$, we have
\[
\begin{aligned}
|\mathcal F(s,Y_1)-\mathcal F(s,Y_2)|
&\le
e^{|\beta|T}|V_1-V_2|
+
e^{|\beta|T}
\big|
\nabla_x\phi(s,X_1)-\nabla_x\phi(s,X_2)
\big|
\\
&\quad
+
|\alpha| |(V_1)_2-(V_2)_2|
+
|\beta| |V_1-V_2|
\\
&\le
C_T
(|X_1-X_2|+|V_1-V_2|).
\end{aligned}
\]
Thus $\mathcal F$ is globally Lipschitz in $(X,V)$, uniformly for $s\in[0,T]$. Moreover, by the assumed continuity of $\nabla_x\phi$ in $t$, the vector field $\mathcal F$ is continuous in $s$. Hence, by the Picard--Lindel\"of theorem, for any $t\in[0,T]$ and any terminal point $(x,v)\in\mathbb T^3\times\mathbb R^3$, the characteristic system \eqref{char-h} admits a unique solution $(X,V)(s;t,x,v)$ on $[0,t]$. Integrating \eqref{char-h} from $s$ to $t$ gives \eqref{char-h-int}. This completes the proof.
\end{proof}

Applying the Duhamel formula to \eqref{hmn-g1}--\eqref{hmn-g2}, we obtain
\begin{equation}\label{duhamel-h12}
\left\{
\begin{aligned}
h_1^{m,n}(t,x,v)
&=
e^{-\int_0^t \mathcal A_{l,\alpha}(\tau,X(\tau),V(\tau))\,d\tau}
w_l\partial_n^m\tilde f_0(X(0;t,x,v),V(0;t,x,v))
\\
&\quad
+\int_0^t
e^{-\int_s^t \mathcal A_{l,\alpha}(\tau,X(\tau),V(\tau))\,d\tau}
\mathcal R_1^{m,n}(s,X(s;t,x,v),V(s;t,x,v))\,ds,
\\[1ex]
h_2^{m,n}(t,x,v)
&=
\int_0^t
e^{-\int_s^t \mathcal A_{l,\alpha}(\tau,X(\tau),V(\tau))\,d\tau}
\mathcal R_2^{m,n}(s,X(s;t,x,v),V(s;t,x,v))\,ds.
\end{aligned}
\right.
\end{equation}
where
\begin{equation}\label{calA-def}
\mathcal A_{l,\alpha}(t,x,v)
:=
\nu_0-(3+|n|)\beta
+2l\beta\frac{|v|^2}{1+|v|^2}
+2l\alpha\frac{v_1v_2}{1+|v|^2}
-2l e^{-\beta t}\frac{v\cdot\nabla_x\phi(t,x)}{1+|v|^2}.
\end{equation}
For convenience, we write
\begin{equation}
\Lambda(s,t):=\int_s^t \mathcal A_{l,\alpha}(\tau,X(\tau),V(\tau))\,d\tau.
\end{equation}
It follows from \eqref{duhamel-h12} that
\[
h_1^{m,n}(t,x,v):=\sum_{j=0}^{8}\mathcal I_j^{(1)},
\quad
h_2^{m,n}(t,x,v):=\sum_{j=1}^{6}\mathcal I_j^{(2)},
\]
where
\begin{equation}\label{I1-decomp}
\left\{
\begin{aligned}
\mathcal I_0^{(1)}
={}&
e^{-\Lambda(0,t)}
\big(w_l\partial_n^m\tilde f_0\big)(X(0),V(0)),
\\
\mathcal I_1^{(1)}
={}&
\int_0^t e^{-\Lambda(s,t)}
\big[w_l\partial_n^m(\chi_M\mathcal K g_1)\big](s,X(s),V(s))\,ds,
\\
\mathcal I_2^{(1)}
={}&
-\int_0^t e^{-\Lambda(s,t)}
\Big[
\frac{\beta}{2}w_l\partial_n^m(|v|^2\mu^{\frac{1}{2}}g_2)
+
\frac{\alpha}{2}w_l\partial_n^m(\mu^{\frac{1}{2}}v_1v_2 g_2)
\Big](s,X(s),V(s))\,ds,
\\
\mathcal I_3^{(1)}
={}&
\frac{1}{2}
\sum_{0\le m'\le m} C_{m,m'}
\int_0^t e^{-\Lambda(s,t)}e^{-\beta s}
\Big[
\partial_x^{m'}\nabla_x\phi(s,X(s))
\cdot
\big(w_l\partial_n^{m-m'}(\mu^{\frac{1}{2}}v g_2)\big)(s,X(s),V(s))
\Big]\,ds,
\\
\mathcal I_4^{(1)}
={}&
\int_0^t e^{-\Lambda(s,t)}
\big(w_l\partial_n^m\widetilde H(g_1,g_2)\big)(s,X(s),V(s))\,ds,
\\
\mathcal I_5^{(1)}
={}&
-\alpha
\int_0^t e^{-\Lambda(s,t)}e^{-\beta s}
\big[w_l\partial_n^m\big(\nabla_x\phi\cdot\nabla_v(\mu^{\frac{1}{2}}G_1)\big)\big](s,X(s),V(s))\,ds,
\\
\mathcal I_6^{(1)}
={}&
-\sum_{i=1}^{3}\mathbf{1}_{\{n_i>0\}}n_i
\int_0^t e^{-\Lambda(s,t)}e^{\beta s}
\big(w_l\partial_{n-e_i}^{m+e_i}g_1\big)(s,X(s),V(s))\,ds,
\\
\mathcal I_7^{(1)}
={}&
-\mathbf{1}_{\{n_2>0\}}\alpha n_2
\int_0^t e^{-\Lambda(s,t)}
\big(w_l\partial_{n-e_2+e_1}^{m}g_1\big)(s,X(s),V(s))\,ds,
\\
\mathcal I_8^{(1)}
={}&
-\sum_{0<m'\le m} C_{m,m'}
\int_0^t e^{-\Lambda(s,t)}e^{-\beta s}
\Big[
\partial_x^{m'}\nabla_x\phi(s,X(s))
\cdot
(w_l\nabla_v\partial_n^{m-m'}g_1)(s,X(s),V(s))
\Big]\,ds.
\end{aligned}
\right.
\end{equation}
and
\begin{equation}\label{I2-decomp}
\left\{
\begin{aligned}
\mathcal I_1^{(2)}
={}&
\int_0^t e^{-\Lambda(s,t)}
\big[w_l\partial_n^m(Kg_2)\big](s,X(s),V(s))\,ds,
\\
\mathcal I_2^{(2)}
={}&
\int_0^t e^{-\Lambda(s,t)}
\big[w_l\partial_n^m\big(\mu^{-\frac{1}{2}}(1-\chi_M)\mathcal K g_1\big)\big](s,X(s),V(s))\,ds,
\\
\mathcal I_3^{(2)}
={}&
\int_0^t e^{-\Lambda(s,t)}e^{-\beta s}
\big[w_l\partial_n^m(\mu^{\frac{1}{2}}v\cdot\nabla_x\phi)\big](s,X(s),V(s))\,ds,
\\
\mathcal I_4^{(2)}
={}&
-\sum_{i=1}^{3}\mathbf{1}_{\{n_i>0\}}n_i
\int_0^t e^{-\Lambda(s,t)}e^{\beta s}
\big(w_l\partial_{n-e_i}^{m+e_i}g_2\big)(s,X(s),V(s))\,ds,
\\
\mathcal I_5^{(2)}
={}&
-\mathbf{1}_{\{n_2>0\}}\alpha n_2
\int_0^t e^{-\Lambda(s,t)}
\big(w_l\partial_{n-e_2+e_1}^{m}g_2\big)(s,X(s),V(s))\,ds,
\\
\mathcal I_6^{(2)}
={}&
-\sum_{0<m'\le m} C_{m,m'}
\int_0^t e^{-\Lambda(s,t)}e^{-\beta s}
\Big[
\partial_x^{m'}\nabla_x\phi(s,X(s))
\cdot
(w_l\nabla_v\partial_n^{m-m'}g_2)(s,X(s),V(s))
\Big]\,ds.
\end{aligned}
\right.
\end{equation}

For notational convenience, define
\begin{equation}\label{eq-def-mfh12mn}
\mathfrak H_1^{m,n}(t)
=
\sup_{0\le s\le t}e^{\lambda_n^m s}
\|h_1^{m,n}(s)\|_{L^\infty_{x,v}},
\quad
\mathfrak H_2^{m,n}(t)
=
\sup_{0\le s\le t}e^{\lambda_n^m s}
\|h_2^{m,n}(s)\|_{L^\infty_{x,v}}.
\end{equation}
The estimates of the terms in the above decompositions are summarized in the
following lemma.

\begin{lemma}\label{lem-Linfty-sum-h1h2}
Let $[g_1,g_2]$ be the solution to
\eqref{g1-eq}--\eqref{g2-eq} satisfying the a priori assumption
\eqref{assa}. For any sufficiently small $\kappa>0$ and any fixed
sufficiently large $l$, the following estimates hold.
\begin{enumerate}
\item[\rm(i)] The case $|m|>0:$
\allowdisplaybreaks\begin{align}
\sum_{\substack{|m|+|n|\le N\\ |m|>0}}
\kappa^{|n|}\mathfrak H_1^{m,n}(t)
&\le
C
\sum_{\substack{|m|+|n|\le N\\ |m|>0}}
\kappa^{|n|}
\big\|w_l\partial_n^m \tilde f_0\big\|_{L^\infty_{x,v}}
+
C(\eta+\alpha)
\sum_{\substack{|m|+|n|\le N\\ |m|>0}}
\kappa^{|n|}
\mathfrak H_2^{m,n}(t),
\label{est-positive-sumh1}
\\
\sum_{\substack{|m|+|n|\le N\\ |m|>0}}
\kappa^{|n|}\mathfrak H_2^{m,n}(t)
&\le
C
\sum_{\substack{|m|+|n|\le N\\ |m|>0}}
\kappa^{|n|}\mathfrak H_1^{m,n}(t)
+
C
\sum_{\substack{|m|+|n|\le N\\ |m|>0}}
\kappa^{|n|}
\sup_{0\le s\le t}
e^{\lambda_n^m s}
\big\|\partial_n^m g_2(s)\big\|_{L^2_{x,v}}.
\label{est-positive-sumh2}
\end{align}
\text{In particular, we also have}
\allowdisplaybreaks\begin{align}
\sum_{0<|m|\le N}
\mathfrak H_1^{m,0}(t)
&\le
C
\sum_{0<|m|\le N}
\big\|w_l\partial_x^m \tilde f_0\big\|_{L^\infty_{x,v}}
+
C(\eta+\alpha)
\sum_{0<|m|\le N}
\mathfrak H_2^{m,0}(t),
\label{est-positive-n0-sumh1}
\\
\sum_{0<|m|\le N}
\mathfrak H_2^{m,0}(t)
&\le
C
\sum_{0<|m|\le N}
\mathfrak H_1^{m,0}(t)
+
C
\sum_{0<|m|\le N}
\sup_{0\le s\le t}
e^{\lambda s}
\big\|\partial_x^m g_2(s)\big\|_{L^2_{x,v}}.
\label{est-positive-n0-sumh2}
\end{align}
\item[\rm(ii)] The case $m=n=0:$
\begin{align}
\mathfrak H_1^{0,0}(t)
&\le
C\|w_l\tilde f_0\|_{L^\infty_{x,v}}
+
C(\eta+\alpha)
\mathfrak H_2^{0,0}(t),
\label{est-zero-sumh1}
\\
\mathfrak H_2^{0,0}(t)
&\le
C\mathfrak H_1^{0,0}(t)
+
C\sup_{0\le s\le t}
\|g_2(s)\|_{L^2_{x,v}}.
\label{est-zero-sumh2}
\end{align}
\item[\rm(iii)] The case $m=0,\ 0<|n|\le N:$
\allowdisplaybreaks\begin{align}
\sum_{0<|n|\le N}
\kappa^{|n|}\mathfrak H_1^{0,n}(t)
&\le
C
\sum_{|n|\le N}
\kappa^{|n|}
\big\|w_l\partial_v^n\tilde f_0\big\|_{L^\infty_{x,v}}
+
C(\eta+\alpha)
\sum_{|n|\le N}
\kappa^{|n|}
\mathfrak H_2^{0,n}(t)
+
C\kappa
\sum_{\substack{|m|+|n|\le N\\ |m|=1}}
\kappa^{|n|}
\mathfrak H_1^{m,n}(t),
\label{est-v-sumh1}
\\
\sum_{0<|n|\le N}
\kappa^{|n|}\mathfrak H_2^{0,n}(t)
&\le
C
\sum_{|n|\le N}
\kappa^{|n|}
\mathfrak H_1^{0,n}(t)
+
C\kappa
\sum_{\substack{|m|+|n|\le N\\ |m|=1}}
\kappa^{|n|}
\mathfrak H_2^{m,n}(t)
+
C
\sum_{|n|\le N}
\kappa^{|n|}
\sup_{0\le s\le t}
\big\|\partial_v^n g_2(s)\big\|_{L^2_{x,v}}.
\label{est-v-sumh2}
\end{align}
\end{enumerate}
\end{lemma}

\begin{proof}

\noindent{\textbf{\underline{Step 1. Weighted $L^\infty$ Estimates.}}}
By taking $\eta,\alpha>0$ sufficiently small and using Lemma \ref{lemma-poisson-3D}, together with the a priori assumption \eqref{assa}, we have
\[
\begin{aligned}
\mathcal A_{l,\alpha}(t,x,v)
&\ge
\nu_0-(3+N)\beta-l\alpha
-l e^{-\beta t}\|\nabla_x\phi(t)\|_{L^\infty_x}
\\
&\ge
\nu_0-(3+N)\beta-l\alpha
-C l^{-\frac{1}{2}}\eta
\\
&\ge \frac{\nu_0}{2},
\end{aligned}
\]
which leads to
\[
\Lambda(s,t)\geq\frac{\nu_0}{2}(t-s),
\quad
0\leq s\leq t\leq T.
\]
By the choice of $\lambda$ in \eqref{choice-lambda} and the smallness
condition on $\alpha$ imposed after \eqref{def-lambda-mn}, we have
\[
0\leq\lambda_n^m<\frac{\nu_0}{2},
\quad
|m|+|n|\leq N.
\]
Moreover, since $|n'|\le |n|$ implies $\lambda_{n'}^m\ge \lambda_n^m$, it follows that
\begin{equation}\label{easy-weight-bound}
\int_0^t e^{\lambda_n^m t-\Lambda(s,t)-\lambda_{n'}^m s}\,ds
\le
\int_0^t e^{-(\frac{\nu_0}{2}-\lambda_n^m)(t-s)}\,ds
\le C,
\quad \forall\ |n'|\le |n|.
\end{equation}

For the initial term, it follows immediately from the definition \eqref{I1-decomp} that
\begin{equation}\label{weighted-I01}
e^{\lambda_n^m t}|\mathcal I_0^{(1)}|
\le \|w_l\partial_n^m\tilde f_0\|_{L^\infty_{x,v}}.
\end{equation}

For $\mathcal I_1^{(1)}$, invoking Lemma~\ref{lem-DL-K}, we set
$M=M(l)=l^2$ with $l$ sufficiently large, which yields
\allowdisplaybreaks\begin{align}
e^{\lambda_n^m t}|\mathcal I_1^{(1)}|
&\le \int_0^t e^{\lambda_n^m t-\Lambda(s,t)}
\|w_l\partial_n^m(\chi_M\mathcal K g_1)(s)\|_{L^\infty_{x,v}}\,ds
\notag\\
&\le \frac{C}{l}\sum_{n'\le n}\int_0^t
e^{-(\frac{\nu_0}{2}-\lambda_n^m)(t-s)}
e^{\lambda_{n'}^m s}\|w_l\partial_{n'}^m g_1(s)\|_{L^\infty_{x,v}}\,ds
\notag\\
&\le \frac{C}{l}\sum_{n'\le n}\mathfrak H_1^{m,n'}(t).
\label{weighted-I11}
\end{align}

For $\mathcal I_2^{(1)}$, a direct estimate gives
\allowdisplaybreaks\begin{align}
e^{\lambda_n^m t}|\mathcal I_2^{(1)}|
&\le C(\beta+\alpha)\sum_{n'\le n}\int_0^t
e^{-(\frac{\nu_0}{2}-\lambda_n^m)(t-s)}
e^{\lambda_{n'}^m s}\|w_l\partial_{n'}^m g_2(s)\|_{L^\infty_{x,v}}\,ds
\notag\\
&\le C\alpha\sum_{n'\le n}\mathfrak H_2^{m,n'}(t).
\label{weighted-I21}
\end{align}

For $\mathcal I_5^{(1)}$, by Lemma~\ref{lem-DL-G} and Lemma~\ref{lemma-poisson-3D}, we obtain
\[
e^{-\beta s}
\|w_l\partial_n^m(\nabla_x\phi\cdot\nabla_v(\mu^{\frac{1}{2}}G_1))(s)\|_{L^\infty_{x,v}}
\le C_l\Big(
\|w_l\partial_x^m g_1(s)\|_{L^\infty_{x,v}}
+\|w_l\partial_x^m g_2(s)\|_{L^\infty_{x,v}}
\Big).
\]
Hence
\allowdisplaybreaks\begin{align}
e^{\lambda_n^m t}|\mathcal I_5^{(1)}|
&\le C_l\alpha\int_0^t
e^{-(\frac{\nu_0}{2}-\lambda_n^m)(t-s)}
\Big(
e^{\lambda_0^m s}\|w_l\partial_x^m g_1(s)\|_{L^\infty_{x,v}}+
e^{\lambda_0^m s}\|w_l\partial_x^m g_2(s)\|_{L^\infty_{x,v}}
\Big)\,ds
\notag\\
&\le C_l\alpha
\Big(
\mathfrak H_1^{m,0}(t)+
\mathfrak H_2^{m,0}(t)
\Big).
\label{weighted-I51}
\end{align}

For $\mathcal I_7^{(1)}$ and $\mathcal I_5^{(2)}$, if $n_2\ne0$, then $|n+e_1-e_2|=|n|$. Hence,
\allowdisplaybreaks\begin{align}
e^{\lambda_n^m t}|\mathcal I_7^{(1)}|
&\le \mathbf{1}_{\{n_2>0\}}C\alpha\int_0^t
e^{-(\frac{\nu_0}{2}-\lambda_n^m)(t-s)}
e^{\lambda_n^m s}\|w_l\partial_{n+e_1-e_2}^m g_1(s)\|_{L^\infty_{x,v}}\,ds
\notag\\
&\le \mathbf{1}_{\{n_2>0\}}C\alpha\mathfrak H_1^{m,n+e_1-e_2}(t),
\label{weighted-I71}\\
e^{\lambda_n^m t}|\mathcal I_5^{(2)}|
&\le \mathbf{1}_{\{n_2>0\}}C\alpha\int_0^t
e^{-(\frac{\nu_0}{2}-\lambda_n^m)(t-s)}
e^{\lambda_n^m s}\|w_l\partial_{n+e_1-e_2}^m g_2(s)\|_{L^\infty_{x,v}}\,ds
\notag\\
&\le \mathbf{1}_{\{n_2>0\}}C\alpha\mathfrak H_2^{m,n+e_1-e_2}(t).
\label{weighted-I52}
\end{align}

For $\mathcal I_2^{(2)}$, it follows from \eqref{easy-weight-bound} and the boundedness of $(1-\chi_M)\mathcal K$ that
\allowdisplaybreaks\begin{equation}\label{weighted-I22}
e^{\lambda_n^m t}|\mathcal I_2^{(2)}|
\le \int_0^t e^{\lambda_n^m t-\Lambda(s,t)}
\|w_l\partial_n^m[\mu^{-\frac{1}{2}}(1-\chi_M)\mathcal K g_1](s)\|_{L^\infty_{x,v}}\,ds
\le C\sum_{n'\le n}\mathfrak H_1^{m,n'}(t).
\end{equation}

For $\mathcal I_6^{(1)}$ and $\mathcal I_4^{(2)}$, since the corresponding terms vanish when $n_i=0$, one has
\allowdisplaybreaks\begin{align}
e^{\lambda_n^m t}|\mathcal I_6^{(1)}|
&\le
C\sum_{i=1}^3\mathbf{1}_{\{n_i>0\}}n_i
\mathfrak H_1^{m+e_i,n-e_i}(t),
\label{weighted-I61}\\
e^{\lambda_n^m t}|\mathcal I_4^{(2)}|
&\le
C\sum_{i=1}^3\mathbf{1}_{\{n_i>0\}}n_i
\mathfrak H_2^{m+e_i,n-e_i}(t).
\label{weighted-I42}
\end{align}

For $\mathcal I_{3}^{(1)}$, 
we distinguish the cases $|m|=0$ and $|m|\ge 1$.
First, by \eqref{poisson-est-3D} in Lemma~\ref{lemma-poisson-3D},
\allowdisplaybreaks\begin{align}
|\mathcal I_3^{(1)}|
&\le
C\sum_{m'\le m}\sum_{n'\le n}
\int_0^t e^{-\Lambda(s,t)}e^{-\beta s}
\Big(
\|w_l\partial_x^{m'}g_1(s)\|_{L^\infty_{x,v}}
+
\|w_l\partial_x^{m'}g_2(s)\|_{L^\infty_{x,v}}
\Big)
\notag\\
&\quad\quad\quad\quad\times
\|w_l\partial_{n-n'}^{m-m'}g_2(s)\|_{L^\infty_{x,v}}\,ds.
\label{I31-pre}
\end{align}
\noindent
\textbf{Case 1: $|m|=0$.}
In this case $m'=0$ only. Since $\lambda_n^0=0$, from \eqref{I31-pre} and using the a priori assumption \eqref{assa} on the second factor
\[
\sup_{0\le s\le t}\|w_l\partial_v^{n-n'}g_2(s)\|_{L^\infty_{x,v}}\le \eta,
\]
we obtain
\begin{equation}\label{weighted-I31-m0}
|\mathcal I_3^{(1)}|
\le
C\eta
\Big(
\mathfrak H_1^{0,0}(t)
+
\mathfrak H_2^{0,0}(t)
\Big).
\end{equation}

\noindent
\textbf{Case 2: $|m|> 0$.}
For the term $m'=0$, we use the a priori assumption \eqref{assa} on the zero-th spatial factor in \eqref{I31-pre}; for the terms $0<m'\le m$, we preserve the positive spatial derivative factor and apply \eqref{assa} to the other factor. Thus
\begin{equation}\label{weighted-I31-mge1}
e^{\lambda_n^m t}|\mathcal I_3^{(1)}|
\le
C\eta
\sum_{n'\le n}
\mathfrak H_2^{m,n'}(t)
+
C\eta
\sum_{0<m'\le m}
\Big(
\mathfrak H^{m',0}_1(t)
+
\mathfrak H^{m',0}_2(t)
\Big).
\end{equation}

By the same argument, one immediately obtains the following bounds for $\mathcal I_8^{(1)}$ and $\mathcal I_6^{(2)}$,
\begin{equation}\label{weighted-I81}
e^{\lambda_n^m t}\big(|\mathcal I_8^{(1)}|+|\mathcal I_6^{(2)}|\big)
\le
C\eta
\sum_{0<m'\le m}
\Big(
\mathfrak H^{m',0}_1(t)
+
\mathfrak H^{m',0}_2(t)
\Big).
\end{equation}

For $\mathcal I_4^{(1)}$, it follows from \eqref{eq-Q-Linf} that, for any functions $F_1$ and $F_2$,
\begin{equation}\label{est-Q-F12}
\|w_l\partial_n^m Q(F_1,F_2)\|_{L^\infty_{x,v}}
\le
C\sum_{m'\le m}\sum_{n'\le n}
\|w_l\partial_{n'}^{m'}F_1\|_{L^\infty_{x,v}}
\|w_l\partial_{n-n'}^{m-m'}F_2\|_{L^\infty_{x,v}}.
\end{equation}
Using the a priori assumption~\eqref{assa}, for
$F_1,F_2\in\{g_1,\mu^{\frac{1}{2}}g_2,\alpha\mu^{\frac{1}{2}}G_1\}$, we have
\begin{equation}\label{est-Q-F12-1}
\|w_l\partial_{n'}^{m'}F_1\|_{L^\infty_{x,v}}
+
\|w_l\partial_{n-n'}^{m-m'}F_2\|_{L^\infty_{x,v}}
\le C\eta+C_l\alpha.
\end{equation}
\textbf{Case 1: $|m|=0$.}
We directly obtain
\begin{equation}\label{est-eI410}
|\mathcal I_4^{(1)}|
\le
(C\eta+C_l\alpha)\sum_{n'\le n}
\Big(
\mathfrak H_1^{0,n'}(t)
+\mathfrak H_2^{0,n'}(t)
\Big).
\end{equation}
\textbf{Case 2: $|m|>0$.} We apply the bound \eqref{est-Q-F12-1} to the lower spatial order factor in each product term of \eqref{est-Q-F12}, and keep the factor with positive spatial derivatives. Hence
\begin{equation}\label{weighted-eI41}
e^{\lambda_n^m t}|\mathcal I_4^{(1)}|
\le
(C\eta+C_l\alpha)\sum_{0< m'\le m}\sum_{n'\le n}
\Big(
\mathfrak H_1^{m',n'}(t)
+\mathfrak H_2^{m',n'}(t)
\Big),
\quad |m|>0.
\end{equation}

For $\mathcal I_3^{(2)}$. For each fixed $n$, since $\mu$ is Gaussian and $w_l$ is polynomial, we have
\begin{equation}\label{I32-1}
e^{\lambda_n^m t}|\mathcal I_3^{(2)}|
\le
C_l\int_0^t
e^{-(\frac{\nu_0}{2}-\lambda_n^m)(t-s)}
e^{(\lambda_n^m-\beta)s}
\|\partial_x^m\nabla_x\phi(s)\|_{L_x^\infty}\,ds.
\end{equation}
Next, by the same argument as in \eqref{Sobolev},
\[
\begin{aligned}
\|\partial_x^m\nabla_x\phi\|_{L_x^\infty}
&\leq
C\|\partial_x^m\Delta_x\phi\|_{L_x^4}
\\
&\leq
C\Big\|\int_{\mathbb R^3}\partial_x^m g_1\,dv\Big\|_{L_x^4}
+
C\Big\|\int_{\mathbb R^3}
\mu^{\frac12}\partial_x^m g_2\,dv\Big\|_{L_x^4}
\\
&\leq
C\|w_l\partial_x^m g_1\|_{L^\infty_{x,v}}
+
C\|\partial_x^m g_2\|_{L_x^4L_v^2}.
\end{aligned}
\]
Using the interpolation inequality in the $x$-variable, for any
$\varepsilon_1>0$, we have
\[
\|\partial_x^m g_2\|_{L_x^4L_v^2}
\le
\|\partial_x^m g_2\|_{L_{x,v}^2}^{\frac12}
\|\partial_x^m g_2\|_{L_x^\infty L_v^2}^{\frac12}
\le
\varepsilon_1\|w_l\partial_x^m g_2\|_{L^\infty_{x,v}}
+
C_{\varepsilon_1}\|\partial_x^m g_2\|_{L_{x,v}^2}.
\]
Combining this with the preceding estimates, we obtain
\begin{equation}\label{phi-Linf}
\|\partial_x^m\nabla_x\phi\|_{L_x^\infty}
\le
C_l\|h_1^{m,0}\|_{L^\infty_{x,v}}
+\varepsilon_1\|h_2^{m,0}\|_{L^\infty_{x,v}}
+C_{l,\varepsilon_1}\|\partial_x^m g_2\|_{L_{x,v}^2}.
\end{equation}
Substituting \eqref{phi-Linf} into \eqref{I32-1}, we deduce
\begin{equation}\label{weihghted-eI32}
e^{\lambda_n^m t}|\mathcal I_3^{(2)}|
\le
C_l\Big(
\mathfrak H_1^{m,0}(t)
+\varepsilon_1\mathfrak H_2^{m,0}(t)
\Big)
+C_{l,\varepsilon_1}\sup_{0\le s\le t}
e^{\lambda_0^m s}\|\partial_x^m g_2(s)\|_{L^2_{x,v}}.
\end{equation}

\noindent{\textbf{\underline{Step 2. $L^\infty$-$L^2$ Estimates for $\mathcal I_1^{(2)}$.}}}
For any $0<\varepsilon<t$, we first decompose $\mathcal I_1^{(2)}$ in time as
\[
\mathcal I_1^{(2)}:=\mathcal I_{1,1}^{(2)}+\mathcal I_{1,2}^{(2)},
\]
where
\[
\left\{
\begin{aligned}
\mathcal I_{1,1}^{(2)}
&=
\int_0^{t-\varepsilon} e^{-\Lambda(s,t)}
\big[w_l\partial_n^m(Kg_2)\big](s,X(s),V(s))\,ds,\\
\mathcal I_{1,2}^{(2)}
&=
\int_{t-\varepsilon}^t e^{-\Lambda(s,t)}
\big[w_l\partial_n^m(Kg_2)\big](s,X(s),V(s))\,ds.
\end{aligned}
\right.
\]
It is easy to obtain
\begin{equation}\label{I122}
e^{\lambda_n^m t}|\mathcal I_{1,2}^{(2)}|
\le C\varepsilon
\sum_{|n'|\le |n|}
\mathfrak H_2^{m,n'}(t).
\end{equation}

Next, we further decompose $\mathcal I_{1,1}^{(2)}$ according to the velocity regions as
\begin{equation}\label{I12-v}
\mathcal I_{1,1}^{(2)}
:=
\mathcal I_{1,1,1}^{(2)}
+\mathcal I_{1,1,2}^{(2)}
+\mathcal I_{1,1,3}^{(2)}
+\mathcal I_{1,1,4}^{(2)},
\end{equation}
where
\begin{equation}\label{I12-vsplit}
\left\{
\begin{aligned}
\mathcal I_{1,1,1}^{(2)}
&=
\int_0^{t-\varepsilon} e^{-\Lambda(s,t)}
\mathbf 1_{\{|V(s)|\ge M_0\}}
\big[w_l\partial_n^m(Kg_2)\big](s,X(s),V(s))\,ds,\\
\mathcal I_{1,1,2}^{(2)}
&=
\int_0^{t-\varepsilon} e^{-\Lambda(s,t)}
\mathbf 1_{\{|V(s)|\le M_0\}}
\big[w_l\partial_n^m\big(K(\mathbf 1_{\{|v_*|\ge 2M_0\}}g_2)\big)\big](s,X(s),V(s))\,ds,\\
\mathcal I_{1,1,3}^{(2)}
&=
\int_0^{t-\varepsilon} e^{-\Lambda(s,t)}
\mathbf 1_{\{|V(s)|\le M_0\}}
\big[w_l\partial_n^m\big(K(\mathbf 1_{\Omega_1(s)}g_2)\big)\big](s,X(s),V(s))\,ds,\\
\mathcal I_{1,1,4}^{(2)}
&=
\int_0^{t-\varepsilon} e^{-\Lambda(s,t)}
\mathbf 1_{\{|V(s)|\le M_0\}}
\big[w_l\partial_n^m\big(K(\mathbf 1_{\Omega_2(s)}g_2)\big)\big](s,X(s),V(s))\,ds,
\end{aligned}
\right.
\end{equation}
with
\[
\left\{
\begin{aligned}
\Omega_1(s)&=\Big\{v_*\in\mathbb R^3:\ |v_*|\le 2M_0,\ |V(s)-v_*|<\frac{1}{M_0}\Big\},\\
\Omega_2(s)&=\Big\{v_*\in\mathbb R^3:\ |v_*|\le 2M_0,\ |V(s)-v_*|\ge \frac{1}{M_0}\Big\}.
\end{aligned}
\right.
\]

For $\mathcal I_{1,1,1}^{(2)}$, since $|V(s)|\ge M_0$, by Lemma \ref{lem-K}, namely \eqref{est-Kq-C1} and \eqref{est-Kq-C2}, we have
\allowdisplaybreaks\begin{align}
e^{\lambda_n^m t}|\mathcal I_{1,1,1}^{(2)}|
&\le
\int_0^{t-\varepsilon}
e^{-(\frac{\nu_0}{2}-\lambda_n^m)(t-s)}
\mathbf 1_{\{|V(s)|\ge M_0\}}
\left(
C\int_{\mathbb R^3}
K_q(V(s),v_*)
\frac{w_l(V(s))}{w_l(v_*)}\,dv_*
\right)\,ds
\notag\\
&\quad\times
\sum_{|n'|\le |n|}
\sup_{0\le \tau\le t}
e^{\lambda_{n'}^m\tau}
\|w_l\partial_{n'}^m g_2(\tau)\|_{L_{x,v}^\infty}
\notag\\
&\le
\int_0^{t-\varepsilon}
e^{-(\frac{\nu_0}{2}-\lambda_n^m)(t-s)}
\mathbf 1_{\{|V(s)|\ge M_0\}}
\frac{C_l}{1+|V(s)|}\,ds
\sum_{|n'|\le |n|}
\sup_{0\le \tau\le t}
e^{\lambda_{n'}^m\tau}
\|w_l\partial_{n'}^m g_2(\tau)\|_{L_{x,v}^\infty}
\notag\\
&\le
\frac{C_l}{M_0}
\sum_{|n'|\le |n|}
\mathfrak H_2^{m,n'}(t),
\label{I111}
\end{align}
where $K_q$ is defined by \eqref{def-Kq}.

For $\mathcal I_{1,1,2}^{(2)}$, note that
$
|V(s)-v_*|\ge |v_*|-|V(s)|\ge M_0
$
in this case. We apply \eqref{def-Kq}--\eqref{est-Kq-C1} in Lemma \ref{lem-K} to obtain
\allowdisplaybreaks\begin{align}
e^{\lambda_n^m t}|\mathcal I_{1,1,2}^{(2)}|
&\le
\int_0^{t-\varepsilon}
e^{-(\frac{\nu_0}{2}-\lambda_n^m)(t-s)}
\mathbf 1_{\{|V(s)|\le M_0\}}
\left(
C\int_{\mathbb R^3}
K_q(V(s),v_*)
\frac{w_l(V(s))}{w_l(v_*)}
1_{\{|v_*|\ge 2M_0\}}\,dv_*
\right)\,ds
\notag\\
&\quad\times
\sum_{|n'|\le |n|}
\sup_{0\le \tau\le t}
e^{\lambda_{n'}^m\tau}
\|w_l\partial_{n'}^m g_2(\tau)\|_{L_{x,v}^\infty}
\notag\\
&\le
Ce^{-\frac{q}{16} M_0^2}
\sum_{|n'|\le |n|}
\mathfrak H_2^{m,n'}(t), \quad q\in(0,1).
\label{I1122}
\end{align}

For $\mathcal I_{1,1,3}^{(2)}$, using \eqref{est-Kq-C1} in Lemma \ref{lem-K} again, we have
\allowdisplaybreaks\begin{align}
e^{\lambda_n^m t}|\mathcal I_{1,1,3}^{(2)}|
&\le
\int_0^{t-\varepsilon}
e^{-(\frac{\nu_0}{2}-\lambda_n^m)(t-s)}
\mathbf 1_{\{|V(s)|\le M_0\}}
\left(
C\int_{\mathbb R^3}
K_q(V(s),v_*)
\frac{w_l(V(s))}{w_l(v_*)}
1_{\Omega_1(s)}\,dv_*
\right)\,ds
\notag\\
&\quad\times
\sum_{|n'|\le |n|}
\sup_{0\le \tau\le t}
e^{\lambda_{n'}^m\tau}
\|w_l\partial_{n'}^m g_2(\tau)\|_{L_{x,v}^\infty}.
\label{I1132-1}
\end{align}
On $\Omega_1(s)$, since
$\big||V(s)|-|v_*|\big|\leq|V(s)-v_*|<M_0^{-1}$, we have
\[
\frac{w_l(V(s))}{w_l(v_*)}
=
\left(
\frac{1+|V(s)|^2}{1+|v_*|^2}
\right)^l
\leq
C_l,\quad l>0.
\]
Therefore, by the definition of $K_q$ in \eqref{def-Kq},
\[
\begin{aligned}
\int_{\mathbb R^3}
K_q(V(s),v_*)
\frac{w_l(V(s))}{w_l(v_*)}
\mathbf 1_{\Omega_1(s)}\,dv_*
&\le
C_l\int_{|V(s)-v_*|<\frac{1}{M_0}}
\big(|V(s)-v_*|+|V(s)-v_*|^{-1}\big)\,dv_*
\\
&=
C_l\int_{|z|<\frac{1}{M_0}}
(|z|+|z|^{-1})\,dz
\le
\frac{C_l}{M_0^2}.
\end{aligned}
\]
Substituting the above estimate into \eqref{I1132-1}, we obtain
\begin{equation}\label{I1132}
e^{\lambda_n^m t}|\mathcal I_{1,1,3}^{(2)}|
\le
\frac{C_l}{M_0^2}
\sum_{|n'|\le |n|}
\mathfrak H_2^{m,n'}(t).
\end{equation}

Combining the estimates for
$e^{\lambda_n^m t}\sum_{i=2}^6 |\mathcal I_i^{(2)}|$,
$e^{\lambda_n^m t}|\mathcal I_{1,2}^{(2)}|$, and
$e^{\lambda_n^m t}\sum_{i=1}^3 |\mathcal I_{1,1,i}^{(2)}|$, we obtain, for any fixed multi-indices $m,n\in\mathbb N^3$ satisfying $|m|+|n|\le N$,
\begin{equation}\label{eh2-R3}
e^{\lambda_n^m t}|h_2^{m,n}(t,x,v)|
\le
\mathcal R_3^{m,n}(t)
+
e^{\lambda_n^m t}\big|\mathcal I_{1,1,4}^{(2)}\big|,
\end{equation}
where
\allowdisplaybreaks\begin{align}
\mathcal R_3^{m,n}(t):={}&
\mathbf 1_{\{|m|=0\}}C_l
\sum_{|n'|\le |n|}
\Big(
\mathfrak H_1^{0,n'}(t)+
\varepsilon_1\mathfrak H_2^{0,n'}(t)
\Big)
\notag\\
&
+\mathbf 1_{\{|m|>0\}}C_l
\sum_{0<m'\le m}\sum_{|n'|\le |n|}
\Big[
\mathfrak H_1^{m',n'}(t)+
\big(\frac{1}{M_0}+\varepsilon_1+\eta+\alpha\big)
\mathfrak H_2^{m',n'}(t)
\Big]
\notag\\
&
+C
\sum_{i=1}^3\mathbf 1_{\{n_i>0\}}n_i
\mathfrak H_2^{m+e_i,n-e_i}(t)
+C_{l,\varepsilon_1}
\sup_{0\le s\le t}
e^{\lambda_0^m s}
\|\partial_x^m g_2(s)\|_{L^2_{x,v}}.
\label{R3mn-def}
\end{align}

Similarly, by Lemma~\ref{lem-K}, we have
\begin{align}
e^{\lambda_n^m t}|\mathcal I_{1,1,4}^{(2)}|
&\le
C\sum_{|n'|\le |n|}
\int_0^{t-\varepsilon}
e^{-(\frac{\nu_0}{2}-\lambda_n^m)(t-s)}
\mathbf 1_{\{|V(s)|\le M_0\}}
\int_{\mathbb R^3}
K_q\big(V(s),v_*\big)
\frac{w_l(V(s))}{w_l(v_*)}
1_{\Omega_2(s)}
\notag\\
&\quad\times
\big|e^{\lambda_{n'}^m s}h_2^{m,n'}\big(s,X(s),v_*\big)\big|
\,dv_*ds.
\label{I114-step1}
\end{align}

For any given $s\in[0,t-\varepsilon]$ and $v_*\in\mathbb R^3$, along the backward characteristic curve starting from $[s,X(s),v_*]$, we denote by $[\tilde s,\widetilde X(\tilde s),\widetilde V(\tilde s)]$ the corresponding position and velocity at time $\tilde s\in[0,s]$. Combining \eqref{eh2-R3} with \eqref{I114-step1} and applying \eqref{est-Kq-C2}, we obtain
\begin{align}
e^{\lambda_n^m t}|h_2^{m,n}|
&\le
\mathcal R_3^{m,n}(t)
+
C\sum_{|n'|\le |n|}
\int_0^{t-\varepsilon}
e^{-(\frac{\nu_0}{2}-\lambda_n^m)(t-s)}
\mathbf 1_{\{|V(s)|\le M_0\}}
\int_{\mathbb R^3}
K_q\big(V(s),v_*\big)
\notag\\
&\qquad\qquad\quad\quad\times
\frac{w_l(V(s))}{w_l(v_*)}
\mathbf 1_{\Omega_2(s)}
\mathcal R_3^{m,n'}(s)
\,dv_*ds
\notag\\
&\quad{}+
C\sum_{|n'|\le |n|}
\int_0^{t-\varepsilon}
e^{-(\frac{\nu_0}{2}-\lambda_n^m)(t-s)}
\mathbf 1_{\{|V(s)|\le M_0\}}
\int_{\mathbb R^3}
K_q\big(V(s),v_*\big)
\frac{w_l(V(s))}{w_l(v_*)}
\mathbf 1_{\Omega_2(s)}
\notag\\
&\quad\quad\times
\sum_{|n''|\le |n'|}
\int_0^{s-\varepsilon}
e^{-(\frac{\nu_0}{2}-\lambda_{n'}^m)(s-\tilde s)}
\mathbf 1_{\{|\widetilde V(\tilde s)|\le M_0\}}
\int_{\mathbb R^3}
K_q\big(\widetilde V(\tilde s),\widetilde v_*\big)
\frac{w_l(\widetilde V(\tilde s))}{w_l(\widetilde v_*)}
\mathbf 1_{\widetilde\Omega_2(\tilde s)}
\notag\\
&\quad\quad\times
\big|e^{\lambda_{n''}^m\tilde s}
h_2^{m,n''}\big(\tilde s,\widetilde X(\tilde s),\widetilde v_*\big)\big|
\,d\widetilde v_*d\tilde sdv_*ds
\notag\\
&\le
C_l\sum_{|n'|\le |n|}\mathcal R_3^{m,n'}(t)
+
C\sum_{|n'|\le |n|}
\int_0^{t-\varepsilon}
e^{-(\frac{\nu_0}{2}-\lambda_n^m)(t-s)}
\mathbf 1_{\{|V(s)|\le M_0\}}
\int_{\mathbb R^3}
K_q\big(V(s),v_*\big)
\frac{w_l(V(s))}{w_l(v_*)}
\mathbf 1_{\Omega_2(s)}
\notag\\
&\quad\quad\times
\sum_{|n''|\le |n'|}
\int_0^{s-\varepsilon}
e^{-(\frac{\nu_0}{2}-\lambda_{n'}^m)(s-\tilde s)}
\mathbf 1_{\{|\widetilde V(\tilde s)|\le M_0\}}
\int_{\mathbb R^3}
K_q\big(\widetilde V(\tilde s),\widetilde v_*\big)
\frac{w_l(\widetilde V(\tilde s))}{w_l(\widetilde v_*)}
\mathbf 1_{\widetilde\Omega_2(\tilde s)}
\notag\\
&\quad\quad\times
\big|e^{\lambda_{n''}^m\tilde s}
h_2^{m,n''}\big(\tilde s,\widetilde X(\tilde s),\widetilde v_*\big)\big|
\,d\widetilde v_*d\tilde sdv_*ds ,
\label{eh2-step3}
\end{align}
where
\[
\widetilde\Omega_2(\tilde s)=\Big\{\widetilde v_*\in\mathbb R^3:
|\widetilde v_*|\le 2M_0,\
|\widetilde V(\tilde s)-\widetilde v_*|\ge \frac{1}{M_0} \Big\}.
\]

On the other hand, recalling \eqref{char-h}--\eqref{char-h-int}, we obtain
\begin{equation}\label{wideX}
\widetilde X(\tilde s)=X(s)-T(s,\tilde s)v_*+\mathcal R_\phi(\tilde s),
\end{equation}
where $X(s)$ is given by \eqref{char-h-int},
\[
T(s,\tilde s)
:=
e^{\beta s}
\begin{pmatrix}
s-\tilde s & \frac{\alpha}{2}(s-\tilde s)^2 & 0\\
0 & s-\tilde s & 0\\
0 & 0 & s-\tilde s
\end{pmatrix},
\]
and
\[
\mathcal R_\phi(\tilde s)
:=
\begin{pmatrix}
\int_{\tilde s}^s (\tau-\tilde s)\partial_{x_1}\phi(\tau,\widetilde X(\tau))\,d\tau
+\frac{\alpha}{2}\int_{\tilde s}^s (\tau-\tilde s)^2\partial_{x_2}\phi(\tau,\widetilde X(\tau))\,d\tau
\\[0.25cm]
\int_{\tilde s}^s (\tau-\tilde s)\partial_{x_2}\phi(\tau,\widetilde X(\tau))\,d\tau
\\[0.25cm]
\int_{\tilde s}^s (\tau-\tilde s)\partial_{x_3}\phi(\tau,\widetilde X(\tau))\,d\tau
\end{pmatrix}.
\]

\begin{lemma}\label{lem-cov-3d}
For any fixed $s\in[0,t-\varepsilon]$ and $\tilde s\in[0,s-\varepsilon]$, the change of variables
\[
y=\widetilde X(\tilde s;s,X(s),v_*)
\]
is well-defined on its image, and
\begin{equation}\label{jac-det-final}
\frac12e^{3\beta s}(s-\tilde s)^3
\le
\left|\det\frac{\partial \widetilde X(\tilde s)}{\partial v_*}\right|
\le
\frac32e^{3\beta s}(s-\tilde s)^3.
\end{equation}
Consequently,
\begin{equation}\label{cov-jac}
dv_*
=
\left|
\det \frac{\partial \widetilde X(\tilde s)}{\partial v_*}
\right|^{-1}dy
\le
C e^{-3\beta s}(s-\tilde s)^{-3}\,dy.
\end{equation}
\end{lemma}

\begin{proof}
Let
\[
J(s_1):=\frac{\partial \widetilde X(s_1)}{\partial v_*}\in\mathbb R^{3\times3},
\quad 0\le s_1<s.
\]
Differentiating \eqref{wideX} with respect to $v_*$, we obtain
\[
J(s_1)=-T(s,s_1)+T_\phi(s_1),
\]
where
\[
T_\phi(s_1)
=
\begin{pmatrix}
\int_{s_1}^s(\tau-s_1)\nabla_x\partial_{x_1}\phi(\tau,\widetilde X(\tau))J(\tau)\,d\tau
+
\frac{\alpha}{2}\int_{s_1}^s(\tau-s_1)^2
\nabla_x\partial_{x_2}\phi(\tau,\widetilde X(\tau))J(\tau)\,d\tau
\\[0.25cm]
\int_{s_1}^s(\tau-s_1)\nabla_x\partial_{x_2}\phi(\tau,\widetilde X(\tau))J(\tau)\,d\tau
\\[0.25cm]
\int_{s_1}^s(\tau-s_1)\nabla_x\partial_{x_3}\phi(\tau,\widetilde X(\tau))J(\tau)\,d\tau
\end{pmatrix}.
\]
We first claim that
\begin{equation}\label{jac-det-relative}
\frac12|\det T(s,\tilde s)|
\le
\left|\det\frac{\partial\widetilde X(\tilde s)}{\partial v_*}\right|
\le
\frac32|\det T(s,\tilde s)|.
\end{equation}
Since
\[
\det T(s,\tilde s)=e^{3\beta s}(s-\tilde s)^3,
\]
\eqref{jac-det-relative} immediately implies \eqref{jac-det-final}.

It remains to prove \eqref{jac-det-relative}. To this end, define
\[
\begin{aligned}
Z(s_1)
&:=-J(s_1)T(s,s_1)^{-1}  \\
&=I_3-T_\phi(s_1)T(s,s_1)^{-1}
=:I_3+\mathcal T_\phi[Z](s_1),
\quad 0\le s_1<s.
\end{aligned}
\]
More explicitly, using $J(\tau)=-Z(\tau)T(s,\tau)$, the rows of
$\mathcal T_\phi[Z](s_1)$ are given by
\[
\left\{
\begin{aligned}
\big(\mathcal T_\phi[Z](s_1)\big)_1
={}&
\int_{s_1}^s(\tau-s_1)
\nabla_x\partial_{x_1}\phi(\tau,\widetilde X(\tau))
Z(\tau)T(s,\tau)T(s,s_1)^{-1}\,d\tau
\\
&+
\frac{\alpha}{2}\int_{s_1}^s(\tau-s_1)^2
\nabla_x\partial_{x_2}\phi(\tau,\widetilde X(\tau))
Z(\tau)T(s,\tau)T(s,s_1)^{-1}\,d\tau,
\\
\big(\mathcal T_\phi[Z](s_1)\big)_i
={}&
\int_{s_1}^s(\tau-s_1)
\nabla_x\partial_{x_i}\phi(\tau,\widetilde X(\tau))
Z(\tau)T(s,\tau)T(s,s_1)^{-1}\,d\tau,
\quad i=2,3.
\end{aligned}
\right.
\]
A direct computation gives
\[
T(s,\tau)T(s,s_1)^{-1}
=
\frac{s-\tau}{s-s_1}
\begin{pmatrix}
1 & -\frac{\alpha}{2}(\tau-s_1) & 0\\
0 & 1 & 0\\
0 & 0 & 1
\end{pmatrix},
\]
and hence
\begin{equation}\label{est-normalized-T}
\big\|T(s,\tau)T(s,s_1)^{-1}\big\|_{\max}
\le
C\frac{s-\tau}{s-s_1}\big[1+\alpha(\tau-s_1)\big].
\end{equation}
By Lemma~\ref{lemma-poisson-3D} and the a priori assumption \eqref{assa}, we have
\[
\|\nabla_x^2\phi(\tau)\|_{L_x^\infty}\le C\eta l^{-\frac32}e^{-\lambda\tau}.
\]
Therefore, using \eqref{est-normalized-T}, for $0\le s_1<s$,
\allowdisplaybreaks\begin{align}
\|\mathcal T_\phi[Z](s_1)\|_{\max}
&\le
C\eta l^{-\frac32}
\sup_{s_1\le\sigma<s}\|Z(\sigma)\|_{\max}
\int_{s_1}^s
\frac{(\tau-s_1)(s-\tau)}{s-s_1}
e^{-\lambda\tau}
\big[1+\alpha(\tau-s_1)\big]\,d\tau
\notag\\
&\le
C\eta l^{-\frac32}e^{-\lambda s_1}
\sup_{s_1\le\sigma<s}\|Z(\sigma)\|_{\max}
\notag\\
&\le
C\eta l^{-\frac32}
\sup_{0\le\sigma<s}\|Z(\sigma)\|_{\max}.
\label{est-Tphi-Z}
\end{align}
Taking the supremum over $0\le s_1<s$
in \eqref{est-Tphi-Z} yields
\[
\sup_{0\le s_1<s}\|Z(s_1)\|_{\max}
\le
1+
C\eta l^{-\frac32}
\sup_{0\le s_1<s}\|Z(s_1)\|_{\max}.
\]
Choosing $l\gg1$ and then $\eta>0$ sufficiently small so that
$C\eta l^{-\frac32}\le \frac14$, we get
\begin{equation}\label{est-Z-uniform}
\sup_{0\le s_1<s}\|Z(s_1)\|_{\max}\le2.
\end{equation}
Substituting \eqref{est-Z-uniform} into \eqref{est-Tphi-Z}, we obtain
\[
\|Z(s_1)-I_3\|_{\max}
=
\|\mathcal T_\phi[Z](s_1)\|_{\max}
\le
C\eta l^{-\frac32},
\quad 0\le s_1<s.
\]
Taking $l$ larger and $\eta$ smaller if necessary, the continuity of the
determinant at $I_3$ gives
\[
\frac12\le |\det Z(\tilde s)|\le \frac32,
\quad 0\le\tilde s\le s-\varepsilon,
\]
this implies \eqref{jac-det-relative}.

In particular,
\[
\det\frac{\partial\widetilde X(\tilde s)}{\partial v_*}\neq0.
\]
Moreover, since $Z(\tilde s)$ is uniformly close to $I_3$ and
$T(s,\tilde s)$ is invertible, the mean value formula implies that
the map
\[
v_*\mapsto \widetilde X(\tilde s;s,X(s),v_*)
\]
is injective on the relevant velocity support. Hence the standard
change-of-variables formula applies on its image, and
\eqref{jac-det-final} yields \eqref{cov-jac}. This completes the proof of Lemma~\ref{lem-cov-3d}.
\end{proof}

Substituting \eqref{jac-det-final} in Lemma~\ref{lem-cov-3d} into
\eqref{eh2-step3}, we further obtain
\allowdisplaybreaks\begin{align}
e^{\lambda_n^m t}|h_2^{m,n}|
&\le
C_l\sum_{|n'|\le |n|}\mathcal R_3^{m,n'}(t)
+
C_{l,M_0}
\sum_{|n'|\le |n|}
\sum_{|n''|\le |n'|}
\int_0^{t-\varepsilon}
e^{-(\frac{\nu_0}{2}-\lambda_n^m)(t-s)}
\int_0^{s-\varepsilon}
e^{-(\frac{\nu_0}{2}-\lambda_{n'}^m)(s-\tilde s)}
\notag\\
&\quad\times
\int_{|v_*|\le 2M_0}
\int_{|\widetilde v_*|\le 2M_0}
\big|e^{\lambda_{n''}^m\tilde s}
h_2^{m,n''}(\tilde s,\widetilde X(\tilde s),\widetilde v_*)\big|
\,d\widetilde v_*dv_*d\tilde sds
\notag\\
&\le
C_l\sum_{|n'|\le |n|}\mathcal R_3^{m,n'}(t)
+
C_{l,M_0}
\sum_{|n'|\le |n|}
\sum_{|n''|\le |n'|}
\int_0^{t-\varepsilon}
e^{-(\frac{\nu_0}{2}-\lambda_n^m)(t-s)}
\int_0^{s-\varepsilon}
e^{-(\frac{\nu_0}{2}-\lambda_{n'}^m)(s-\tilde s)}
\notag\\
&\quad\times
e^{-3\beta s}(s-\tilde s)^{-3}
\int_{|\widetilde v_*|\le 2M_0}
\int_{E_{s,\tilde s}}
\big|e^{\lambda_{n''}^m\tilde s}
\partial_{n''}^m g_2(\tilde s,y,\widetilde v_*)\big|
\,dyd\widetilde v_*d\tilde sds
\notag\\
&\le
C_l\sum_{|n'|\le |n|}\mathcal R_3^{m,n'}(t)
+
C_{l,M_0}
\sum_{|n'|\le |n|}
\sum_{|n''|\le |n'|}
\int_0^{t-\varepsilon}
e^{-(\frac{\nu_0}{2}-\lambda_n^m)(t-s)}
\int_0^{s-\varepsilon}
e^{-(\frac{\nu_0}{2}-\lambda_{n'}^m)(s-\tilde s)}
\notag\\
&\quad\times
e^{-3\beta s}(s-\tilde s)^{-3}
\Big[
1
+C_{M_0}e^{\beta s}(s-\tilde s)
+C_{M_0}e^{\beta s}(s-\tilde s)^2
+C_{M_0}(s-\tilde s)^3
\Big]
\notag\\
&\quad\times
\Big[
1+C_{M_0}e^{\beta s}(s-\tilde s)
+C_{M_0}(s-\tilde s)^2
\Big]^2
\big\|e^{\lambda_{n''}^m\tilde s}
\partial_{n''}^m g_2(\tilde s)\big\|_{L^2_{x,v}}
\,d\tilde sds
\notag\\
&\le
C_l\sum_{|n'|\le |n|}\mathcal R_3^{m,n'}(t)
+
C_{l,M_0}
\sum_{|n'|\le |n|}
\sum_{|n''|\le |n'|}
\sup_{0\le \tau\le t}
e^{\lambda_{n''}^m\tau}
\big\|\partial_{n''}^m g_2(\tau)\big\|_{L^2_{x,v}}
\notag\\
&\le
C_l\sum_{|n'|\le |n|}\mathcal R_3^{m,n'}(t)
+
C_{l,M_0}
\sum_{|n'|\le |n|}
\sup_{0\le s\le t}
e^{\lambda_{n'}^m s}
\big\|\partial_{n'}^m g_2(s)\big\|_{L^2_{x,v}}.
\label{est-eh2mn-Linf-final}
\end{align}
Here
\[
E_{s,\tilde s}
:=
\Big\{
\widetilde X(\tilde s;s,X(s),v_*):
|v_*|\le2M_0
\Big\}
\subset\mathbb R^3
\]
denotes the image of the change of variables in Lemma~\ref{lem-cov-3d}. We also used the bound
\[
\mathbf 1_{\{|V(s)|\le M_0\}}
K_q\big(V(s),v_*\big)
\frac{w_l(V(s))}{w_l(v_*)}
\mathbf 1_{\Omega_2(s)}
\mathbf 1_{\{|\widetilde V(\tilde s)|\le M_0\}}
K_q\big(\widetilde V(\tilde s),\widetilde v_*\big)
\frac{w_l(\widetilde V(\tilde s))}{w_l(\widetilde v_*)}
\mathbf 1_{\widetilde\Omega_2(\tilde s)}
\le
C_{l,M_0}.
\]
\noindent{\textbf{\underline{Step 3. Summation with Hierarchy Coefficients.}}}
Let $0<\kappa\ll1$. Substituting \eqref{R3mn-def} into \eqref{est-eh2mn-Linf-final}, reducing repeated sums over the same index set to a single sum up to a constant, and multiplying by $\kappa^{|n|}$, we obtain, for any fixed multi-indices $m,n\in\mathbb N^3$ satisfying $|m|+|n|\le N$,
\allowdisplaybreaks\begin{align}
\kappa^{|n|}e^{\lambda_n^m t}|h_2^{m,n}|
&\le
C_{l,M_0,\varepsilon_1}\kappa^{|n|}
\sum_{|n'|\le |n|}
\sup_{0\le s\le t}
e^{\lambda_{n'}^m s}
\big\|\partial_{n'}^m g_2(s)\big\|_{L^2_{x,v}}
+
C\kappa^{|n|}
\sum_{i=1}^3\sum_{|n'|\le |n|}
\mathbf 1_{\{n'_i>0\}}n'_i
\mathfrak H_2^{m+e_i,n'-e_i}(t)
\notag\\
&\quad{}+
\mathbf 1_{\{|m|=0\}}C_l\kappa^{|n|}
\sum_{|n'|\le |n|}
\Big[
\mathfrak H_1^{0,n'}(t)
+
\big(\frac{1}{M_0}+\varepsilon_1+\eta+\alpha\big)
\mathfrak H_2^{0,n'}(t)
\Big]
\notag\\
&
\quad{}+
\mathbf 1_{\{|m|>0\}}C_l\kappa^{|n|}
\sum_{0<m'\le m}\sum_{|n'|\le |n|}
\Big[
\mathfrak H_1^{m',n'}(t)
+
\big(\frac{1}{M_0}+\varepsilon_1+\eta+\alpha\big)
\mathfrak H_2^{m',n'}(t)
\Big].
\label{eh2-final}
\end{align}
Similarly, combining the estimates for $ e^{\lambda_n^m t}\sum_{j=0}^8 |\mathcal I_j^{(1)}|$ and multiplying by $\kappa^{|n|}$, we obtain
\allowdisplaybreaks\begin{align}
\kappa^{|n|}e^{\lambda_n^m t}|h_1^{m,n}|
&\le
\kappa^{|n|}
\|w_l\partial_n^m \tilde f_0\|_{L^\infty_{x,v}}
+
\frac{C}{l}\kappa^{|n|}
\sum_{|n'|\le |n|}
\mathfrak H_1^{m,n'}(t)
+
C\kappa^{|n|}
\sum_{i=1}^3\mathbf 1_{\{n_i>0\}}n_i
\mathfrak H_1^{m+e_i,n-e_i}(t)
\notag\\
&
\quad{}+
\mathbf 1_{\{|m|=0\}}(C\eta+C_l \alpha)\kappa^{|n|}
\sum_{|n'|\le |n|}
\Big(
\mathfrak H_1^{0,n'}(t)
+
\mathfrak H_2^{0,n'}(t)
\Big)
\notag\\
&
\quad{}+
\mathbf 1_{\{|m|>0\}}(C\eta+C_l \alpha)\kappa^{|n|}
\sum_{0<m'\le m}\sum_{|n'|\le |n|}
\Big(
\mathfrak H_1^{m',n'}(t)
+
\mathfrak H_2^{m',n'}(t)
\Big).
\label{eh1-final}
\end{align}

We now use \eqref{eh2-final}--\eqref{eh1-final} and sum over all multi-indices
$m,n\in\mathbb N^3$ satisfying $|m|>0$ and $|m|+|n|\le N$. This gives
\allowdisplaybreaks\begin{align}
\sum_{\substack{|m|+|n|\le N\\ |m|>0}}
\kappa^{|n|}\mathfrak H_2^{m,n}(t)
&\le
C_{l,M_0,\varepsilon_1}
\sum_{\substack{|m|+|n|\le N\\ |m|>0}}
\kappa^{|n|}
\sup_{0\le s\le t}
e^{\lambda_n^m s}
\big\|\partial_n^m g_2(s)\big\|_{L^2_{x,v}}
+
C\kappa
\sum_{\substack{|m|+|n|\le N\\ |m|>0}}
\kappa^{|n|}
\mathfrak H_2^{m,n}(t)
\notag\\
&\quad
{}+
C_l
\sum_{\substack{|m|+|n|\le N\\ |m|>0}}
\kappa^{|n|}
\mathfrak H_1^{m,n}(t)
+
C_l
\Big(\frac{1}{M_0}+\varepsilon_1+\eta+\alpha\Big)
\sum_{\substack{|m|+|n|\le N\\ |m|>0}}
\kappa^{|n|}
\mathfrak H_2^{m,n}(t),
\label{eh2-positive-sum}
\end{align}
and
\allowdisplaybreaks\begin{align}
\sum_{\substack{|m|+|n|\le N\\ |m|>0}}
\kappa^{|n|}\mathfrak H_1^{m,n}(t)
&\le{}
\sum_{\substack{|m|+|n|\le N\\ |m|>0}}
\kappa^{|n|}
\big\|w_l\partial_n^m \tilde f_0\big\|_{L^\infty_{x,v}}
+
\Big(\frac{C}{l}+C\kappa+C\eta+C_l\alpha\Big)
\sum_{\substack{|m|+|n|\le N\\ |m|>0}}
\kappa^{|n|}
\mathfrak H_1^{m,n}(t)
\notag\\
&\quad
+
(C\eta+C_l\alpha)
\sum_{\substack{|m|+|n|\le N\\ |m|>0}}
\kappa^{|n|}
\mathfrak H_2^{m,n}(t).
\label{eh1-positive-sum}
\end{align}
For any fixed sufficiently large $l>0$, we first choose $M_0>0$ sufficiently large and then
$\varepsilon_1>0$ sufficiently small. Next, taking
$\kappa,\eta,\alpha>0$ sufficiently small and absorbing the small terms, we
obtain \eqref{est-positive-sumh1}--\eqref{est-positive-sumh2}. Taking $n=0$ in \eqref{eh1-final} and \eqref{eh2-final}, summing over
$0<|m|\le N$, and using $\lambda_0^m=\lambda$, we immediately obtain
\eqref{est-positive-n0-sumh1}--\eqref{est-positive-n0-sumh2}.

For $m=n=0$, taking $m=n=0$ in
\eqref{eh2-final}--\eqref{eh1-final} and using the same choice of
parameters as above, we absorb the small terms and obtain
\eqref{est-zero-sumh1}--\eqref{est-zero-sumh2}.

We next consider $m=0$ and $0<|n|\leq N$. Summing
\eqref{eh2-final}--\eqref{eh1-final} over $0<|n|\leq N$, combining
the resulting estimates with
\eqref{est-zero-sumh1}--\eqref{est-zero-sumh2}, and using the same
choice of parameters as above, we obtain
\eqref{est-v-sumh1}--\eqref{est-v-sumh2}. This completes the proof of
Lemma~\ref{lem-Linfty-sum-h1h2}.
\end{proof}

\section{\texorpdfstring{$L^2$}{L2} Estimates}\label{sec4}

In this section, we establish the $L^2$ estimates needed to control the
remaining $L^2_{x,v}$-terms in Lemma~\ref{lem-Linfty-sum-h1h2}. Recall Caflisch's decomposition $\mu^{\frac12}\tilde g=g_1+\mu^{\frac12}g_2$. For nonzero-order spatial derivatives, we combine the macroscopic dissipation
estimates for the full perturbation $\tilde g$ with the microscopic estimates
for $g_2$. The zero-order estimates are treated separately through the
corresponding zero-frequency system
\eqref{zero-ODE-matrix}--\eqref{def-Rt}.

We begin by introducing the macroscopic variables associated with the full
perturbation $\tilde g$. Recall that $\mathbf P_0$ denotes the projection
onto $\ker L$ and $\mathbf P_1=I-\mathbf P_0$. We decompose
\begin{equation}\label{def-abc}
\left\{
\begin{aligned}
\tilde g&=\mathbf P_0\tilde g+\mathbf P_1\tilde g,\\
\mathbf P_0\tilde g
&=\Big[a+\mathbf b\cdot v+\frac{1}{\sqrt{6}}c(|v|^2-3)\Big]\mu^{\frac{1}{2}},
\quad \mathbf b=(b_1,b_2,b_3),\\
a&=\langle \tilde g,\chi_0\rangle_v,\quad
b_i=\langle \tilde g,\chi_i\rangle_v\quad (1\le i\le3),\quad
c=\langle \tilde g,\chi_4\rangle_v.
\end{aligned}
\right.
\end{equation}
We also introduce
\begin{equation}\label{def-Aij-Bi}
\left\{
\begin{aligned}
A_{ij}(v)
&:=\mathbf P_1(v_i v_j\chi_0)
=\Big(v_i v_j-\frac13\delta_{ij}|v|^2\Big)\mu^{\frac12},
\\
B_i(v)
&:=\mathbf P_1(v_i\chi_4)
=\frac{|v|^2-5}{\sqrt6}v_i\mu^{\frac12},
\end{aligned}
\right.
\quad 1\le i,j\le3.
\end{equation}

\subsection{Nonzero-Order Spatial Derivative Estimates}

We begin with the macroscopic dissipation estimates. With the above notation, set
\begin{equation}\label{def-dij-qi}
d_{ij}(t,x)=\langle \mathbf P_1\tilde g,A_{ij}\rangle_v,\quad
q_i(t,x)=\langle \mathbf P_1\tilde g,B_i\rangle_v, \quad i,j=1,2,3.
\end{equation}
Taking the $L^2(\mathbb R_v^3)$ inner products of \eqref{tildeg-eq} with $\chi_i$, $0\le i\le4$, we obtain the macroscopic system
\begin{equation}\label{eq-macro-abc}
\left\{
\begin{aligned}
&\partial_t a+e^{\beta t}{\rm div}_x\mathbf b=0,\\
&\partial_t b_i
+e^{\beta t}\partial_{x_i}\Big(a+\sqrt{\frac{2}{3}}c\Big)
+\beta b_i+\alpha\delta_{i1}b_2
-e^{-\beta t}\partial_{x_i}\phi
=
e^{-\beta t}a\partial_{x_i}\phi
-e^{\beta t}\sum_{j=1}^3\partial_{x_j}d_{ij},
\quad i=1,2,3,\\
&\partial_t c
+\sqrt{\frac{2}{3}}e^{\beta t}{\rm div}_x\mathbf b
+\sqrt{6}\beta a+2\beta c
=
\sqrt{\frac{2}{3}}e^{-\beta t}\nabla_x\phi\cdot\mathbf b
-\sqrt{\frac{2}{3}}\alpha d_{12}
-e^{\beta t}\sum_{i=1}^3\partial_{x_i}q_i.
\end{aligned}
\right.
\end{equation}
Here we used the identity
\begin{equation}\label{G1-chi}
\langle G_1,\chi_i\rangle_v=0,\quad 0\le i\le 4.
\end{equation}
Noting that $\Delta_x\phi(t,x)=a(t,x)$, applying $P_0^x$ to \eqref{eq-macro-abc}, we obtain
\[
\left\{
\begin{aligned}
&\partial_t P_0^x a=0,\\
&\partial_t P_0^x b_1+\beta P_0^x b_1+\alpha P_0^x b_2=0,\\
&\partial_t P_0^x b_i+\beta P_0^x b_i=0,\quad i=2,3,
\end{aligned}
\right.
\]
which, together with the initial condition \eqref{initial-con}, yields
\begin{equation}\label{int-ab}
P_0^x a(t)=0,\quad P_0^x \mathbf b(t)=0.
\end{equation}

Applying $ \mathbf P_1 $ to \eqref{tildeg-eq}, we obtain
\begin{equation}\label{micro-eq}
\partial_t\mathbf P_1\tilde g+L\mathbf P_1\tilde g
=-e^{\beta t}\mathbf P_1(v\cdot\nabla_x\mathbf P_0\tilde g)
+\mathbf P_1\mathcal R,
\end{equation}
where
\allowdisplaybreaks\begin{align}
\mathcal R
=&{}
-e^{\beta t}v\cdot\nabla_x\mathbf P_1\tilde g
+\beta\nabla_v\cdot(v\tilde g)
-\frac{\beta}{2}|v|^2\tilde g
+\alpha v_2\partial_{v_1}\tilde g
-\frac{\alpha}{2}v_1v_2\tilde g
-e^{-\beta t}\nabla_x\phi\cdot\nabla_v\tilde g
\notag\\
&{}+\frac{1}{2}e^{-\beta t}(v\cdot\nabla_x\phi)\tilde g
+\alpha\Gamma(G_1,\tilde g)
+\alpha\Gamma(\tilde g,G_1)
+\Gamma(\tilde g,\tilde g)
-\alpha e^{-\beta t}\nabla_x\phi\cdot\Big(\nabla_vG_1-\frac{v}{2}G_1\Big).
\label{micro-R}
\end{align}
Moreover, by the definitions of $A_{ij}$ and $B_i$, we have
\begin{equation}\notag
\mathbf P_1(v\cdot\nabla_x\mathbf P_0\tilde g)
=
\sum_{i,j=1}^3(\partial_{x_i}b_j)A_{ij}
+
\sum_{i=1}^3(\partial_{x_i}c)B_i.
\end{equation}
For later use, set
\begin{equation}\label{def-R-A-B}
\mathcal R_{ij}^{A}
=
\langle \mathbf P_1\mathcal R,A_{ij}\rangle_v
=
\langle \mathcal R,A_{ij}\rangle_v,
\quad
\mathcal R_i^{B}
=
\langle \mathbf P_1\mathcal R,B_i\rangle_v
=
\langle \mathcal R,B_i\rangle_v.
\end{equation}
Taking the $L^2(\mathbb R_v^3)$ inner products of \eqref{micro-eq}
with $A_{ij}$ and $B_i$, respectively, and using the identities
\begin{equation}\label{LAij}
LA_{ij}=2b_0A_{ij},
\quad
LB_i=\frac{4}{3}b_0B_i,
\end{equation}
which follow from \eqref{L-moments} in Lemma~\ref{lem-Lmoments}, we obtain
\begin{equation}\label{dij-qi-rep}
\left\{
\begin{aligned}
d_{ij}
&=
-\frac{e^{\beta t}}{2b_0}
\Big(
\partial_{x_i}b_j+\partial_{x_j}b_i
-\frac{2}{3}\delta_{ij}{\rm div}_x\mathbf b
\Big)
-\frac{1}{2b_0}\partial_t d_{ij}
+\frac{1}{2b_0}\mathcal R_{ij}^{A},
\\
q_i
&=
-\frac{5}{4b_0}e^{\beta t}\partial_{x_i}c
-\frac{3}{4b_0}\partial_t q_i
+\frac{3}{4b_0}\mathcal R_i^{B}.
\end{aligned}
\right.
\end{equation}
Consequently, substituting \eqref{dij-qi-rep} into \eqref{eq-macro-abc}, the macroscopic system takes the form
\begin{equation}\label{eq-macro-abc-ns}
\left\{
\begin{aligned}
&\partial_t a+e^{\beta t}{\rm div}_x\mathbf b=0,\\
&\partial_t b_i
+e^{\beta t}\partial_{x_i}\Big(a+\sqrt{\frac{2}{3}}c\Big)
-\frac{e^{2\beta t}}{2b_0}
\Big(
\Delta_x b_i+\frac13\partial_{x_i}{\rm div}_x\mathbf b
\Big)
+\beta b_i+\alpha\delta_{i1}b_2
-e^{-\beta t}\partial_{x_i}\phi
\\
&\quad =
e^{-\beta t}a\partial_{x_i}\phi
+\frac{e^{\beta t}}{2b_0}\sum_{j=1}^3\partial_{x_j}\partial_t d_{ij}
-\frac{e^{\beta t}}{2b_0}\sum_{j=1}^3\partial_{x_j}\mathcal R_{ij}^{A},
\quad i=1,2,3,\\
&\partial_t c
+\sqrt{\frac{2}{3}}e^{\beta t}{\rm div}_x\mathbf b
-\frac{5}{4b_0}e^{2\beta t}\Delta_x c
+\sqrt{6}\beta a+2\beta c
\\
&\quad =
\sqrt{\frac{2}{3}}e^{-\beta t}\nabla_x\phi\cdot\mathbf b
-\sqrt{\frac{2}{3}}\alpha d_{12}
+\frac{3e^{\beta t}}{4b_0}\sum_{i=1}^3\partial_{x_i}\partial_t q_i
-\frac{3e^{\beta t}}{4b_0}\sum_{i=1}^3\partial_{x_i}\mathcal R_i^{B}.
\end{aligned}
\right.
\end{equation}

The corresponding  macroscopic dissipation estimate is stated as follows.

\begin{lemma}\label{lem-abc-dis}
Let $[g_1,g_2]$ be the solution to \eqref{g1-eq}--\eqref{g2-eq}
satisfying the a priori assumption \eqref{assa}. Let $N\geq1$ and $l>3$.
Then there exist a sufficiently large constant $M_1>0$ and a constant
$\lambda_0>0$ such that
\allowdisplaybreaks\begin{align}
&e^{-2\beta t}\frac{d}{dt}\mathcal E_N^{(abc)}(t)
+\lambda_0\sum_{|m|= N-1}
\|\nabla_x\partial_x^m[a,\mathbf b,c]\|_{L_x^2}^2
+e^{-2\beta t}\sum_{|m|=N-1}
\|\partial_x^m a\|_{L_x^2}^2
\notag\\
&\le
C\sum_{|m|=N-1}
\|\nabla_x\partial_x^m\mathbf P_1\tilde g\|_{L^2_{x,v}}^2
+C(\eta^2+\alpha^2)\sum_{0<|m|\le N}
\Big(
\|h_1^{m,0}\|_{L^\infty_{x,v}}^2
+\|h_2^{m,0}\|_{L^\infty_{x,v}}^2
\Big),
\label{abc-dis}
\end{align}
where
\allowdisplaybreaks\begin{align}
\mathcal E_N^{(abc)}(t)
=&
\sum_{|m|=N-1}
\Bigg\{
M_1\bigg[
\frac{1}{2}\Big(
\|\partial_x^m[a,\mathbf b]\|_{L_x^2}^2
+\|P_{\neq0}^x\partial_x^m c\|_{L_x^2}^2
+e^{-2\beta t}\|\partial_x^m\nabla_x\phi\|_{L_x^2}^2
\Big)
\notag\\
&-\frac{3e^{\beta t}}{4b_0}\sum_{i=1}^3
\big\langle
\partial_{x_i}\partial_x^m q_i,
P_{\neq0}^x\partial_x^m c
\big\rangle_x
-\frac{e^{\beta t}}{2b_0}\sum_{i,j=1}^3
\big\langle
\partial_{x_j}\partial_x^m d_{ij},
\partial_x^m b_i
\big\rangle_x
\bigg]
+e^{\beta t}
\big\langle
\partial_x^m\mathbf b,
\nabla_x\partial_x^m a
\big\rangle_x
\Bigg\}.
\label{Eabc-def}
\end{align}
Here $h_1^{m,0}$ and $h_2^{m,0}$ are defined in
\eqref{eq-def-h12mn}.
\end{lemma}

\begin{proof}

\noindent\underline{\textbf{Step 1. Dissipation of $c$.}}
Let $m\in\mathbb N^3$ satisfy $|m|\le N-1$. The projection
$P_{\neq0}^x$ is used in the $c$-estimate to remove the zero-frequency mode
$P_0^x c$, which cannot be controlled in the positive spatial derivative
$\nabla_x c$. Applying $P_{\neq0}^x\partial_x^m$ to the third
equation of \eqref{eq-macro-abc-ns} gives
\allowdisplaybreaks\begin{align}
&\partial_tP_{\neq0}^x\partial_x^m c
+\sqrt{\frac{2}{3}}e^{\beta t}{\rm div}_x\partial_x^m\mathbf b
-\frac{5}{4b_0}e^{2\beta t}\Delta_x\partial_x^m c
+\sqrt6\beta\partial_x^m a
+2\beta P_{\neq0}^x\partial_x^m c
\notag\\
& =
\sqrt{\frac{2}{3}}e^{-\beta t}
P_{\neq0}^x\partial_x^m(\nabla_x\phi\cdot\mathbf b)
-\sqrt{\frac{2}{3}}\alpha P_{\neq0}^x\partial_x^m d_{12}
+\frac{3e^{\beta t}}{4b_0}\sum_{i=1}^3
\partial_{x_i}\partial_x^m\partial_t q_i
-\frac{3e^{\beta t}}{4b_0}\sum_{i=1}^3
\partial_{x_i}\partial_x^m\mathcal R_i^B .
\label{eq-Pneq0-c}
\end{align}
Testing \eqref{eq-Pneq0-c} against $P_{\neq0}^x\partial_x^m c$, and taking the
$L^2(\mathbb T_x^3)$ inner product, we obtain
\allowdisplaybreaks\begin{align}
\frac{d}{dt}&\bigg\{
\frac{1}{2}\|P_{\neq0}^x\partial_x^m c\|_{L_x^2}^2
-\frac{3e^{\beta t}}{4b_0}\sum_{i=1}^3
\big\langle
\partial_{x_i}\partial_x^m q_i,
P_{\neq0}^x\partial_x^m c
\big\rangle_x
\bigg\}
+\frac{5}{4b_0}e^{2\beta t}
\|\nabla_x\partial_x^m c\|_{L_x^2}^2
\notag\\
&
+2\beta\|P_{\neq0}^x\partial_x^m c\|_{L_x^2}^2
+\sqrt{\frac{2}{3}}e^{\beta t}
\big\langle
{\rm div}_x\partial_x^m\mathbf b,
P_{\neq0}^x\partial_x^m c
\big\rangle_x
\notag\\
={}&
-\sqrt{6}\beta
\big\langle
\partial_x^m a,
P_{\neq0}^x\partial_x^m c
\big\rangle_x
+\sqrt{\frac{2}{3}}e^{-\beta t}
\big\langle
\partial_x^m(\nabla_x\phi\cdot\mathbf b),
P_{\neq0}^x\partial_x^m c
\big\rangle_x
\notag\\
&
-\sqrt{\frac{2}{3}}\alpha
\big\langle
P_{\neq0}^x\partial_x^m d_{12},
P_{\neq0}^x\partial_x^m c
\big\rangle_x
-\frac{3\beta e^{\beta t}}{4b_0}\sum_{i=1}^3
\big\langle
\partial_{x_i}\partial_x^m q_i,
P_{\neq0}^x\partial_x^m c
\big\rangle_x
\notag\\
&
-\frac{3e^{\beta t}}{4b_0}\sum_{i=1}^3
\big\langle
\partial_{x_i}\partial_x^m q_i,
\partial_tP_{\neq0}^x\partial_x^m c
\big\rangle_x
-\frac{3e^{\beta t}}{4b_0}\sum_{i=1}^3
\left\langle
\partial_{x_i}\partial_x^m\mathcal R_i^{B},
P_{\neq0}^x\partial_x^m c
\right\rangle_x
=:\sum_{i=1}^6\mathcal J_i^{c}.
\label{c-step}
\end{align}
 Here we have used
\begin{align}
\frac{3e^{\beta t}}{4b_0}\sum_{i=1}^3
\big\langle
\partial_{x_i}\partial_x^m\partial_t q_i,
P_{\neq0}^x\partial_x^m c
\big\rangle_x
={}&
\frac{d}{dt}\bigg\{
\frac{3e^{\beta t}}{4b_0}\sum_{i=1}^3
\big\langle
\partial_{x_i}\partial_x^m q_i,
P_{\neq0}^x\partial_x^m c
\big\rangle_x
\bigg\}
-\frac{3\beta e^{\beta t}}{4b_0}\sum_{i=1}^3
\big\langle
\partial_{x_i}\partial_x^m q_i,
P_{\neq0}^x\partial_x^m c
\big\rangle_x
\notag\\
&-\frac{3e^{\beta t}}{4b_0}\sum_{i=1}^3
\big\langle
\partial_{x_i}\partial_x^m q_i,
\partial_tP_{\neq0}^x\partial_x^m c
\big\rangle_x .\notag
\end{align}

For $\mathcal J_1^c,\mathcal J_3^c$, and $\mathcal J_4^c$, using
\eqref{def-dij-qi}, the Poincar\'e inequality for the nonzero-frequency modes
$P_{\neq0}^x\partial_x^m c$ and $P_{\neq0}^x\partial_x^m d_{12}$, and the fact
$\beta=O(\alpha^2)$, we have
\allowdisplaybreaks\begin{align}
|\mathcal J_1^c|+|\mathcal J_3^c|+|\mathcal J_4^c|
&\le C\beta\|\partial_x^m a\|_{L_x^2}\|P_{\neq0}^x\partial_x^m c\|_{L_x^2}
+C\alpha\|P_{\neq0}^x\partial_x^m d_{12}\|_{L_x^2}\|P_{\neq0}^x\partial_x^m c\|_{L_x^2}
\notag\\
&\quad+C\beta e^{\beta t}\sum_{i=1}^3\|\partial_{x_i}\partial_x^m q_i\|_{L_x^2}
\|P_{\neq0}^x\partial_x^m c\|_{L_x^2}
\notag\\
&\le
C\beta\|\nabla_x\partial_x^m a\|_{L_x^2}\|\nabla_x\partial_x^m c\|_{L_x^2}
+C\alpha\|\nabla_x\partial_x^m\mathbf P_1\tilde g\|_{L^2_{x,v}}
\|\nabla_x\partial_x^m c\|_{L_x^2}
\notag\\
&\quad
+C\beta e^{\beta t}\|\nabla_x\partial_x^m\mathbf P_1\tilde g\|_{L^2_{x,v}}
\|\nabla_x\partial_x^m c\|_{L_x^2}
\notag\\
&\le
C\alpha e^{2\beta t}\Big(\|\nabla_x\partial_x^m[a,\mathbf b,c]\|_{L_x^2}^2
+\|\nabla_x\partial_x^m\mathbf P_1\tilde g\|_{L^2_{x,v}}^2\Big).
\label{Jc134}
\end{align}

For $\mathcal J_2^c$, by the definition of $\mathbf b$, Caflisch's decomposition $\mu^{\frac12}\tilde g=g_1+\mu^{\frac12}g_2$, and the a priori
assumption \eqref{assa}, for any $|m'|\le |m|$, we have
\[
\begin{aligned}
\|\partial_x^{m'}\mathbf b\|_{L_x^\infty}
&= \left\|\int_{\mathbb R^3}v\big(
\partial_x^{m'}g_1+\mu^{\frac12}\partial_x^{m'}g_2
\big)\,dv\right\|_{L_x^\infty}
\\
&\le
C
\bigg(\int_{\mathbb R^3}|v|\,w_l(v)^{-1}\,dv\bigg)
\Big(
\|w_l\partial_x^{m'}g_1\|_{L^\infty_{x,v}}
+
\|w_l\partial_x^{m'}g_2\|_{L^\infty_{x,v}}
\Big)
\\
&\le
C\eta,\quad l>3.
\end{aligned}
\]
Hence
\allowdisplaybreaks\begin{align}
|\mathcal J_2^c|
&\le
C e^{-\beta t}
\|\partial_x^m(\nabla_x\phi\cdot\mathbf b)\|_{L_x^2}
\|P_{\neq0}^x\partial_x^m c\|_{L_x^2}
\notag\\
&\le
C\eta e^{-\beta t}
\sum_{|m'|\le |m|}
\|\partial_x^{m'}\nabla_x\phi\|_{L_x^2}
\|\nabla_x\partial_x^m c\|_{L_x^2}
\notag\\
&\le
C\eta
\Big(\sum_{|m'|\le |m|}\|\partial_x^{m'}\nabla_x\phi\|_{L_x^2}^2
+\|\nabla_x\partial_x^m c\|_{L_x^2}^2\Big).
\label{Jc2}
\end{align}

To estimate $\mathcal J_5^c$, we substitute the third equation of
\eqref{eq-macro-abc} into $\partial_tP_{\neq0}^x\partial_x^m c$. From $\eqref{int-ab}_3$, we have
\[
\mathcal J_5^c=\sum_{\ell=1}^5\mathcal J_{5,\ell}^c,
\]
where
\begin{equation}\label{J5c-dec}
\left\{
\begin{aligned}
\mathcal J_{5,1}^c
&=
\frac{3}{4b_0}\sqrt{\frac{2}{3}}\,e^{2\beta t}
\sum_{i=1}^3
\big\langle
\partial_{x_i}\partial_x^m q_i,
{\rm div}_x\partial_x^m\mathbf b
\big\rangle_x,
\\
\mathcal J_{5,2}^c
&=
\frac{3\sqrt6\,\beta e^{\beta t}}{4b_0}
\sum_{i=1}^3
\big\langle
\partial_{x_i}\partial_x^m q_i,
\partial_x^m a
\big\rangle_x
+
\frac{3\beta e^{\beta t}}{2b_0}
\sum_{i=1}^3
\big\langle
\partial_{x_i}\partial_x^m q_i,
P_{\neq0}^x\partial_x^m c
\big\rangle_x,
\\
\mathcal J_{5,3}^c
&=
-\frac{3}{4b_0}\sqrt{\frac{2}{3}}
\sum_{i=1}^3
\big\langle
\partial_{x_i}\partial_x^m q_i,
P_{\neq0}^x\partial_x^m(\nabla_x\phi\cdot\mathbf b)
\big\rangle_x,
\\
\mathcal J_{5,4}^c
&=
\frac{3}{4b_0}\sqrt{\frac{2}{3}}\,\alpha e^{\beta t}
\sum_{i=1}^3
\big\langle
\partial_{x_i}\partial_x^m q_i,
P_{\neq0}^x\partial_x^m d_{12}
\big\rangle_x,
\\
\mathcal J_{5,5}^c
&=
\frac{3e^{2\beta t}}{4b_0}
\sum_{i,k=1}^3
\big\langle
\partial_{x_i}\partial_x^m q_i,
\partial_{x_k}\partial_x^m q_k
\big\rangle_x .
\end{aligned}
\right.
\end{equation}
Similarly to \eqref{Jc134} and \eqref{Jc2}, we obtain
\allowdisplaybreaks\begin{align}
|\mathcal J_{5,1}^c|+|\mathcal J_{5,2}^c|+|\mathcal J_{5,3}^c|
&\le
C e^{2\beta t}
\|\nabla_x\partial_x^m\mathbf P_1\tilde g\|_{L^2_{x,v}}
\|\nabla_x\partial_x^m\mathbf b\|_{L_x^2}
+C\beta e^{\beta t}
\|\nabla_x\partial_x^m\mathbf P_1\tilde g\|_{L^2_{x,v}}
\|\nabla_x\partial_x^m[a,c]\|_{L_x^2}
\notag\\
&\quad
+C\eta
\|\nabla_x\partial_x^m\mathbf P_1\tilde g\|_{L^2_{x,v}}
\sum_{|m'|\le |m|}
\|\partial_x^{m'}\nabla_x\phi\|_{L_x^2}
\notag\\
&\le
\varepsilon e^{2\beta t}
\|\nabla_x\partial_x^m\mathbf b\|_{L_x^2}^2
+
C\beta e^{2\beta t}
\|\nabla_x\partial_x^m[a,c]\|_{L_x^2}^2
\notag\\
&\quad
+
C_\varepsilon e^{2\beta t}
\|\nabla_x\partial_x^m\mathbf P_1\tilde g\|_{L^2_{x,v}}^2
+
C\eta^2
\sum_{|m'|\le |m|}
\|\partial_x^{m'}\nabla_x\phi\|_{L_x^2}^2 .
\label{Jc51-53}
\end{align}
For $\mathcal J_{5,4}^c$ and $\mathcal J_{5,5}^c$, the definition in \eqref{def-dij-qi} and the Poincar\'e inequality directly yield
\allowdisplaybreaks\begin{align}
|\mathcal J_{5,4}^c|+|\mathcal J_{5,5}^c|
&\le
C\alpha e^{\beta t}
\|\nabla_x\partial_x^m\mathbf P_1\tilde g\|_{L^2_{x,v}}
\|P_{\neq0}^x\partial_x^m d_{12}\|_{L_x^2}
+
Ce^{2\beta t}
\|\nabla_x\partial_x^m\mathbf P_1\tilde g\|_{L^2_{x,v}}^2
\notag\\
&\le
C\alpha e^{\beta t}
\|\nabla_x\partial_x^m\mathbf P_1\tilde g\|_{L^2_{x,v}}^2
+
Ce^{2\beta t}
\|\nabla_x\partial_x^m\mathbf P_1\tilde g\|_{L^2_{x,v}}^2
\notag\\
&\le
Ce^{2\beta t}
\|\nabla_x\partial_x^m\mathbf P_1\tilde g\|_{L^2_{x,v}}^2 .
\label{Jc54-55}
\end{align}
Combining \eqref{Jc51-53} and \eqref{Jc54-55}, we obtain
\begin{align}
|\mathcal J_5^c|
&\le
\varepsilon e^{2\beta t}
\|\nabla_x\partial_x^m\mathbf b\|_{L_x^2}^2
+
C\beta e^{2\beta t}
\|\nabla_x\partial_x^m[a,c]\|_{L_x^2}^2
\notag\\
&\quad
+
C_\varepsilon e^{2\beta t}
\|\nabla_x\partial_x^m\mathbf P_1\tilde g\|_{L^2_{x,v}}^2
+
C_\varepsilon\eta^2
\sum_{|m'|\le |m|}
\|\partial_x^{m'}\nabla_x\phi\|_{L_x^2}^2 .
\label{Jc5}
\end{align}

To estimate $\mathcal J_6^c$, we insert \eqref{micro-R} into
$\mathcal R_i^B$. We decompose
\[
\mathcal J_6^c=\sum_{\ell=1}^5\mathcal J_{6,\ell}^c,
\]
where
\begin{equation}\label{J6c-dec}
\left\{
\begin{aligned}
\mathcal J_{6,1}^c
&=
\frac{3e^{2\beta t}}{4b_0}
\sum_{i=1}^3
\big\langle
\partial_{x_i}\partial_x^m
\big\langle
v\cdot\nabla_x\mathbf P_1\tilde g,B_i
\big\rangle_v,
P_{\neq0}^x\partial_x^m c
\big\rangle_x,
\\
\mathcal J_{6,2}^c
&=
-\frac{3e^{\beta t}}{4b_0}
\sum_{i=1}^3
\big\langle
\partial_{x_i}\partial_x^m
\big\langle
\beta\nabla_v\cdot(v\tilde g)
-\frac{\beta}{2}|v|^2\tilde g
+\alpha v_2\partial_{v_1}\tilde g
-\frac{\alpha}{2}v_1v_2\tilde g,
B_i
\big\rangle_v,
P_{\neq0}^x\partial_x^m c
\big\rangle_x,
\\
\mathcal J_{6,3}^c
&=
-\frac{3}{4b_0}
\sum_{i=1}^3
\big\langle
\partial_{x_i}\partial_x^m
\big\langle
-\nabla_x\phi\cdot\nabla_v\tilde g
+\frac12(v\cdot\nabla_x\phi)\tilde g,
B_i
\big\rangle_v,
P_{\neq0}^x\partial_x^m c
\big\rangle_x,
\\
\mathcal J_{6,4}^c
&=
-\frac{3e^{\beta t}}{4b_0}
\sum_{i=1}^3
\big\langle
\partial_{x_i}\partial_x^m
\big\langle
\alpha\Gamma(G_1,\tilde g)
+\alpha\Gamma(\tilde g,G_1)
+\Gamma(\tilde g,\tilde g),
B_i
\big\rangle_v,
P_{\neq0}^x\partial_x^m c
\big\rangle_x,
\\
\mathcal J_{6,5}^c
&=
\frac{3\alpha}{4b_0}
\sum_{i=1}^3
\big\langle
\partial_{x_i}\partial_x^m
\big\langle
\nabla_x\phi\cdot
\big(\nabla_vG_1-\frac{v}{2}G_1\big),
B_i
\big\rangle_v,
P_{\neq0}^x\partial_x^m c
\big\rangle_x .
\end{aligned}
\right.
\end{equation}

For $\mathcal J_{6,1}^c$, integrating by parts in $x$, then applying the
Cauchy--Schwarz inequality first in $x$ and then in $v$, and using the
Maxwellian decay in $B_i$, we have
\begin{align}
|\mathcal J_{6,1}^c|
&\le
C e^{2\beta t}
\sum_{i=1}^3
\Big\|
\partial_x^m
\big\langle
v\cdot\nabla_x\mathbf P_1\tilde g,B_i
\big\rangle_v
\Big\|_{L_x^2}
\|\partial_{x_i}P_{\neq0}^x\partial_x^m c\|_{L_x^2}
\notag\\
&\le
\varepsilon e^{2\beta t}
\|\nabla_x\partial_x^m c\|_{L_x^2}^2
+
C_\varepsilon e^{2\beta t}
\|\nabla_x\partial_x^m\mathbf P_1\tilde g\|_{L^2_{x,v}}^2 .
\label{Jc61}
\end{align}
For $\mathcal J_{6,2}^c$, similarly as in \eqref{Jc61}, and additionally
integrating by parts in $v$ for the terms involving $\nabla_v$ and
$\partial_{v_1}$, we obtain
\begin{align}
|\mathcal J_{6,2}^c|
&\le
C(\alpha+\beta)e^{\beta t}
\|\nabla_x\partial_x^m\tilde g\|_{L^2_{x,v}}
\|P_{\neq0}^x\partial_x^m c\|_{L_x^2}
\notag\\
&\le
C\alpha e^{\beta t}
\Big(
\|\nabla_x\partial_x^m[a,\mathbf b,c]\|_{L_x^2}
+
\|\nabla_x\partial_x^m\mathbf P_1\tilde g\|_{L^2_{x,v}}
\Big)
\|\nabla_x\partial_x^m c\|_{L_x^2}
\notag\\
&\le
C\alpha e^{2\beta t}
\Big(
\|\nabla_x\partial_x^m[a,\mathbf b,c]\|_{L_x^2}^2
+
\|\nabla_x\partial_x^m\mathbf P_1\tilde g\|_{L^2_{x,v}}^2
\Big).
\label{Jc62}
\end{align}
For $\mathcal J_{6,3}^c$, using Lemma \ref{lemma-poisson-3D} and the a priori assumption \eqref{assa}, we obtain
\begin{align}
|\mathcal J_{6,3}^c|
&\le
C\eta
\Big(
\|\nabla_x\partial_x^m[a,\mathbf b,c]\|_{L_x^2}
+
\|\nabla_x\partial_x^m\mathbf P_1\tilde g\|_{L^2_{x,v}}
+
\sum_{|m'|\le |m|}
\|\partial_x^{m'}\nabla_x\phi\|_{L_x^2}
\Big)
\|\nabla_x\partial_x^m c\|_{L_x^2}
\notag\\
&\le
C\eta
\Big(
\|\nabla_x\partial_x^m[a,\mathbf b,c]\|_{L_x^2}^2
+
\|\nabla_x\partial_x^m\mathbf P_1\tilde g\|_{L^2_{x,v}}^2
+
\sum_{|m'|\le |m|}
\|\partial_x^{m'}\nabla_x\phi\|_{L_x^2}^2
\Big).
\label{Jc63}
\end{align}
For $\mathcal J_{6,4}^c$, one has
\[
\begin{aligned}
|\mathcal J_{6,4}^c|
&\le
C e^{\beta t}\sum_{i=1}^3\Big\|\partial_x^{m+e_i}
\big\langle\alpha\Gamma(G_1,\tilde g)+\alpha\Gamma(\tilde g,G_1)
+\Gamma(\tilde g,\tilde g),B_i\big\rangle_v\Big\|_{L_x^2}\|P_{\neq0}^x\partial_x^m c\|_{L_x^2}
\\
&\le
C e^{\beta t}
\sum_{i=1}^3
\Big\|
\partial_x^{m+e_i}
\big\langle
\alpha\Gamma(G_1,\tilde g)
+\alpha\Gamma(\tilde g,G_1)
+\Gamma(\tilde g,\tilde g),
B_i
\big\rangle_v
\Big\|_{L_x^\infty}
\|\nabla_x\partial_x^m c\|_{L_x^2}.
\end{aligned}
\]
We next estimate the collision moment. Recalling that
$\tilde g=\mu^{-\frac12}g_1+g_2$, we have
\[
\begin{aligned}
\partial_x^{m+e_i}
\big\langle \Gamma(G_1,\tilde g),B_i\big\rangle_v
&=
\int_{\mathbb R^3}
w_l Q\big(
\mu^{\frac12}G_1,
\partial_x^{m+e_i}g_1
+\mu^{\frac12}\partial_x^{m+e_i}g_2
\big)
w_l^{-1}\frac{|v|^2-5}{\sqrt6}v_i\,dv,
\\
\partial_x^{m+e_i}
\big\langle \Gamma(\tilde g,G_1),B_i\big\rangle_v
&=
\int_{\mathbb R^3}
w_l Q\big(
\partial_x^{m+e_i}g_1
+\mu^{\frac12}\partial_x^{m+e_i}g_2,
\mu^{\frac12}G_1
\big)
w_l^{-1}\frac{|v|^2-5}{\sqrt6}v_i\,dv .
\end{aligned}
\]
Thus, by \eqref{G1-bound} in Lemma \ref{lem-DL-G} and
\eqref{eq-Q-Linf} in Lemma \ref{lem-Q}, choosing $l>3$ so that
\[
w_l^{-1}\frac{|v|^2-5}{\sqrt6}v_i\in L_v^1,
\]
we obtain
\[
\begin{aligned}
&\alpha
\sum_{i=1}^3
\Big\|
\partial_x^{m+e_i}
\big\langle \Gamma(G_1,\tilde g)+\Gamma(\tilde g,G_1),B_i\big\rangle_v
\Big\|_{L_x^\infty}
\le
C\alpha
\sum_{i=1}^3
\Big(
\big\|w_l\partial_x^{m+e_i}g_1\big\|_{L^\infty_{x,v}}
+
\big\|w_l\partial_x^{m+e_i}g_2\big\|_{L^\infty_{x,v}}
\Big).
\end{aligned}
\]
For the quadratic term, we expand
\[
\begin{aligned}
&\quad\partial_x^{m+e_i}
\big\langle \Gamma(\tilde g,\tilde g),B_i\big\rangle_v
\\
&=
\sum_{m'\le m+e_i}C_{m+e_i,m'}
\int_{\mathbb R^3}
Q\big(
\partial_x^{m'}g_1+\mu^{\frac12}\partial_x^{m'}g_2,
\partial_x^{m+e_i-m'}g_1
+\mu^{\frac12}\partial_x^{m+e_i-m'}g_2
\big)
\frac{|v|^2-5}{\sqrt6}v_i\,dv .
\end{aligned}
\]
Applying \eqref{eq-Q-Linf} in Lemma \ref{lem-Q} term by term, and using the
a priori assumption \eqref{assa}, we get
\[
\sum_{i=1}^3
\Big\|
\partial_x^{m+e_i}
\big\langle \Gamma(\tilde g,\tilde g),B_i\big\rangle_v
\Big\|_{L_x^\infty}
\le
C\eta
\sum_{0<|m'|\le |m|+1}
\Big(
\big\|w_l\partial_x^{m'}g_1\big\|_{L^\infty_{x,v}}
+
\big\|w_l\partial_x^{m'}g_2\big\|_{L^\infty_{x,v}}
\Big).
\]
Consequently,
\allowdisplaybreaks\begin{align}
|\mathcal J_{6,4}^c|
&\le
C e^{\beta t}(\alpha+\eta)
\sum_{0<|m'|\le |m|+1}
\Big(
\big\|w_l\partial_x^{m'}g_1\big\|_{L^\infty_{x,v}}
+
\big\|w_l\partial_x^{m'}g_2\big\|_{L^\infty_{x,v}}
\Big)
\|\nabla_x\partial_x^m c\|_{L_x^2}
\notag\\
&\le
\varepsilon e^{2\beta t}
\|\nabla_x\partial_x^m c\|_{L_x^2}^2
+
C_\varepsilon(\alpha^2+\eta^2)
\sum_{0<|m'|\le |m|+1}
\Big(
\big\|w_l\partial_x^{m'}g_1\big\|_{L^\infty_{x,v}}^2
+
\big\|w_l\partial_x^{m'}g_2\big\|_{L^\infty_{x,v}}^2
\Big).
\label{Jc64}
\end{align}
For $\mathcal J_{6,5}^c$, noticing that
$
\mu^{\frac12}(\nabla_vG_1-\frac{v}{2}G_1)=\nabla_v(\mu^{\frac12}G_1),
$
also by \eqref{G1-bound} in Lemma \ref{lem-DL-G} and $l>3$, we obtain
\allowdisplaybreaks\begin{align}
|\mathcal J_{6,5}^c|
&\le
C\alpha
\sum_{|m'|\le |m|}
\|\partial_x^{m'}\nabla_x\phi\|_{L_x^2}
\|\nabla_x\partial_x^m c\|_{L_x^2}
\notag\\
&\le
C\alpha
\Big(
\sum_{|m'|\le |m|}
\|\partial_x^{m'}\nabla_x\phi\|_{L_x^2}^2
+
\|\nabla_x\partial_x^m c\|_{L_x^2}^2
\Big).
\label{Jc65}
\end{align}
For $\mathcal J_{6,\ell}^c$, $2\leq\ell\leq5$, we cannot integrate by
parts in $x$. Otherwise, in the borderline case $N=1$ with $m=0$, the
resulting estimate would involve zero-order quantities such as
$\mathbf P_1\tilde g$ or $c$. These quantities are not controlled by the
nonzero-order spatial derivative estimates, while the dissipation of $c$ is too
weak to control its zero-frequency mode. The projection $P_{\neq0}^x c$ is
therefore essential, since it allows us to apply the Poincar\'e inequality
while retaining the spatial derivative on the source term.

Combining the estimates for $\mathcal J_{6,\ell}^c$, $1\le \ell\le5$, we obtain
\allowdisplaybreaks\begin{align}
|\mathcal J_6^c|
&\le
(\varepsilon+C\alpha+C\eta)e^{2\beta t}
\|\nabla_x\partial_x^m[a,\mathbf b,c]\|_{L_x^2}^2
+C(\alpha+\eta)
\sum_{|m'|\le |m|}
\|\partial_x^{m'}\nabla_x\phi\|_{L_x^2}^2
\notag\\
&\quad
{}+
C_\varepsilon e^{2\beta t}
\|\nabla_x\partial_x^m\mathbf P_1\tilde g\|_{L^2_{x,v}}^2
+
C_\varepsilon(\alpha^2+\eta^2)
\sum_{0<|m'|\le |m|+1}
\Big(
\|w_l\partial_x^{m'}g_1\|_{L^\infty_{x,v}}^2
+
\|w_l\partial_x^{m'}g_2\|_{L^\infty_{x,v}}^2
\Big).
\label{Jc6}
\end{align}

Substituting \eqref{Jc134}, \eqref{Jc2}, \eqref{Jc5}, and \eqref{Jc6} into
\eqref{c-step}, we obtain
\allowdisplaybreaks\begin{align}
&\frac{d}{dt}\bigg\{
\frac{1}{2}\|P_{\neq0}^x\partial_x^m c\|_{L_x^2}^2
-\frac{3e^{\beta t}}{4b_0}\sum_{i=1}^3
\big\langle
\partial_{x_i}\partial_x^m q_i,
P_{\neq 0}^x\partial_x^m c
\big\rangle_x
\bigg\}
+\frac{5}{4b_0}e^{2\beta t}
\|\nabla_x\partial_x^m c\|_{L_x^2}^2
\notag\\
&\quad
+2\beta\|P_{\neq0}^x\partial_x^m c\|_{L_x^2}^2
+\sqrt{\frac{2}{3}}e^{\beta t}
\big\langle
{\rm div}_x\partial_x^m\mathbf b,
P_{\neq0}^x\partial_x^m c
\big\rangle_x
\notag\\
&\le
(\varepsilon+C\alpha+C\eta)e^{2\beta t}
\|\nabla_x\partial_x^m[a,\mathbf b,c]\|_{L_x^2}^2
+
C(\alpha+\eta)
\sum_{|m'|\le |m|}
\|\partial_x^{m'}\nabla_x\phi\|_{L_x^2}^2
\notag\\
&\quad
+
C_\varepsilon e^{2\beta t}
\|\nabla_x\partial_x^m\mathbf P_1\tilde g\|_{L^2_{x,v}}^2
+
C_\varepsilon(\alpha^2+\eta)
\sum_{0<|m'|\le |m|+1}
\Big(
\|w_l\partial_x^{m'}g_1\|_{L^\infty_{x,v}}^2
+
\|w_l\partial_x^{m'}g_2\|_{L^\infty_{x,v}}^2
\Big).
\label{c-dis}
\end{align}

\noindent\underline{\textbf{Step 2. Dissipation of $\mathbf b$.}}
Let $m\in\mathbb N^3$ satisfy $|m|=N-1$. Applying
$\partial_x^m$ to the second equation of \eqref{eq-macro-abc-ns}, testing
the resulting equation against $\partial_x^m b_i$, taking the $L^2(\mathbb T_x^3)$ inner product, and summing over $i=1,2,3$, we obtain
\allowdisplaybreaks\begin{align}
&\frac{d}{dt}\bigg\{\frac{1}{2}\|\partial_x^m[a,\mathbf b]\|_{L_x^2}^2
+\frac{1}{2}e^{-2\beta t}\|\partial_x^m\nabla_x\phi\|_{L_x^2}^2-\frac{e^{\beta t}}{2b_0}\sum_{i,j=1}^3
\big\langle\partial_{x_j}\partial_x^m d_{ij},\partial_x^m b_i\big\rangle_x
\bigg\}+\beta\|\partial_x^m\mathbf b\|_{L_x^2}^2
\notag\\
&
+\frac{e^{2\beta t}}{2b_0}\|\nabla_x\partial_x^m\mathbf b\|_{L_x^2}^2
+\frac{e^{2\beta t}}{6b_0}\|{\rm div}_x\partial_x^m\mathbf b\|_{L_x^2}^2
+\beta e^{-2\beta t}\|\partial_x^m\nabla_x\phi\|_{L_x^2}^2+\sqrt{\frac{2}{3}}e^{\beta t}
\big\langle\nabla_x\partial_x^m c,\partial_x^m\mathbf b\big\rangle_x
\notag\\
={}&-\alpha\big\langle\partial_x^m b_2,\partial_x^m b_1
\big\rangle_x+e^{-\beta t}\big\langle\partial_x^m(a\nabla_x\phi),
\partial_x^m\mathbf b\big\rangle_x-\frac{\beta e^{\beta t}}{2b_0}\sum_{i,j=1}^3\big\langle
\partial_{x_j}\partial_x^m d_{ij},\partial_x^m b_i\big\rangle_x
\notag\\
&-\frac{e^{\beta t}}{2b_0}\sum_{i,j=1}^3
\big\langle\partial_{x_j}\partial_x^m d_{ij},\partial_t\partial_x^m b_i
\big\rangle_x-\frac{e^{\beta t}}{2b_0}\sum_{i,j=1}^3\big\langle
\partial_{x_j}\partial_x^m\mathcal R_{ij}^{A},\partial_x^m b_i\big\rangle_x
:=\sum_{i=1}^5\mathcal J_i^b .
\label{b-step}
\end{align}
Here we used
\[
\begin{aligned}
&\quad e^{\beta t}
\big\langle
\nabla_x\partial_x^m a,
\partial_x^m\mathbf b
\big\rangle_x
-e^{-\beta t}
\big\langle
\partial_x^m\nabla_x\phi,
\partial_x^m\mathbf b
\big\rangle_x
\\
&=
\frac12\frac{d}{dt}
\Big\{
\|\partial_x^m a\|_{L_x^2}^2
+e^{-2\beta t}\|\partial_x^m\nabla_x\phi\|_{L_x^2}^2
\Big\}
+\beta e^{-2\beta t}
\|\partial_x^m\nabla_x\phi\|_{L_x^2}^2,
\end{aligned}
\]
and
\[
\begin{aligned}
\frac{e^{\beta t}}{2b_0}\sum_{i,j=1}^3
\big\langle
\partial_{x_j}\partial_t\partial_x^m d_{ij},
\partial_x^m b_i
\big\rangle_x
=&
\frac{d}{dt}\bigg\{
\frac{e^{\beta t}}{2b_0}\sum_{i,j=1}^3
\big\langle
\partial_{x_j}\partial_x^m d_{ij},
\partial_x^m b_i
\big\rangle_x
\bigg\}
-\frac{\beta e^{\beta t}}{2b_0}\sum_{i,j=1}^3
\big\langle
\partial_{x_j}\partial_x^m d_{ij},
\partial_x^m b_i
\big\rangle_x
\\
&
-\frac{e^{\beta t}}{2b_0}\sum_{i,j=1}^3
\big\langle
\partial_{x_j}\partial_x^m d_{ij},
\partial_t\partial_x^m b_i
\big\rangle_x .
\end{aligned}
\]

Similarly to \eqref{Jc134} and \eqref{Jc2}, we have
\begin{align}
&\quad |\mathcal J_1^b|+|\mathcal J_2^b|+|\mathcal J_3^b|
\notag \\
&\le
C\alpha \|\partial_x^m\mathbf b\|_{L_x^2}^2
+C e^{-\beta t}\|\partial_x^m(a\nabla_x\phi)\|_{L_x^2}
\|\partial_x^m\mathbf b\|_{L_x^2}+C\beta e^{\beta t}
\sum_{i,j=1}^3\|\partial_{x_j}\partial_x^m d_{ij}\|_{L_x^2}
\|\partial_x^m b_i\|_{L_x^2}
\notag\\
&\le
C\alpha e^{2\beta t}
\|\nabla_x\partial_x^m\mathbf b\|_{L_x^2}^2
+
C\eta e^{-\beta t}
\sum_{|m'|\le |m|}
\|\partial_x^{m'}\nabla_x\phi\|_{L_x^2}
\|\nabla_x\partial_x^m\mathbf b\|_{L_x^2}
+C\beta e^{\beta t}
\|\nabla_x\partial_x^m\mathbf P_1\tilde g\|_{L^2_{x,v}}
\|\nabla_x\partial_x^m\mathbf b\|_{L_x^2}
\notag\\
&\le
C(\alpha+\eta)e^{2\beta t}
\|\nabla_x\partial_x^m\mathbf b\|_{L_x^2}^2
+
C\alpha e^{2\beta t}
\|\nabla_x\partial_x^m\mathbf P_1\tilde g\|_{L^2_{x,v}}^2
+C\eta
\sum_{|m'|\le |m|}
\|\partial_x^{m'}\nabla_x\phi\|_{L_x^2}^2 .
\label{Jb123}
\end{align}

To estimate $\mathcal J_4^b$, we use the second equation in \eqref{eq-macro-abc} to eliminate $\partial_t\partial_x^m b_i$. Arguing as in \eqref{J5c-dec}--\eqref{Jc5}, we obtain
\begin{equation}\label{Jb4}
\begin{aligned}
|\mathcal J_4^b|
&\le
(\varepsilon+\alpha^2)e^{2\beta t}
\|\nabla_x\partial_x^m[a,\mathbf b,c]\|_{L_x^2}^2
+
C_\varepsilon e^{2\beta t}
\|\nabla_x\partial_x^m\mathbf P_1\tilde g\|_{L^2_{x,v}}^2
+
(\varepsilon+\eta^2)
\sum_{|m'|\le |m|}
\|\partial_x^{m'}\nabla_x\phi\|_{L_x^2}^2 .
\end{aligned}
\end{equation}

For $\mathcal J_5^b$, we substitute \eqref{micro-R} into $\mathcal R_{ij}^A$. Proceeding as in \eqref{J6c-dec}--\eqref{Jc6}, we obtain
\begin{align}
|\mathcal J_5^b|
&\le
(\varepsilon+C\alpha+C\eta)e^{2\beta t}
\|\nabla_x\partial_x^m[a,\mathbf b,c]\|_{L_x^2}^2
+
C(\alpha+\eta)
\sum_{|m'|\le |m|}
\|\partial_x^{m'}\nabla_x\phi\|_{L_x^2}^2
\notag\\
&\quad
+
C_\varepsilon e^{2\beta t}
\|\nabla_x\partial_x^m\mathbf P_1\tilde g\|_{L^2_{x,v}}^2
+
C_\varepsilon(\alpha^2+\eta^2)
\sum_{0<|m'|\le |m|+1}
\Big(
\|w_l\partial_x^{m'}g_1\|_{L^\infty_{x,v}}^2
+
\|w_l\partial_x^{m'}g_2\|_{L^\infty_{x,v}}^2
\Big).
\label{Jb5}
\end{align}

Substituting \eqref{Jb123}, \eqref{Jb4}, and \eqref{Jb5} into
\eqref{b-step}, we obtain
\begin{align}
&\frac{d}{dt}\bigg\{
\frac{1}{2}\|\partial_x^m[a,\mathbf b]\|_{L_x^2}^2
+\frac{1}{2}e^{-2\beta t}\|\partial_x^m\nabla_x\phi\|_{L_x^2}^2
-\frac{e^{\beta t}}{2b_0}\sum_{i,j=1}^3
\big\langle\partial_{x_j}\partial_x^m d_{ij},\partial_x^m b_i\big\rangle_x\bigg\}
+\beta\|\partial_x^m\mathbf b\|_{L_x^2}^2
\notag\\
&\quad
+\frac{e^{2\beta t}}{2b_0}\|\nabla_x\partial_x^m\mathbf b\|_{L_x^2}^2
+\frac{e^{2\beta t}}{6b_0}\|{\rm div}_x\partial_x^m\mathbf b\|_{L_x^2}^2
+\beta e^{-2\beta t}\|\partial_x^m\nabla_x\phi\|_{L_x^2}^2+\sqrt{\frac{2}{3}}e^{\beta t}
\big\langle\nabla_x\partial_x^m c,\partial_x^m\mathbf b\big\rangle_x
\notag\\
&\le
(\varepsilon+C\alpha+C\eta)e^{2\beta t}\|\nabla_x\partial_x^m[a,\mathbf b,c]\|_{L_x^2}^2
+(\varepsilon+C_\varepsilon\alpha+C_\varepsilon\eta)
\sum_{|m'|\le |m|}\|\partial_x^{m'}\nabla_x\phi\|_{L_x^2}^2
\notag\\
&\quad+C_\varepsilon e^{2\beta t}
\|\nabla_x\partial_x^m\mathbf P_1\tilde g\|_{L^2_{x,v}}^2
+C_\varepsilon(\alpha^2+\eta^2)\sum_{0<|m'|\le |m|+1}
\Big(\|w_l\partial_x^{m'}g_1\|_{L^\infty_{x,v}}^2+\|w_l\partial_x^{m'}g_2\|_{L^\infty_{x,v}}^2
\Big).
\label{b-dis}
\end{align}

\noindent\underline{\textbf{Step 3. Dissipation of $a$.}}
Let $m\in\mathbb N^3$ satisfy $|m|=N-1$. Applying $\partial_x^m$ to
the second equation of \eqref{eq-macro-abc}, testing the resulting equation
against $e^{\beta t}\nabla_x\partial_x^m a$, and using the first equation of
\eqref{eq-macro-abc}, similarly to \eqref{c-dis} and \eqref{b-dis}, we obtain
\begin{align}
\frac{d}{dt}&
\Big\{
e^{\beta t}
\big\langle
\partial_x^m\mathbf b,\nabla_x\partial_x^m a
\big\rangle_x
\Big\}
+e^{2\beta t}\|\nabla_x\partial_x^m a\|_{L_x^2}^2
+\|\partial_x^m a\|_{L_x^2}^2
\notag\\
\le{}&
\big(1+C\alpha^2\big)e^{2\beta t}
\|\nabla_x\partial_x^m\mathbf b\|_{L_x^2}^2
+
\Big[\frac{1}{4}+C(\alpha^2+\eta^2)\Big]
e^{2\beta t}
\|\nabla_x\partial_x^m a\|_{L_x^2}^2
\notag\\
&
+
C e^{2\beta t}
\|\nabla_x\partial_x^m c\|_{L_x^2}^2
+
C e^{2\beta t}
\|\nabla_x\partial_x^m\mathbf P_1\tilde g\|_{L^2_{x,v}}^2 .
\label{a-dis}
\end{align}

By taking the linear combination $M_1[\eqref{c-dis}+\eqref{b-dis}]+\eqref{a-dis}$, using
\[
\big\langle
{\rm div}_x\partial_x^m\mathbf b,
P_{\neq0}^x\partial_x^m c
\big\rangle_x
+
\big\langle
\nabla_x\partial_x^m c,
\partial_x^m\mathbf b
\big\rangle_x
=0,
\]
and choosing $M_1>0$ sufficiently large and fixed, then taking $\varepsilon>0$ sufficiently small, and finally choosing $\eta,\alpha>0$ small enough, we deduce that there exists a constant $\lambda_2>0$ such that
\allowdisplaybreaks\begin{align}
&\frac{d}{dt}\sum_{|m|=N-1}\Bigg\{
M_1\bigg[\frac{1}{2}\Big(\|\partial_x^m[a,\mathbf b]\|_{L_x^2}^2
+\|P_{\neq0}^x\partial_x^m c\|_{L_x^2}^2+e^{-2\beta t}\|\partial_x^m\nabla_x\phi\|_{L_x^2}^2\Big)
-\frac{3e^{\beta t}}{4b_0}\sum_{i=1}^3\big\langle
\partial_{x_i}\partial_x^m q_i,P_{\neq0}^x\partial_x^m c\big\rangle_x
\notag\\
&\quad
-\frac{e^{\beta t}}{2b_0}\sum_{i,j=1}^3\big\langle
\partial_{x_j}\partial_x^m d_{ij},\partial_x^m b_i\big\rangle_x
\bigg]+e^{\beta t}\big\langle\partial_x^m\mathbf b,\nabla_x\partial_x^m a
\big\rangle_x\Bigg\}+\lambda_2e^{2\beta t}\sum_{|m|=N-1}
\|\nabla_x\partial_x^m[a,\mathbf b,c]\|_{L_x^2}^2
\notag\\
&\quad
+\sum_{|m|=N-1}\|\partial_x^m a\|_{L_x^2}^2+M_1\beta
\sum_{|m|=N-1}\Big(\|\partial_x^m\mathbf b\|_{L_x^2}^2
+2\|P_{\neq0}^x\partial_x^m c\|_{L_x^2}^2
+e^{-2\beta t}\|\partial_x^m\nabla_x\phi\|_{L_x^2}^2\Big)
\notag\\
\le{}&
C e^{2\beta t}\sum_{|m|=N-1}\|\nabla_x\partial_x^m\mathbf P_1\tilde g\|_{L^2_{x,v}}^2
+C(\eta^2+\alpha^2)\sum_{0<|m'|\le N}\Big(\big\|h_1^{m',0}\big\|_{L^\infty_{x,v}}^2
+\big\|h_2^{m',0}\big\|_{L^\infty_{x,v}}^2\Big).
\label{abc-dis-1}
\end{align}
Here we have used the elliptic estimate \eqref{elliptic-Wkp} in Lemma~\ref{lem-elliptic-T3}, together with the Poincar\'e inequality. More precisely,
\[
\sum_{|m'|\le N-1}\|\partial_x^{m'}\nabla_x\phi\|_{L_x^2}^2
\le C\sum_{|m|=N-1}\|\partial_x^m\nabla_x\phi\|_{L_x^2}^2
\le C\sum_{|m|=N-1}\|\partial_x^m a\|_{L_x^2}^2
\le C\sum_{|m|=N-1}\|\partial_x^m\nabla_x a\|_{L_x^2}^2.
\]
Dividing \eqref{abc-dis-1} by $e^{2\beta t}$, and invoking the Poincar\'e inequality once again, we immediately arrive at \eqref{abc-dis}. This completes the proof.
\end{proof}

Having obtained the macroscopic dissipation from the moment system of
$\tilde g$, we now turn to the microscopic estimate for $g_2$.

\begin{lemma}\label{lem-micro-g2}
Let $[g_1,g_2]$ be the solution to \eqref{g1-eq}--\eqref{g2-eq}
satisfying the a priori assumption \eqref{assa}. Let $N\ge1$ and $l>3$.
Then the following estimates hold for all multi-indices
$m,n\in\mathbb N^3$ with $|m|+|n|\le N$.
\begin{enumerate}
\item[\rm (i)] The case $|m|>0,\ n=0:$
\allowdisplaybreaks\begin{align}
&\frac{d}{dt}
\Big\{
\|\partial_x^m g_2\|_{L^2_{x,v}}^2
+e^{-2\beta t}\|\nabla_x\partial_x^m\phi\|_{L_x^2}^2
\Big\}
+\delta_0\|\partial_x^m\mathbf P_1 g_2\|_{L^2_{x,v}}^2
\notag\\
&\le
C\varepsilon \|\partial_x^m[a,\mathbf b,c]\|_{L_x^2}^2
+C_\varepsilon\eta^2
\sum_{0<|m'|\le |m|}
\|\partial_x^{m'}a\|_{L_x^2}^2
+C_\varepsilon\|h_1^{m,0}\|_{L^\infty_{x,v}}^2 .
\label{est-micro-g2}
\end{align}

\item[\rm (ii)] The case $|m|\ge0,\ |n|>0:$
\allowdisplaybreaks\begin{align}
\frac{d}{dt}\|\partial_n^m g_2\|_{L^2_{x,v}}^2
+\delta_0\|\partial_n^m g_2\|_{L^2_{x,v}}^2
&\le
C\|\partial_x^m[a,\mathbf b,c]\|_{L_x^2}^2
+C\eta^2\sum_{0<|m'|\le |m|}
\|\partial_x^{m'}a\|_{L_x^2}^2
+C\|\partial_x^m \mathbf P_1g_2\|_{L^2_{x,v}}^2
\notag\\
&\quad
+C\sum_{n'\le n}\|h_1^{m,n'}\|_{L^\infty_{x,v}}^2
+C e^{2\beta t}\sum_{i=1}^3\mathbf 1_{\{n_i>0\}}n_i^2
\|h_2^{m+e_i,n-e_i}\|_{L^\infty_{x,v}}^2
\notag\\
&\quad
+C\alpha^2\mathbf 1_{\{n_2>0\}}n_2^2
\|h_2^{m,n-e_2+e_1}\|_{L^\infty_{x,v}}^2 .
\label{est-micro-g2mn}
\end{align}
\end{enumerate}
Here $\delta_0>0$ is the constant in Lemma~\ref{lem-est-L}, and
$h_1^{m,n}$ and $h_2^{m,n}$ are defined in
\eqref{eq-def-h12mn}.
\end{lemma}

\begin{proof}
Rewriting the equation $\eqref{g2-eq}_1$ as
\allowdisplaybreaks\begin{align}
& \partial_t g_2
+e^{\beta t}v\cdot\nabla_x g_2
-\beta\nabla_v\cdot(vg_2)
-\alpha v_2\partial_{v_1}g_2
+e^{-\beta t}\nabla_x\phi\cdot\nabla_v g_2
+Lg_2 \notag\\
& \quad
=\mu^{-\frac12}(1-\chi_M)\mathcal K g_1
+e^{-\beta t}\mu^{\frac12}v\cdot\nabla_x\phi.\label{g2-micro-rewrite}
\end{align}
For any fixed multi-index $m\in\mathbb N^3$ with $0<|m|\le N$, applying $\partial_x^m$ to \eqref{g2-micro-rewrite} and taking the $L^2(\mathbb T_x^3\times\mathbb R_v^3)$ inner product with $\partial_x^m g_2$, we have
\allowdisplaybreaks\begin{align}
&\frac12\frac{d}{dt}\|\partial_x^m g_2\|_{L^2_{x,v}}^2
+\langle L\partial_x^m g_2,\partial_x^m g_2\rangle_{x,v}
-\big\langle e^{-\beta t}\mu^{\frac12}v\cdot\nabla_x\partial_x^m\phi,\partial_x^m g_2\big\rangle_{x,v}
\notag\\
&\le \Big(\frac{3\beta}{2}+\varepsilon\Big)\|\partial_x^m g_2\|_{L^2_{x,v}}^2
+C_\varepsilon\|w_l\partial_x^m g_1\|_{L^\infty_{x,v}}^2
+e^{-\beta t}
\Big|
\big\langle
\partial_x^m\big(\nabla_x\phi\cdot\nabla_v g_2\big),
\partial_x^m g_2
\big\rangle_{x,v}
\Big|  .
\label{g2-micro-1}
\end{align}
For the last term in \eqref{g2-micro-1}, we first write
\[
\big\langle
\partial_x^m\big(\nabla_x\phi\cdot\nabla_v g_2\big),
\partial_x^m g_2
\big\rangle_{x,v} =
\sum_{0<m'\le m} C_{m,m'}
\big\langle
\partial_x^{m'}\nabla_x\phi\cdot
\nabla_v\partial_x^{m-m'}g_2,
\partial_x^m g_2
\big\rangle_{x,v} .
\]
Therefore, by the a priori assumption \eqref{assa}, the Poisson equation
$\Delta_x\phi=a$, and the Poincar\'e inequality, we obtain
\allowdisplaybreaks\begin{align}
e^{-\beta t}
\Big|
\big\langle
\partial_x^m\big(\nabla_x\phi\cdot\nabla_v g_2\big),
\partial_x^m g_2
\big\rangle_{x,v}
\Big|
&\le
Ce^{-\beta t}
\sum_{0<m'\le m}
\|\partial_x^{m'}\nabla_x\phi\|_{L_x^2}
\|w_l\nabla_v\partial_x^{m-m'}g_2\|_{L_{x,v}^\infty}
\|\partial_x^m g_2\|_{L^2_{x,v}}
\notag\\
&\le
C\eta
\sum_{0<m'\le m}
\|\partial_x^{m'}\nabla_x\phi\|_{L_x^2}
\|\partial_x^m g_2\|_{L^2_{x,v}}
\notag\\
&\le
\varepsilon\|\partial_x^m g_2\|_{L^2_{x,v}}^2
+
C_\varepsilon\eta^2
\sum_{0<m'\le m}\|\partial_x^{m'}\nabla_x\phi\|_{L_x^2}^2
\notag\\
&\le
\varepsilon\|\partial_x^m g_2\|_{L^2_{x,v}}^2
+
C_\varepsilon\eta^2
\sum_{0<|m'|\le |m|}\|\partial_x^{m'}a\|_{L_x^2}^2.
\label{g2-micro-2}
\end{align}
Next, for $\big\langle e^{-\beta t}\mu^{\frac 12}v\cdot\nabla_x\partial_x^m\phi,\partial_x^m g_2\big\rangle_{x,v}$, note that
\[
\int_{\mathbb R^3}v\mu^{\frac 12}\partial_x^m g_2\,dv
=\partial_x^m\mathbf b-\partial_x^m\mathbf b^{(1)} \quad{\rm with}\quad \mathbf b^{(1)}=\int_{\mathbb R^3}vg_1\,dv.
\]
Therefore,
\begin{equation}\label{g2-micro-3}
-\big\langle e^{-\beta t}\mu^{
\frac 12}v\cdot\nabla_x\partial_x^m\phi,\partial_x^m g_2\big\rangle_{x,v}
=
-e^{-\beta t}\big\langle \nabla_x\partial_x^m\phi,\partial_x^m\mathbf b-\partial_x^m\mathbf b^{(1)}\big\rangle_x
=: I_{\mu,1}+I_{\mu,2}.
\end{equation}
Using the first equation in \eqref{eq-macro-abc} and $a=\Delta_x\phi$, we have
\begin{align}
I_{\mu,1}
&=e^{-\beta t}\big\langle \partial_x^m\phi,{\rm div}_x\partial_x^m\mathbf b\big\rangle_x
=-e^{-2\beta t}\big\langle \partial_x^m\phi,\partial_t\Delta_x\partial_x^m\phi\big\rangle_x
\notag\\
&=\frac12\frac{d}{dt}\Big\{e^{-2\beta t}\|\partial_x^m \nabla_x\phi\|_{L_x^2}^2\Big\}
+\beta e^{-2\beta t}\|\partial_x^m \nabla_x\phi\|_{L_x^2}^2.
\label{g2-micro-4}
\end{align}
On the other hand,
\begin{equation}\label{g2-micro-5}
|I_{\mu,2}|
\le
\varepsilon\|\partial_x^m a\|_{L_x^2}^2
+C_\varepsilon e^{-2\beta t}\|\partial_x^m\mathbf b^{(1)}\|_{L_x^2}^2
\le
\varepsilon\|\partial_x^m a\|_{L_x^2}^2
+C_\varepsilon\|h_1^{m,0}\|_{L^\infty_{x,v}}^2,\quad l>3.
\end{equation}
In addition, note the truth
\allowdisplaybreaks\begin{align}
\|\partial_x^m g_2\|^2_{L^2_{x,v}}
&=\|\partial_x^m \mathbf P_0g_2\|^2_{L^2_{x,v}} +\|\partial_x^m \mathbf P_1g_2\|^2_{L^2_{x,v}}
\notag\\
&\le C\Big(\|\partial_x^m \mathbf P_0 \tilde g\|^2_{L^2_{x,v}}+ \|\partial_x^m \mathbf P_0 (\mu^{-\frac{1}{2}} g_1)\|^2_{L^2_{x,v}} +\|\partial_x^m \mathbf P_1g_2\|^2_{L^2_{x,v}}\Big)
\notag\\
&\le C\Big(\|\partial_x^m [a,\mathbf b,c]\|^2_{L^2_{x}}+\|\partial_x^m \mathbf P_1g_2\|^2_{L^2_{x,v}}+ \|h_1^{m,0}\|^2_{L^\infty_{x,v}} \Big),\quad l>3.
\label{g2-micro-6}
\end{align}

To conclude, plugging \eqref{g2-micro-2}, \eqref{g2-micro-4}, and \eqref{g2-micro-5} into \eqref{g2-micro-1}, and then applying the coercivity estimate \eqref{coercivity-L} from Lemma~\ref{lem-est-L}, we arrive at \eqref{est-micro-g2} by choosing $\varepsilon>0$ sufficiently small and subsequently taking $\alpha>0$ sufficiently small. Applying $\partial_n^m$ to \eqref{g2-micro-rewrite} and taking the $L^2(\mathbb T_x^3\times \mathbb R_v^3)$ inner product with $\partial_n^m g_2$, one obtains \eqref{est-micro-g2mn} after carrying out the same argument as above. This completes the proof.
\end{proof}

The following proposition provides the $L^2$-bounds needed in
\eqref{est-positive-sumh2} and \eqref{est-positive-n0-sumh2}.

\begin{proposition}\label{prop-EN1g2}
Let $[g_1,g_2]$ be the solution to \eqref{g1-eq}--\eqref{g2-eq}
satisfying the a priori assumption \eqref{assa}. Let $N\geq1$ and $l>3$.
Then, for any sufficiently small $\kappa>0$, we have
\allowdisplaybreaks\begin{align}
\sum_{\substack{|m|+|n|\le N\\ |m|>0}}
\kappa^{|n|}
\sup_{0\le s\le t}
e^{\lambda_n^m s}
\big\|\partial_n^m g_2(s)\big\|_{L^2_{x,v}}
&\le
C
\sum_{0<|m|\le N}
\big\|w_l\partial_x^m\widetilde f_0\big\|_{L^\infty_{x,v}}
+
C
\sum_{\substack{|m|+|n|\le N\\ |m|>0}}
\kappa^{|n|}
\mathfrak H_1^{m,n}(t)
\notag\\
&\quad
+
C(\eta+\alpha+\kappa)
\sum_{\substack{|m|+|n|\le N\\ |m|>0}}
\kappa^{|n|}
\mathfrak H_2^{m,n}(t).
\label{est-L2-g2-high}
\end{align}
In particular,
\allowdisplaybreaks\begin{align}
\sum_{0<|m|\le N}
\sup_{0\le s\le t}
e^{\lambda s}
\big\|\partial_x^m g_2(s)\big\|_{L^2_{x,v}}
&\le
C
\sum_{0<|m|\le N}
\big\|w_l\partial_x^m\widetilde f_0\big\|_{L^\infty_{x,v}}
\notag\\
&\quad
+
C
\sum_{0<|m|\le N}
\mathfrak H_1^{m,0}(t)
+
C(\eta+\alpha)
\sum_{0<|m|\le N}
\mathfrak H_2^{m,0}(t).
\label{est-L2-g2-high-n0}
\end{align}
Here $\lambda_n^m$ is defined in \eqref{def-lambda-mn}, while
$\mathfrak H_1^{m,n}$ and $\mathfrak H_2^{m,n}$ are defined in
\eqref{eq-def-mfh12mn}. In particular, $\lambda_0^m=\lambda$ for
$|m|>0$.
\end{proposition}

\begin{proof}
Fix $n\in\mathbb N^3$ with $0\le |n|<N$. We denote
\allowdisplaybreaks\begin{align}
\mathcal E_{N,n}^{(g_2)}(t)
={}&
\sum_{\substack{0<|m|\le N-|n|\\ |n|>0}}
\big\|\partial_n^m g_2\big\|_{L^2_{x,v}}^2
+
M_3M_2\sum_{0<|m|\le N}
\Big(
\big\|\partial_x^m g_2\big\|_{L^2_{x,v}}^2+
e^{-2\beta t}
\big\|\nabla_x\partial_x^m\phi\big\|_{L_x^2}^2\Big)
\notag\\
&\quad
+
2M_3e^{-2\beta t}\mathcal E_N^{(abc)}(t),
\label{EN1-def}
\\
\mathcal D_{N,n}^{(g_2)}(t)
={}&
\sum_{\substack{0<|m|\le N-|n|\\ |n|>0}}
\big\|\partial_n^m g_2\big\|_{L^2_{x,v}}^2
+
\sum_{0<|m|\le N}
\big\|\partial_x^m\mathbf P_1g_2\big\|_{L^2_{x,v}}^2
+
\sum_{0<|m|\le N}
\big\|\partial_x^m[a,\mathbf b,c]\big\|_{L_x^2}^2 .
\label{DN1-def}
\end{align}
By Lemma~\ref{lem-abc-dis} and Lemma~\ref{lem-micro-g2}, summing
\eqref{est-micro-g2mn} over all multi-indices $m\in\mathbb N^3$ with
$0<|m|\le N-|n|$ when $|n|>0$, summing \eqref{est-micro-g2} over all multi-indices
$m\in\mathbb N^3$ with $0<|m|\le N$, multiplying the latter by $M_2M_3$, and adding also
$2M_3$ times \eqref{abc-dis}, we obtain
\allowdisplaybreaks\begin{align}
&\quad\frac{d}{dt}\mathcal E_{N,n}^{(g_2)}(t)
+\delta_0
\sum_{\substack{0<|m|\le N-|n|\\ |n|>0}}
\big\|\partial_n^m g_2\big\|_{L^2_{x,v}}^2
+
M_3M_2\delta_0
\sum_{0<|m|\le N}
\big\|\partial_x^m\mathbf P_1g_2\big\|_{L^2_{x,v}}^2
\notag\\
&\quad
+
2M_3\lambda_0
\sum_{|m|\le N-1}
\big\|\nabla_x\partial_x^m[a,\mathbf b,c]\big\|_{L_x^2}^2
\notag\\
&\le
C
\sum_{|m|\le N-1}
\big\|\nabla_x\partial_x^m[a,\mathbf b,c]\big\|_{L_x^2}^2
+
C
\sum_{\substack{0<|m|\le N-|n|\\ |n|>0}}
\Bigg(
\big\|\partial_x^m\mathbf P_1g_2\big\|_{L^2_{x,v}}^2
+
\sum_{n'\le n}
\big\|h_1^{m,n'}\big\|_{L^\infty_{x,v}}^2
\notag\\
&\quad
+
e^{2\beta t}
\sum_{i=1}^3
\mathbf 1_{\{n_i>0\}}n_i^2
\big\|h_2^{m+e_i,n-e_i}\big\|_{L^\infty_{x,v}}^2
+
\alpha^2\mathbf 1_{\{n_2>0\}}n_2^2
\big\|h_2^{m,n-e_2+e_1}\big\|_{L^\infty_{x,v}}^2
\Bigg)
\notag\\
&\quad
+
CM_3M_2(\eta^2+\varepsilon)
\sum_{|m|\le N-1}
\big\|\nabla_x\partial_x^m[a,\mathbf b,c]\big\|_{L_x^2}^2
+
2CM_3
\sum_{|m|=N-1}
\big\|\nabla_x\partial_x^m\mathbf P_1\tilde g\big\|_{L^2_{x,v}}^2
\notag\\
&\quad
+
C_\varepsilon M_2M_3
\sum_{0<|m|\le N}
\big\|h_1^{m,0}\big\|_{L^\infty_{x,v}}^2
+
2CM_3\eta^2
\sum_{0<|m|\le N}
\Big(
\big\|h_1^{m,0}\big\|_{L^\infty_{x,v}}^2
+
\big\|h_2^{m,0}\big\|_{L^\infty_{x,v}}^2
\Big)
\notag\\
&\quad
+
4M_3\beta e^{-2\beta t}
\big|\mathcal E_N^{(abc)}(t)\big|.
\label{EN1g2-preabsorb}
\end{align}
We first choose $M_2>0$ sufficiently large, then choose $M_3>0$
sufficiently large, and finally take $\varepsilon>0$ and $\eta>0$
sufficiently small. This gives
\allowdisplaybreaks\begin{align}
\frac{d}{dt}\mathcal E_{N,n}^{(g_2)}(t)
+
2\tilde\lambda_0\mathcal D_{N,n}^{(g_2)}(t)
\le{}&
C
\sum_{\substack{0<|m|\le N-|n'|\\ |n'|\le |n|}}
\Big[
\big\|h_1^{m,n'}\big\|_{L^\infty_{x,v}}^2
+
(\eta^2+\alpha^2)
\big\|h_2^{m,n'}\big\|_{L^\infty_{x,v}}^2
\Big]
\notag\\
&+
C e^{2\beta t}
\sum_{\substack{0<|m|\le N-|n|\\ |n|>0}}
\sum_{i=1}^3
\mathbf 1_{\{n_i>0\}}n_i^2
\big\|h_2^{m+e_i,n-e_i}\big\|_{L^\infty_{x,v}}^2 ,
\label{EN1g2-absorb}
\end{align}
where $\tilde\lambda_0>0$ is some constant.

We next record the coercivity properties of
$\mathcal E_{N,n}^{(g_2)}(t)$. From \eqref{Eabc-def}, the definitions of
$a,\mathbf b,c,d_{ij},q_i$, and Poincar\'e inequality, we have
\[
\begin{aligned}
\big|e^{-2\beta t}\mathcal E_N^{(abc)}(t)\big|
&\le{}
CM_1\sum_{|m|=N-1}\Big(
\big\|\nabla_x\partial_x^m[a,\mathbf b,c]\big\|_{L_x^2}^2
+
e^{-2\beta t}\big\|\partial_x^m\nabla_x\phi\big\|_{L_x^2}^2
\\
&\quad\quad\quad\quad
+
\sum_{i=1}^3
\big\|\partial_{x_i}\partial_x^m q_i\big\|_{L_x^2}^2
+
\sum_{i,j=1}^3
\big\|\partial_{x_j}\partial_x^m d_{ij}\big\|_{L_x^2}^2
\Big)
\\
&\le
CM_1
\sum_{0<|m'|\le N}
\Big(
\big\|\partial_x^{m'}g_2\big\|_{L^2_{x,v}}^2
+
e^{-2\beta t}\big\|\nabla_x\partial_x^{m'}\phi\big\|_{L_x^2}^2
+
\big\|h_1^{m',0}\big\|_{L^\infty_{x,v}}^2
\Big),
\quad l>3.
\end{aligned}
\]
Therefore, from \eqref{EN1-def},
\[
\begin{aligned}
\mathcal E_{N,n}^{(g_2)}(t)
\ge{}&
\sum_{\substack{0<|m|\le N-|n|\\ |n|>0}}
\big\|\partial_n^m g_2\big\|_{L^2_{x,v}}^2
+
M_3(M_2-CM_1)
\sum_{0<|m|\le N}
\Big(\big\|\partial_x^m g_2\big\|_{L^2_{x,v}}^2
+
e^{-2\beta t}
\big\|\nabla_x\partial_x^m\phi\big\|_{L_x^2}^2\Big)
\\
&\quad-
CM_1M_3
\sum_{0<|m|\le N}
\big\|h_1^{m,0}\big\|_{L^\infty_{x,v}}^2 .
\end{aligned}
\]
Choosing $M_2\gg M_1$, we obtain
\begin{equation}\label{EN1g2-target}
\sum_{0<|m|\le N-|n|}
\big\|\partial_n^m g_2(t)\big\|_{L^2_{x,v}}^2
+
\sum_{0<|m|\le N}
e^{-2\beta t}
\big\|\nabla_x\partial_x^m\phi(t)\big\|_{L_x^2}^2
\le
C\mathcal E_{N,n}^{(g_2)}(t)
+
C\sum_{0<|m|\le N}
\big\|h_1^{m,0}(t)\big\|_{L^\infty_{x,v}}^2 .
\end{equation}
Moreover, using the elliptic estimate \eqref{elliptic-Wkp} for $\phi$ and $a=\Delta_x \phi$, we have
\begin{equation}\label{EN1g2-ED}
\mathcal E_{N,n}^{(g_2)}(t)
\le
C\mathcal D_{N,n}^{(g_2)}(t)
+
C\sum_{0<|m|\le N}
\big\|h_1^{m,0}(t)\big\|_{L^\infty_{x,v}}^2 .
\end{equation}

Next, combining \eqref{EN1g2-absorb} with \eqref{EN1g2-ED} and applying
Gr\"onwall inequality, we obtain
\allowdisplaybreaks\begin{align}
e^{2(\lambda-|n|\beta)t}\mathcal E_{N,n}^{(g_2)}(t)
&\le
e^{-[\tilde\lambda_0-2(\lambda-|n|\beta)]t}
\mathcal E_{N,n}^{(g_2)}(0)
+
C\int_0^t
e^{-[\tilde\lambda_0-2(\lambda-|n|\beta)](t-s)}
\sum_{\substack{0<|m|\le N-|n'|\\ |n'|\le |n|}}
e^{2(\lambda-|n|\beta)s}
\notag\\
&\quad\quad\quad\quad\quad\quad\quad
\times
\Big[
\big\|h_1^{m,n'}(s)\big\|_{L^\infty_{x,v}}^2
+
(\eta^2+\alpha^2)
\big\|h_2^{m,n'}(s)\big\|_{L^\infty_{x,v}}^2
\Big]\,ds
\notag\\
&\quad
+
C\int_0^t
e^{-[\tilde\lambda_0-2(\lambda-|n|\beta)](t-s)}
e^{2[\lambda-(|n|-1)\beta]s}
\notag\\
&\quad\quad
\times
\sum_{0<|m|\le N-|n|}
\sum_{i=1}^3
\mathbf 1_{\{n_i>0\}}n_i^2
\big\|h_2^{m+e_i,n-e_i}(s)\big\|_{L^\infty_{x,v}}^2\,ds .
\label{EN1g2-weighted}
\end{align}
By \eqref{choice-lambda}, we have
\[
0<\lambda<\frac{\widetilde\lambda_0}{2}.
\]
Combining \eqref{EN1g2-target} with \eqref{EN1g2-weighted}, we obtain
\eqref{EN1g2-weighted} yield
\allowdisplaybreaks\begin{align}
&\quad \sum_{0<|m|\le N-|n|}
\sup_{0\le s\le t} e^{(\lambda-|n|\beta)s}
\big\|\partial_n^m g_2(s)\big\|_{L^2_{x,v}}
\notag\\
&\le
C
\sum_{0<|m|\le N}
\big\|w_l\partial_x^m\tilde f_0\big\|_{L^\infty_{x,v}}
+
C
\sum_{\substack{0<|m|\le N-|n'|\\ |n'|\le |n|}}
\Big[
\mathfrak H_1^{m,n'}(t)
+
(\eta+\alpha)\mathfrak H_2^{m,n'}(t)
\Big]
\notag\\
&\quad
+
C
\sum_{0<|m|\le N-|n|}
\sum_{i=1}^3
\mathbf 1_{\{n_i>0\}}n_i
\mathfrak H_2^{m+e_i,n-e_i}(t).
\label{EN1g2-fixed-n}
\end{align}

Multiplying \eqref{EN1g2-fixed-n} by $\kappa^{|n|}$, and then summing over all $0\le |n|<N$, we obtain \eqref{est-L2-g2-high}. Taking $n=0$ in \eqref{EN1g2-fixed-n}, we obtain
\eqref{est-L2-g2-high-n0}. This completes the proof.
\end{proof}

\subsection{Zero-Order Estimates}

We now derive the zero-order estimate for $g_2$. The following lemma reduces
the problem to controlling the temperature mode $c$.

\begin{lemma}\label{lem-zero-g2-L2}
Let $[g_1,g_2]$ be the solution to \eqref{g1-eq}--\eqref{g2-eq}. Let
$l>3$. Then
\begin{equation}\label{est-g2-L2-1}
\sup_{0\le s\le t}\|g_2(s)\|_{L^2_{x,v}}
\le
C\sup_{0\le s\le t}\|c(s)\|_{L_x^2}
+
C\mathfrak H_1^{0,0}(t)
+
C\sum_{|m|=1}
\Big(
\mathfrak H_1^{m,0}(t)
+
\mathfrak H_2^{m,0}(t)
\Big).
\end{equation}
Here $\mathfrak H_1^{m,n}$ and $\mathfrak H_2^{m,n}$ are defined in
\eqref{eq-def-mfh12mn}.
\end{lemma}

\begin{proof}
Taking the $L^2(\mathbb T_x^3\times\mathbb R_v^3)$ inner product of \eqref{g2-eq} with $g_2$, and using the coercivity estimate \eqref{coercivity-L} in Lemma~\ref{lem-est-L}, we obtain
\begin{equation}\label{zero-g2-energy}
\frac{d}{dt}\|g_2\|_{L^2_{x,v}}^2
+\delta_0\|\mathbf P_1 g_2\|_{L^2_{x,v}}^2
\le
C\|\mathbf P_0 g_2\|_{L^2_{x,v}}^2
+
C\|h_1^{0,0}\|_{L^\infty_{x,v}}^2.
\end{equation}

Since $\tilde g=\mu^{-\frac{1}{2}}g_1+g_2$, we have
\begin{align}
\|\mathbf P_0g_2\|_{L^2_{x,v}}^2
&\le
C\|\mathbf P_0\tilde g\|_{L^2_{x,v}}^2
+C\|\mathbf P_0(\mu^{-\frac12}g_1)\|_{L^2_{x,v}}^2
\notag\\
&\le
C\|\nabla_x[a,\mathbf b]\|_{L^2_x}^2
+C\|c\|_{L^2_x}^2
+C\|h_1^{0,0}\|_{L^\infty_{x,v}}^2
\notag\\
&\le
C\|c\|_{L^2_x}^2
+C\sum_{|m|=1}\|h_2^{m,0}\|_{L^\infty_{x,v}}^2
+C\sum_{|m|\le1}\|h_1^{m,0}\|_{L^\infty_{x,v}}^2,
\quad l>3.
\label{P0g2-est}
\end{align}

Combining \eqref{zero-g2-energy} and \eqref{P0g2-est}, for some $\lambda_3>0$, we arrive at
\[
\frac {d}{ dt}\|g_2\|_{L^2_{x,v}}^2
+\lambda_3\|g_2\|_{L^2_{x,v}}^2
\le
C\|c\|_{L_x^2}^2
+
C\sum_{|m|=1}\|h_2^{m,0}\|_{L^\infty_{x,v}}^2
+
C\sum_{|m|\le1}\|h_1^{m,0}\|_{L^\infty_{x,v}}^2.
\]
Since $g_2(0)=0$, Gr\"onwall inequality yields
\[
\|g_2(t)\|_{L^2_{x,v}}^2
\le
C\int_0^t e^{-\lambda_3(t-s)}
\Big(
\|c(s)\|_{L_x^2}^2
+
\sum_{|m|=1}\|h_2^{m,0}(s)\|_{L^\infty_{x,v}}^2
+
\sum_{|m|\le1}\|h_1^{m,0}(s)\|_{L^\infty_{x,v}}^2
\Big)\,ds.
\]
Taking square roots and the supremum gives \eqref{est-g2-L2-1}. This completes the proof.
\end{proof}

By Lemma~\ref{lem-zero-g2-L2}, it remains to control the temperature mode $c$.
We decompose
\[
c=P_0^x c+P_{\neq0}^x c.
\]
By Poincar\'e inequality, the nonzero-frequency mode $P_{\neq0}^x c$ is
controlled by its spatial derivatives, and hence by the estimates obtained
above. It remains to study the zero-frequency mode $P_0^x c$. The key step is
to derive a closed zero-frequency ODE system coupling a renormalized energy with suitable second-order moments. Before deriving this system, we recall the
following lemma, which will be used below.

\begin{lemma}[{\cite[Lemma 4.4]{DuanLiuShen25}}]\label{lem-Qsym-moment}
In the case of Maxwell molecules, the following identities hold:
\begin{equation}\label{qsym-offdiag}
\langle Q(\tilde f,G)+Q(G,\tilde f),v_iv_j\rangle_v
=
-2b_0\left[
\langle \tilde f,v_iv_j\rangle_v \langle G,1\rangle_v
+
\langle \tilde f,1\rangle_v \langle G,v_iv_j\rangle_v
\right],
\quad i\neq j,
\end{equation}
\begin{equation}\label{qsym-diag}
\langle Q(\tilde f,G)+Q(G,\tilde f),v_iv_i\rangle_v
=
-2b_0\bigg[
\big\langle \tilde f,v_iv_i-\frac{|v|^2}{3}\big\rangle_v \langle G,1\rangle_v
+
\langle \tilde f,1\rangle_v
\big\langle G,v_iv_i-\frac{|v|^2}{3}\big\rangle_v
\bigg],
\end{equation}
\begin{equation}\label{Qff-offdiag}
\langle Q(\tilde f,\tilde f),v_iv_j\rangle_v
=
-2b_0\Big[
\langle \tilde f,v_iv_j\rangle_v \langle \tilde f,1\rangle_v
-
\langle \tilde f,v_i\rangle_v \langle \tilde f,v_j\rangle_v
\Big],
\quad i\neq j,
\end{equation}
\begin{equation}\label{Qff-diag}
\langle Q(\tilde f,\tilde f),v_iv_i\rangle_v
=
-2b_0\bigg[
\big\langle \tilde f,v_iv_i-\frac{|v|^2}{3}\big\rangle_v \langle \tilde f,1\rangle_v
+
\langle \tilde f,v_i\rangle_v^2
-
\frac{1}{3}\sum_{j=1}^3 \langle \tilde f,v_j\rangle_v^2
\bigg].
\end{equation}
Moreover, under the conservation laws \eqref{int-ab}, the zero-frequency projections satisfy
\begin{equation}\label{P0-qsym-d12}
P_0^x\langle Q(\tilde f,G)+Q(G,\tilde f),v_1v_2\rangle_v
=
-2b_0 P_0^x \langle \tilde f,v_1 v_2\rangle_v,
\end{equation}
\begin{equation}\label{P0-qsym-diag}
P_0^x\langle Q(\tilde f,G)+Q(G,\tilde f),v_iv_i\rangle_v
=
-2b_0\Big[
P_0^x d_{ii}
-
\frac{1}{3}P_0^x \langle \tilde f,|v|^2\rangle_v
\Big],
\end{equation}
and the identities \eqref{Qff-offdiag} and \eqref{Qff-diag} remain valid with $\tilde f$ replaced by $G$. Here $\tilde f$ denotes the perturbation defined in \eqref{def-tildef}, and $d_{ij}$ is defined in \eqref{def-dij-qi}.
\end{lemma}

We now derive the zero-frequency ODE system. For convenience, we recall that
\allowdisplaybreaks\begin{align}\label{tildeg-eq-rewrite}
& \partial_t\tilde g
+ e^{\beta t}v\cdot\nabla_x\tilde g
- \beta\nabla_v\cdot(v\tilde g)
+ \frac{\beta}{2}|v|^2\tilde g
- \alpha v_2\partial_{v_1}\tilde g
+ \frac{\alpha}{2}v_1v_2\tilde g
+ e^{-\beta t}\nabla_x\phi\cdot\nabla_v\tilde g\notag\\
&
- \frac{1}{2} e^{-\beta t}(v\cdot\nabla_x\phi)\tilde g
+ L\tilde g
- e^{-\beta t}\mu^{\frac{1}{2}}v\cdot\nabla_x\phi
+ \alpha e^{-\beta t}\nabla_x\phi\cdot\Big(\nabla_v G_1-\frac{v}{2}G_1\Big)\notag\\
&
=\alpha\Gamma(G_1,\tilde g)
+ \alpha\Gamma(\tilde g,G_1)
+ \Gamma(\tilde g,\tilde g)=:\widetilde{\mathcal F}.
\end{align}
We define the renormalized total energy by
\begin{equation}
\mathcal E_\alpha(t)=\sqrt6\,P_0^x c(t)+e^{-2\beta t}\|\nabla_x\phi(t)\|_{L_x^2}^2,
\end{equation}

Applying $P_0^x$ to the third
equation of \eqref{eq-macro-abc}, we obtain
\begin{equation}\label{P0c-eq}
\partial_t P_0^x c
+\sqrt6\beta P_0^x a
+2\beta P_0^x c
+\sqrt{\frac23}\alpha P_0^x d_{12}
-\sqrt{\frac23}e^{-\beta t}P_0^x(\nabla_x\phi\cdot\mathbf b)
=0.
\end{equation}
On the other hand, combining the first
equation of \eqref{eq-macro-abc} with $\Delta_x\phi=a$, we get
\begin{equation}\label{field-energy-eq}
\frac{d}{dt}\Big(e^{-2\beta t}\|\nabla_x\phi\|_{L_x^2}^2\Big)
+2\beta e^{-2\beta t}\|\nabla_x\phi\|_{L_x^2}^2
+2e^{-\beta t}P_0^x(\nabla_x\phi\cdot\mathbf b)=0.
\end{equation}
Using \eqref{P0c-eq}, \eqref{field-energy-eq}, and $P_0^x a=0$, we obtain the evolution equation for $\mathcal E_\alpha$:
\begin{equation}\label{E-eq-zero}
\partial_t\mathcal E_\alpha+2\beta\mathcal E_\alpha+2\alpha P_0^x d_{12}=0.
\end{equation}

It remains to close \eqref{E-eq-zero}. We derive the zero-frequency equations for the second-order moments $P_0^x d_{12}$ and $P_0^x d_{22}$. Taking the $L^2(\mathbb R_v^3)$ inner products of \eqref{tildeg-eq-rewrite} with $A_{12}$ and $A_{22}$, respectively, and then applying $P_0^x$, we obtain
\begin{align}
\partial_t P_0^x d_{12}
&+(2\beta+2b_0)P_0^x d_{12}
+\alpha\Big(\frac13\mathcal E_\alpha-\frac13 e^{-2\beta t}\|\nabla_x\phi\|_{L_x^2}^2+P_0^x d_{22}\Big)
\notag\\
&-e^{-\beta t}P_0^x(b_1\partial_{x_2}\phi+b_2\partial_{x_1}\phi)
=
P_0^x\langle \widetilde{\mathcal F},A_{12}\rangle_v,
\label{P0d12}
\end{align}
and
\begin{equation}\label{P0d22}
\partial_t P_0^x d_{22}
+(2\beta+2b_0)P_0^x d_{22}
-\frac{2\alpha}{3}P_0^x d_{12}
-\frac23 e^{-\beta t}P_0^x(-b_1\partial_{x_1}\phi+2b_2\partial_{x_2}\phi-b_3\partial_{x_3}\phi)
=
P_0^x\langle \widetilde {\mathcal F},A_{22}\rangle_v.
\end{equation}

Recalling that
$\tilde f=\mu^{\frac{1}{2}}\tilde g,$ $G=\mu+\alpha\mu^{\frac{1}{2}}G_1,$
we have, by the definition of $\Gamma$,
\[
\big\langle \alpha\Gamma(G_1,\tilde g)+\alpha\Gamma(\tilde g,G_1),A_{12}\big\rangle_v
=
\langle Q(\tilde f,G)+Q(G,\tilde f),v_1v_2\rangle_v
-
\langle Q(\tilde f,\mu)+Q(\mu,\tilde f),v_1v_2\rangle_v.
\]
By \eqref{qsym-offdiag}, the normalization
$\langle G,1\rangle_v=1$ in \eqref{G-normal}, and \eqref{LAij}, we have
\[
\langle Q(\tilde f,G)+Q(G,\tilde f),v_1v_2\rangle_v
=
-2b_0\big(\langle \tilde f,v_1v_2\rangle_v+\langle \tilde f,1\rangle_v\langle G,v_1v_2\rangle_v\big),
\]
whereas
\[
\langle Q(\tilde f,\mu)+Q(\mu,\tilde f),v_1v_2\rangle_v
=
-\langle L\tilde g,A_{12}\rangle_v
=
-2b_0d_{12}.
\]
Since $\langle \tilde f,v_1v_2\rangle_v=d_{12}$ and
$\langle \tilde f,1\rangle_v=a$, we obtain
\begin{equation}\label{est-P0A12}
\begin{aligned}
P_0^x\big\langle
\alpha\Gamma(G_1,\tilde g)+\alpha\Gamma(\tilde g,G_1),
A_{12}
\big\rangle_v
&=
-2b_0\langle G,v_1v_2\rangle_vP_0^x a
=0.
\end{aligned}
\end{equation}
The same argument, using \eqref{qsym-diag}, gives
\begin{equation}\label{est-P0A22}
P_0^x\big\langle \alpha\Gamma(G_1,\tilde g)+\alpha\Gamma(\tilde g,G_1),A_{22}\big\rangle_v=0.
\end{equation}

Therefore, combining \eqref{E-eq-zero}--\eqref{P0d22} with
\eqref{est-P0A12}--\eqref{est-P0A22}, we can write the resulting
zero-frequency equations in the matrix form
\begin{equation}\label{zero-ODE-matrix}
\frac{d}{dt}U(t)+A_\alpha U(t)=R(t),
\end{equation}
where
\begin{equation}\label{def-U-Aa}
U(t)=\begin{pmatrix}
\mathcal E_\alpha(t)\\
P_0^x d_{12}(t)\\
P_0^x d_{22}(t)
\end{pmatrix},
\quad
A_\alpha=\begin{pmatrix}
2\beta & 2\alpha & 0\\
\frac{\alpha}{3} & 2\beta+2b_0 & \alpha\\
0 & -\frac{2\alpha}{3} & 2\beta+2b_0
\end{pmatrix},
\end{equation}
and
\begin{equation}\label{def-Rt}
R(t)=\begin{pmatrix}
0\\[0.25cm]
P_0^x\langle \Gamma(\tilde g,\tilde g),A_{12}\rangle_v
+e^{-\beta t}P_0^x(b_1\partial_{x_2}\phi+b_2\partial_{x_1}\phi)
+\frac{\alpha}{3}e^{-2\beta t}\|\nabla_x\phi\|_{L_x^2}^2\\[0.25cm]
P_0^x\langle \Gamma(\tilde g,\tilde g),A_{22}\rangle_v
+\frac{2}{3} e^{-\beta t}P_0^x(-b_1\partial_{x_1}\phi+2b_2\partial_{x_2}\phi-b_3\partial_{x_3}\phi)
\end{pmatrix}.
\end{equation}

This system is the core of the zero-frequency analysis and provides the control
of $P_0^x c$, hence of the temperature mode $c$. The following lemma gives
the corresponding uniform estimates.

\begin{lemma}\label{lem-Uc-L2}
Let $[g_1,g_2]$ be the solution to \eqref{g1-eq}--\eqref{g2-eq}
satisfying the a priori assumption \eqref{assa}. Let $l>3$, and let $U(t)$
be defined by \eqref{def-U-Aa} and satisfy \eqref{zero-ODE-matrix}. Then the
following estimates hold:
\begin{equation}\label{est-U-L2-sharp}
\begin{aligned}
\sup_{0\le s\le t}|U(s)|
\le
C|U(0)|
+
C\eta
\sum_{|m|=1}
\Big(
\mathfrak H_1^{m,0}(t)
+
\mathfrak H_2^{m,0}(t)
\Big),
\end{aligned}
\end{equation}
and
\begin{equation}\label{est-c-L2-sharp}
\begin{aligned}
\sup_{0\le s\le t}\|c(s)\|_{L_x^2}
\le
C\|w_l\tilde f_0\|_{L^\infty_{x,v}}
+
C
\sum_{|m|=1}
\Big(
\mathfrak H_1^{m,0}(t)
+
\mathfrak H_2^{m,0}(t)
\Big).
\end{aligned}
\end{equation}
Here $\mathfrak H_1^{m,0}$ and $\mathfrak H_2^{m,0}$ are defined in
\eqref{eq-def-mfh12mn}.
\end{lemma}

\begin{proof}
Applying Duhamel's formula to \eqref{zero-ODE-matrix}, we obtain
\[
U(t)=e^{-tA_\alpha}U(0)+\int_0^t e^{-(t-s)A_\alpha}R(s)\,ds.
\]

\noindent{\textbf{\underline{Step 1. Spectral analysis of $A_\alpha$.}}}
A direct computation from \eqref{def-U-Aa} shows that
\begin{equation}\label{charpoly-Aa}
\det(\tilde\lambda I-A_\alpha)=
\tilde\lambda^3-(4b_0+6\beta)\tilde\lambda^2+(4b_0^2+16b_0\beta+12\beta^2)\tilde\lambda
+\frac{4}{3}\alpha^2 b_0-8\beta(b_0+\beta)^2 =:p(\tilde\lambda).
\end{equation}
To verify that $0$ is an eigenvalue of $A_\alpha$, we return to the profile equation for $G$,
\[
-\beta\nabla_v\cdot(vG)-\alpha v_2\partial_{v_1}G=Q(G,G).
\]
Testing this equation against $|v|^2$, $v_1v_2$, and
$\displaystyle v_2^2-\frac{|v|^2}{3}$, and then applying
\eqref{Qff-offdiag} and \eqref{Qff-diag} in Lemma~\ref{lem-Qsym-moment} with $\tilde f$ replaced by $G$, together with \eqref{G-normal}, we obtain
\begin{equation}\label{steady-second-moment}
\left\{
\begin{aligned}
2\beta\langle G,|v|^2\rangle_v+2\alpha\langle G,v_1v_2\rangle_v&=0,\\
\frac{\alpha}{3}\langle G,|v|^2\rangle_v
+(2\beta+2b_0)\langle G,v_1v_2\rangle_v
+\alpha\big\langle G,v_2^2-\frac{|v|^2}{3}\big\rangle_v&=0,\\
-\frac{2\alpha}{3}\langle G,v_1v_2\rangle_v
+(2\beta+2b_0)\big\langle G,v_2^2-\frac{|v|^2}{3}\big\rangle_v&=0.
\end{aligned}
\right.
\end{equation}
Hence
\[
A_\alpha
\begin{pmatrix}
\langle G,|v|^2\rangle_v\\
\langle G,v_1v_2\rangle_v\\
\displaystyle\big\langle G,v_2^2-\frac{|v|^2}{3}\big\rangle_v
\end{pmatrix}
=0.
\]
Solving the linear relations \eqref{steady-second-moment} yields a corresponding eigenvector
\[
r_0=
\begin{pmatrix}
1\\[1mm]
-\frac{\beta}{\alpha}\\[2mm]
-\frac{\beta}{3(b_0+\beta)}
\end{pmatrix},
\quad
A_\alpha r_0=0.
\]
Moreover, the same relations imply
\begin{equation}\label{alpha-beta-relation}
\alpha^2 b_0=6\beta(b_0+\beta)^2,\quad\beta=\frac{\alpha^2}{6b_0}+O(\alpha^4)
\end{equation}

Since $0$ is an eigenvalue, the characteristic polynomial of $A_\alpha$ factors as
\[
\det(\tilde\lambda I-A_\alpha)
=
\tilde\lambda\Big[\tilde\lambda^2-(4b_0+6\beta)\tilde\lambda+(4b_0^2+16b_0\beta+12\beta^2)\Big].
\]
Thus the remaining two eigenvalues are
\[
\lambda_\pm
=
2b_0+3\beta\pm i \omega_\alpha,
\quad
\omega_\alpha=\sqrt{\beta(4b_0+3\beta)}.
\]
By \eqref{alpha-beta-relation}, for $\alpha>0$ sufficiently small, we have
\begin{equation}\label{lambda-alpha-relation}
\lambda_\pm
=
2b_0\pm i\frac{\sqrt6}{3}\alpha+O(\alpha^2).
\end{equation}
Hence $0,\lambda_+,\lambda_-$ are three distinct eigenvalues, and $A_\alpha$ is diagonalizable over $\mathbb C$.

For $\lambda_\pm$, we choose the corresponding eigenvectors
\[
r_\pm=
\begin{pmatrix}
\frac{2\alpha}{2b_0+\beta\pm i\omega_\alpha}\\[3mm]
1\\[3mm]
-\frac{2\alpha}{3(\beta\pm i\omega_\alpha)}
\end{pmatrix},
\quad
A_\alpha r_\pm=\lambda_\pm r_\pm.
\]
We then define
\begin{equation}\label{def-Qa}
Q_\alpha=(r_0,r_+,r_-).
\end{equation}
By construction,
\[
Q_\alpha^{-1}A_\alpha Q_\alpha
=
\operatorname{Diag}(0,\lambda_+,\lambda_-).
\]

By \eqref{lambda-alpha-relation},
\[
\Big|\frac{\beta}{\alpha}\Big|\le C,
\quad
\Big|\frac{\beta}{b_0+\beta}\Big|\le C,
\quad
\left|\frac{2\alpha}{2b_0+\beta\pm i\omega_\alpha}\right|\le C,
\quad
\left|\frac{2\alpha}{3(\beta\pm i\omega_\alpha)}\right|\le C.
\]
Hence $\|Q_\alpha\|_{\max}\le C$. Moreover, as $\alpha\to0$, one has $\beta\to0$, and therefore
\[
Q_\alpha\to
Q_0=
\begin{pmatrix}
1&0&0\\
0&1&1\\
0&i\frac{\sqrt{6}}{3}&-i \frac{\sqrt{6}}{3}
\end{pmatrix},
\quad
\det Q_0\neq0.
\]
By continuity, $Q_\alpha$ remains invertible for $\alpha>0$ sufficiently small, and
\[
\|Q_\alpha^{-1}\|_{\max}\le C.
\]
Therefore,
\[
\|e^{-tA_\alpha}\|_{\max}
\le
C\|Q_\alpha\|_{\max}
\|\operatorname{Diag}(1,e^{-\lambda_+ t},e^{-\lambda_- t})\|_{\max}
\|Q_\alpha^{-1}\|_{\max}\le C,
\quad t\ge0.
\]
Hence
\begin{equation}\label{est-UU0R}
|U(t)|
\le
C|U(0)|
+
C\int_0^t |R(s)|\,ds.
\end{equation}
\noindent{\textbf{\underline{Step 2. Estimate of the nonlinear term $R(s)$.}}} For $l>3$, we first estimate the macroscopic components. Since
$P_0^x a=P_0^x\mathbf b=P_0^xP_{\neq0}^x c=0$, the Poincar\'e inequality gives
\begin{equation}\label{abc-nonzero}
\|a(s)\|_{L_x^2}
+
\|\mathbf b(s)\|_{L_x^2}
+
\|P_{\neq0}^x c(s)\|_{L_x^2}
\le
C e^{-\lambda s}
\sum_{|m|=1}
\Big(
\mathfrak H_1^{m,0}(t)
+
\mathfrak H_2^{m,0}(t)
\Big).
\end{equation}
Similarly, by the Poincar\'e inequality and the definitions of $d_{12}$ and
$d_{22}$ in \eqref{def-dij-qi}, we have
\allowdisplaybreaks\begin{align}
\|P_{\neq0}^x d_{12}(s)\|_{L_x^2}
+\|P_{\neq0}^x d_{22}(s)\|_{L_x^2}
&\le
C\sum_{|m|=1}
\Big[
\Big\|
\big\langle
\mu^{-\frac12}\partial_x^m g_1(s)+\partial_x^m g_2(s),
v_1v_2\mu^{\frac12}
\big\rangle_v
\Big\|_{L_x^2}
\notag\\
&\quad\quad
+
\Big\|
\big\langle
\mu^{-\frac12}\partial_x^m g_1(s)+\partial_x^m g_2(s),
\Big(v_2^2-\frac{|v|^2}{3}\Big)\mu^{\frac12}
\big\rangle_v
\Big\|_{L_x^2}
\Big]
\notag\\
&\le
C\sum_{|m|=1}
\Big(
\|h_1^{m,0}(s)\|_{L^\infty_{x,v}}
+
\|h_2^{m,0}(s)\|_{L^\infty_{x,v}}
\Big)
\notag\\
&\le
Ce^{-\lambda s}
\sum_{|m|=1}
\Big(
\mathfrak H_1^{m,0}(t)
+
\mathfrak H_2^{m,0}(t)
\Big).
\label{dij-nonzero}
\end{align}
Since $\Delta_x\phi=a$, the elliptic estimate
\eqref{elliptic-Wkp} in Lemma~\ref{lem-elliptic-T3} yields
\begin{equation}\label{phi-nonzero}
\|\nabla_x\phi(s)\|_{L_x^2}
\le
C\|a(s)\|_{L_x^2}
\le
C e^{-\lambda s}
\sum_{|m|=1}
\Big(
\mathfrak H_1^{m,0}(t)
+
\mathfrak H_2^{m,0}(t)
\Big).
\end{equation}

We now estimate the three components of $R(s)$. By the definition of $\Gamma$, and applying \eqref{Qff-offdiag} in Lemma~\ref{lem-Qsym-moment}, we obtain
\[
\langle \Gamma(\tilde g,\tilde g),A_{12}\rangle_v
=\langle Q(\tilde f,\tilde f),v_1v_2\rangle_v
=
-2b_0(d_{12}a-b_1b_2).
\]
Applying $P_0^x$ and using $P_0^x a=P_0^x \mathbf b=0$, we obtain
\[
P_0^x\langle \Gamma(\tilde g,\tilde g),A_{12}\rangle_v
=
-2b_0 P_0^x\Big((P_{\neq0}^x d_{12})a\Big)
+
2b_0 P_0^x\Big((P_{\neq0}^x b_1)b_2\Big).
\]
Hence, by \eqref{abc-nonzero} and \eqref{dij-nonzero}, together with the a priori assumption \eqref{assa}, we obtain
\allowdisplaybreaks\begin{align}
\left|P_0^x\langle \Gamma(\tilde g,\tilde g),A_{12}\rangle_v(s)\right|
&\le
C\|P_{\neq0}^x d_{12}(s)\|_{L_x^2}\|a(s)\|_{L_x^2}
+
C\|P_{\neq0}^x b_1(s)\|_{L_x^2}\|b_2(s)\|_{L_x^2}
\notag\\
&\le
C\|P_{\neq0}^x d_{12}(s)\|_{L_x^2}\|a(s)\|_{L_x^2}
+
C\|\mathbf b(s)\|_{L_x^2}\|b_2(s)\|_{L_x^2}
\notag\\
&\le
C\eta e^{-\lambda s}
\sum_{|m|=1}
\Big(
\mathfrak H_1^{m,0}(t)
+
\mathfrak H_2^{m,0}(t)
\Big).
\label{A12-source}
\end{align}
Similarly, by \eqref{Qff-diag} in Lemma~\ref{lem-Qsym-moment},
\[
\langle \Gamma(\tilde g,\tilde g),A_{22}\rangle_v
=-2b_0\Big(d_{22}a+b_2^2-\frac{1}{3}\sum_{j=1}^3 b_j^2\Big),
\]
and hence,
\begin{align}
\left|P_0^x\langle \Gamma(\tilde g,\tilde g),A_{22}\rangle_v(s)\right|
&\le
C\|P_{\neq0}^x d_{22}(s)\|_{L_x^2}\|a(s)\|_{L_x^2}
+
C\sum_{j=1}^3 \|b_j(s)\|_{L_x^2}^2
\notag\\
&\le
C\eta e^{-\lambda s}
\sum_{|m|=1}
\Big(
\mathfrak H_1^{m,0}(t)
+
\mathfrak H_2^{m,0}(t)
\Big).
\label{A22-source}
\end{align}
For the Poisson terms, by \eqref{phi-nonzero} and the a priori assumption \eqref{assa} again,
\begin{align}
&\left|e^{-\beta s}P_0^x\Big(b_1\partial_{x_2}\phi+b_2\partial_{x_1}\phi\Big)(s)\right|
+
\left|\frac{\alpha}{3}e^{-2\beta s}\|\nabla_x\phi(s)\|^2_{L_x^2}\right|
\notag\\
&\quad
+
\left|\frac{2}{3}e^{-\beta s}P_0^x\Big(-b_1\partial_{x_1}\phi+2b_2\partial_{x_2}\phi-b_3\partial_{x_3}\phi\Big)(s)\right|
\le
C\eta e^{-\lambda s}
\sum_{|m|=1}
\Big(
\mathfrak H_1^{m,0}(t)
+
\mathfrak H_2^{m,0}(t)
\Big).
\label{poisson-source}
\end{align}

Combining \eqref{est-UU0R}--\eqref{poisson-source}, we obtain
\eqref{est-U-L2-sharp}.

It remains to estimate $c$. Since
\[
\begin{aligned}
\|c(s)\|_{L_x^2}
&\le
\|P_0^x c(s)\|_{L_x^2}
+
\|P_{\neq0}^x c(s)\|_{L_x^2} \\
&=
|P_0^x c(s)|
+
\|P_{\neq0}^x c(s)\|_{L_x^2} \\
&\le
C|U(s)|
+
\|P_{\neq0}^x c(s)\|_{L_x^2},
\end{aligned}
\]
it follows from \eqref{est-U-L2-sharp}, \eqref{abc-nonzero}, and
\[
|U(0)|
\le
C\|w_l\tilde f_0\|_{L^\infty_{x,v}}
\]
that \eqref{est-c-L2-sharp} holds. This completes the proof.
\end{proof}

The following proposition provides the $L^2$-bounds needed in
\eqref{est-zero-sumh2} and \eqref{est-v-sumh2}.

\begin{proposition}\label{prop-zero-g2-L2}
Let $[g_1,g_2]$ be the solution to \eqref{g1-eq}--\eqref{g2-eq}
satisfying the a priori assumption \eqref{assa}. Let $N\ge1$ and $l>3$. Then, for any sufficiently small $\kappa>0$, we have
\begin{equation}\label{est-L2-zeroth-order-g2}
\sup_{0\le s\le t}\|g_2(s)\|_{L^2_{x,v}}
\le
C\|w_l \tilde f_0\|_{L^\infty_{x,v}}
+
C\mathfrak H_1^{0,0}(t)
+
C\sum_{|m|=1}
\Big(
\mathfrak H_1^{m,0}(t)
+
\mathfrak H_2^{m,0}(t)
\Big),
\end{equation}
and
\begin{align}
\sum_{|n|\le N}
\kappa^{|n|}
\sup_{0\le s\le t}
\|\partial_v^n g_2(s)\|_{L^2_{x,v}}
&\le
C\|w_l \tilde f_0\|_{L^\infty_{x,v}}
+
C
\sum_{|n|\le N}
\kappa^{|n|}
\mathfrak H_1^{0,n}(t)
+
C
\sum_{|m|=1}
\Big(
\mathfrak H_1^{m,0}(t)
+
\mathfrak H_2^{m,0}(t)
\Big)
\notag\\
&\quad+
C(\alpha+\kappa)
\sum_{\substack{0<|m|+|n|\le N\\ |m|\le1}}
\kappa^{|n|}
\mathfrak H_2^{m,n}(t).
\label{est-L2-g2-mzero}
\end{align}
Here $\mathfrak H_1^{m,n}$ and $\mathfrak H_2^{m,n}$ are defined in
\eqref{eq-def-mfh12mn}.
\end{proposition}

\begin{proof}
Estimate \eqref{est-L2-zeroth-order-g2} follows directly from
Lemmas~\ref{lem-zero-g2-L2} and \ref{lem-Uc-L2}. We next prove
\eqref{est-L2-g2-mzero}. Let $0<|n|\le N$. Taking $m=0$ in
\eqref{est-micro-g2mn} of Lemma~\ref{lem-micro-g2}, we have
\begin{align}
\frac{d}{dt}\|\partial_v^n g_2\|_{L^2_{x,v}}^2
+\delta_0\|\partial_v^n g_2\|_{L^2_{x,v}}^2
&\le
C\|c\|_{L_x^2}^2
+
C\|\mathbf P_1g_2\|_{L^2_{x,v}}^2
+
C\|\nabla_x [a,\mathbf b]\|_{L_x^2}^2
+
C\sum_{n'\le n}\|h_1^{0,n'}\|_{L^\infty_{x,v}}^2
\notag\\
&\quad
+
C e^{2\beta t}
\sum_{i=1}^3
\mathbf 1_{\{n_i>0\}}n_i^2
\|h_2^{e_i,n-e_i}\|_{L^\infty_{x,v}}^2
\notag\\
&\quad
+
C\alpha^2\mathbf 1_{\{n_2>0\}}n_2^2
\|h_2^{0,n-e_2+e_1}\|_{L^\infty_{x,v}}^2 .
\label{est-micro-g2-0n}
\end{align}
Applying Gr\"onwall inequality to \eqref{est-micro-g2-0n}, and then using \eqref{est-c-L2-sharp} in Lemma~\ref{lem-Uc-L2} together with \eqref{est-L2-zeroth-order-g2}, we obtain
\begin{align}
\sup_{0\le s\le t}\|\partial_v^n g_2(s)\|_{L^2_{x,v}}
\le{}&
C\|w_l \tilde f_0\|_{L^\infty_{x,v}}
+
C\sum_{n'\le n}
\mathfrak H_1^{0,n'}(t)
+
C\sum_{|m|=1}
\Big(
\mathfrak H_1^{m,0}(t)
+
\mathfrak H_2^{m,0}(t)
\Big)
\notag\\
&+
C
\sum_{i=1}^3
\mathbf 1_{\{n_i>0\}}n_i
\mathfrak H_2^{e_i,n-e_i}(t)
+
C\alpha\mathbf 1_{\{n_2>0\}}
\mathfrak H_2^{0,n-e_2+e_1}(t).
\label{est-L2-g2-fixed-mzero}
\end{align}

Multiplying \eqref{est-L2-g2-fixed-mzero} by $\kappa^{|n|}$, summing over
all $0<|n|\le N$, and then adding \eqref{est-L2-zeroth-order-g2}, we obtain
\eqref{est-L2-g2-mzero}. This completes the proof.
\end{proof}

\section{Global Existence and Large-Time Behavior}\label{Sec5}

In this section, we conclude the proofs of the main theorems by combining the local well-posedness theory established in Appendix~\ref{app-B} with the a priori estimates obtained in the previous sections. In particular, the weighted $L^\infty_{x,v}$ bounds and the energy-type $L^2$ estimates yield the global-in-time existence and uniqueness of solutions, together with the exponential decay of the perturbation in the weighted $L^\infty_{x,v}$ norm, and further imply the nonnegativity of the global solution. Finally, we establish the large-time asymptotic behavior of the renormalized total energy, which provides a precise characterization of the shear-induced energy growth.

\begin{proof}[\itshape\bfseries Proof of Theorem~\ref{thm1.1}]
\underline{\textbf{Step 1. Existence and Uniqueness}}. We prove existence and uniqueness by dividing the argument into three cases. The proof relies on Lemma~\ref{lem-Linfty-sum-h1h2}, Proposition~\ref{prop-EN1g2}, and Proposition~\ref{prop-zero-g2-L2}.

\noindent{\textbf{Case 1.} $|m|>0.$}
Taking $M_4$ times \eqref{est-positive-sumh1} plus
\eqref{est-positive-sumh2}, and then using \eqref{est-L2-g2-high} to control
the $L^2$-term, we obtain
\[
\begin{aligned}
&
M_4
\sum_{\substack{|m|+|n|\le N\\ |m|>0}}
\kappa^{|n|}
\mathfrak H_1^{m,n}(t)
+
\sum_{\substack{|m|+|n|\le N\\ |m|>0}}
\kappa^{|n|}
\mathfrak H_2^{m,n}(t)
\\
&\le
CM_4
\sum_{\substack{|m|+|n|\le N\\ |m|>0}}
\kappa^{|n|}
\big\|w_l\partial_n^m \tilde f_0\big\|_{L^\infty_{x,v}}
+
C
\sum_{0<|m|\le N}
\big\|w_l\partial_x^m\tilde f_0\big\|_{L^\infty_{x,v}}
\\
&\quad
+
C
\sum_{\substack{|m|+|n|\le N\\ |m|>0}}
\kappa^{|n|}
\mathfrak H_1^{m,n}(t)
+
\Big[
CM_4(\eta+\alpha)
+
C(\eta+\alpha+\kappa)
\Big]
\sum_{\substack{|m|+|n|\le N\\ |m|>0}}
\kappa^{|n|}
\mathfrak H_2^{m,n}(t).
\end{aligned}
\]
Choosing $M_4>0$ sufficiently large, and then taking $\kappa,\eta,\alpha>0$ sufficiently small. Hence
\begin{equation}\label{est-case1}
\sum_{\substack{|m|+|n|\le N\\ |m|>0}}
\Big(
\mathfrak H_1^{m,n}(t)
+
\mathfrak H_2^{m,n}(t)
\Big)
\le
C
\sum_{\substack{|m|+|n|\le N\\ |m|>0}}
\big\|w_l\partial_n^m \tilde f_0\big\|_{L^\infty_{x,v}}.
\end{equation}

\noindent{\textbf{Case 2.} $m=n=0.$}
Taking $M_5$ times \eqref{est-zero-sumh1} plus
\eqref{est-zero-sumh2}, and using \eqref{est-L2-zeroth-order-g2}, we get
\[
\begin{aligned}
M_5\mathfrak H_1^{0,0}(t)+\mathfrak H_2^{0,0}(t)
\le{}&
C\|w_l\tilde f_0\|_{L^\infty_{x,v}}
+
C\mathfrak H_1^{0,0}(t)
+
CM_5(\eta+\alpha)\mathfrak H_2^{0,0}(t)
\\
&+
C
\sum_{|m|=1}
\Big(
\mathfrak H_1^{m,0}(t)
+
\mathfrak H_2^{m,0}(t)
\Big).
\end{aligned}
\]
Choosing $M_5>0$ sufficiently large and then $\eta,\alpha>0$ sufficiently small, together
with \eqref{est-case1}, yields
\begin{equation}\label{est-case2}
\mathfrak H_1^{0,0}(t)
+
\mathfrak H_2^{0,0}(t)
\le
C
\sum_{0\le |m|\le1}
\big\|w_l\partial_x^m\tilde f_0\big\|_{L^\infty_{x,v}}.
\end{equation}

\noindent{\textbf{Case 3.} $m=0,\ 0<|n|\le N.$}
Taking $M_6$ times \eqref{est-v-sumh1} plus
\eqref{est-v-sumh2}, and using \eqref{est-L2-g2-mzero}, we obtain
\[
\begin{aligned}
&
M_6
\sum_{0<|n|\le N}
\kappa^{|n|}
\mathfrak H_1^{0,n}(t)
+
\sum_{0<|n|\le N}
\kappa^{|n|}
\mathfrak H_2^{0,n}(t)
\\
&\le
C
\sum_{|n|\le N}
\kappa^{|n|}
\big\|w_l\partial_v^n\tilde f_0\big\|_{L^\infty_{x,v}}
+
C
\sum_{|n|\le N}
\kappa^{|n|}
\mathfrak H_1^{0,n}(t)
\\
&\quad
+
\Big[
CM_6(\eta+\alpha)
+
C(\eta+\alpha+\kappa)
\Big]
\sum_{|n|\le N}
\kappa^{|n|}
\mathfrak H_2^{0,n}(t)
\\
&\quad
+
C\kappa
\sum_{\substack{|m|+|n|\le N\\ |m|=1}}
\kappa^{|n|}
\mathfrak H_1^{m,n}(t)
+
C(\alpha+\kappa)
\sum_{\substack{|m|+|n|\le N\\ |m|=1}}
\kappa^{|n|}
\mathfrak H_2^{m,n}(t)
\\
&\quad
+
C
\sum_{|m|=1}
\Big(
\mathfrak H_1^{m,0}(t)
+
\mathfrak H_2^{m,0}(t)
\Big).
\end{aligned}
\]
Choosing $M_6>0$ sufficiently large and then taking $\kappa,\eta,\alpha>0$ sufficiently
small, together with \eqref{est-case1} and \eqref{est-case2}, yields
\begin{equation}\label{est-case3}
\sum_{0<|n|\le N}
\Big(
\mathfrak H_1^{0,n}(t)
+
\mathfrak H_2^{0,n}(t)
\Big)
\le
C
\sum_{\substack{0\le |m|+|n|\le N\\ |m|\le1}}
\big\|w_l\partial_n^m\tilde f_0\big\|_{L^\infty_{x,v}}.
\end{equation}
By $\tilde f=g_1+\mu^{\frac{1}{2}}g_2$, for any $|m|+|n|\le N$, we have
\[
\|w_l\partial_n^m\tilde f(t)\|_{L^\infty_{x,v}}
\le
C\|h_1^{m,n}(t)\|_{L^\infty_{x,v}}
+
C\sum_{n'\le n}
\|h_2^{m,n'}(t)\|_{L^\infty_{x,v}}.
\]
Combining \eqref{est-case1}--\eqref{est-case3}, we obtain \eqref{thm1.1-est-1}--\eqref{thm1.1-est-3}.

Finally, we estimate the decay of the electric field. Taking $M_7$ times \eqref{est-positive-n0-sumh1} plus
\eqref{est-positive-n0-sumh2}, and using \eqref{est-L2-g2-high-n0}, we get
\[
\begin{aligned}
&M_7\sum_{0<|m|\le N}\mathfrak H_1^{m,0}(t)
+
\sum_{0<|m|\le N}\mathfrak H_2^{m,0}(t)
\\
&\le
C\sum_{0<|m|\le N}
\big\|w_l\partial_x^m\tilde f_0\big\|_{L^\infty_{x,v}}
+
C\sum_{0<|m|\le N}\mathfrak H_1^{m,0}(t)
+
C(M_7+1)(\eta+\alpha)
\sum_{0<|m|\le N}\mathfrak H_2^{m,0}(t).
\end{aligned}
\]
Choosing $M_7>0$ sufficiently large and then $\eta,\alpha>0$ sufficiently small gives
\begin{equation}\label{est-case0}
\sum_{0<|m|\le N}
\Big(
\mathfrak H_1^{m,0}(t)
+
\mathfrak H_2^{m,0}(t)
\Big)
\le
C
\sum_{0<|m|\le N}
\big\|w_l\partial_x^m\tilde f_0\big\|_{L^\infty_{x,v}}.
\end{equation}
Since $P_0^x\partial_{x_i}\phi=0$, for any $x\in\mathbb T^3$,
\[
|\partial_{x_i}\phi(t,x)|
\leq
\int_{\mathbb T^3}
\big|
\partial_{x_i}\phi(t,x)-\partial_{y_i}\phi(t,y)
\big|\,dy
\leq
C\|\nabla_x\partial_{x_i}\phi(t)\|_{L_x^\infty}.
\]
Taking the supremum over $x\in\mathbb T^3$ and $1\leq i\leq3$, we obtain
\[
\|\nabla_x\phi(t)\|_{L_x^\infty}
\leq
C\|\nabla_x^2\phi(t)\|_{L_x^\infty}.
\]
Combining this with Lemma~\ref{lemma-poisson-3D} and \eqref{est-case0} yields \eqref{thm1.1-est-4}.

\noindent{\underline{\textbf{Step 2. Nonnegativity}}}.
Following the standard positivity-preserving frozen iteration for the distribution function as in \cite[Section~5]{Guo10}, we define the closed subset
\[
\mathcal X_T
=
\Big\{
[\mathcal G_1,\mathcal G_2]\in\widetilde{\mathcal Y}_T
\;\Big|\;
\mathcal G_1(0,x,v)=\tilde f_0(x,v),\quad
\mathcal G_2(0,x,v)=0,\quad
G+\mathcal G_1+\mu^{\frac12}\mathcal G_2\ge0
\Big\}.
\]
For any $[g_1^n,g_2^n]\in\mathcal X_T$, set
\[
f^n=G+g_1^n+\mu^{\frac12}g_2^n\ge0,
\]
and let $\phi^n$ be the unique zero-mean solution of
\[
\Delta_x\phi^n(t,x)
=\int_{\mathbb R^3}f^n(t,x,v)\,dv-1,
\quad
\int_{\mathbb T^3}\phi^n(t,x)\,dx=0.
\]
We consider the following frozen problem for $f^{n+1}$:
\begin{equation}
\label{eq-pos-frozen}
\left\{
\begin{aligned}
&\partial_t f^{n+1}
+e^{\beta t}v\cdot\nabla_x f^{n+1}
-\beta\nabla_v\cdot(vf^{n+1})
-\alpha v_2\partial_{v_1}f^{n+1}\\
&\quad
+e^{-\beta t}\nabla_x\phi^n\cdot\nabla_v f^{n+1}
+Q_-(f^{n+1},f^n)
=
Q_+(f^n,f^n),\\
&f^{n+1}(0,x,v)=f_0(x,v)\ge0.
\end{aligned}
\right.
\end{equation}
The above frozen iteration can be carried out locally by the same argument as in the proof of the local existence theorem in Appendix~\ref{app-B}, and we omit the repeated details. The standard positivity argument for \eqref{eq-pos-frozen} yields $f^{n+1}(t,x,v)\ge0$. Hence, the frozen iteration preserves nonnegativity. Combining this local nonnegativity with the global continuation argument established above, we conclude that the global solution satisfies
\[
f(t,x,v)
=
G(v)+g_1(t,x,v)+\mu^{\frac12}g_2(t,x,v)
\ge0,
\quad t\ge0.
\]
This completes the proof.
\end{proof}

\begin{proof}[\itshape\bfseries Proof of Theorem~\ref{thm1.2}]
By \eqref{zero-ODE-matrix}, the zero-frequency dynamics satisfies
\begin{equation}\label{zero-ODE-matrix-2}
\frac{d}{dt}U(t)+A_\alpha U(t)=R(t),
\end{equation}
where $U(t), A_\alpha$ are defined in \eqref{def-U-Aa}. By the spectral analysis of $A_\alpha$ in the proof of
Lemma~\ref{lem-Uc-L2}, its eigenvalues are
\[
0,\quad \lambda_\pm=2b_0+3\beta\pm i\,\omega_\alpha,
\]
with
\[
\Re\lambda_\pm=2b_0+3\beta>0.
\]
Recalling that $Q_\alpha$, defined in \eqref{def-Qa}, diagonalizes
$A_\alpha$ as
\[
Q_\alpha^{-1}A_\alpha Q_\alpha
=
\operatorname{Diag}(0,\lambda_+,\lambda_-),
\]
we may write
\[
S_\alpha(t)
:=
e^{-tA_\alpha}
=
Q_\alpha
\begin{pmatrix}
1&0&0\\
0&e^{-\lambda_+ t}&0\\
0&0&e^{-\lambda_- t}
\end{pmatrix}
Q_\alpha^{-1}.
\]
Taking $t\to\infty$ in the above representation, we obtain
\begin{equation}\label{def-Pi0}
\Pi_0
:=\lim_{t\to\infty}S_\alpha(t)
=
Q_\alpha
\begin{pmatrix}
1&0&0\\
0&0&0\\
0&0&0
\end{pmatrix}
Q_\alpha^{-1}.
\end{equation}
Consequently,
\begin{align}
\|S_\alpha(t)-\Pi_0\|_{\max}
&\le
C\|Q_\alpha\|_{\max}
\left\|
\begin{pmatrix}
0&0&0\\
0&e^{-\lambda_+ t}&0\\
0&0&e^{-\lambda_- t}
\end{pmatrix}
\right\|_{\max}
\|Q_\alpha^{-1}\|_{\max}\notag
\\
&\le
C\max\{|e^{-\lambda_+ t}|,|e^{-\lambda_- t}|\}\notag
\\
&\le
Ce^{-(2b_0+3\beta)t}.
\label{S-decay-thm12}
\end{align}
On the other hand, applying the same estimates as those used for
\eqref{A12-source}, \eqref{A22-source}, and \eqref{poisson-source}, together
with the global estimates in Theorem~\ref{thm1.1}, we obtain
\begin{equation}\label{R-decay-thm12}
|R(t)|\le C\mathfrak I_0^2 e^{-\lambda t}\le C e^{-\lambda t},
\end{equation}
where
\[
\mathfrak I_0=
\sum_{|m|+|n|\le N}
\Big\|w_l\partial_n^m\big[F_0(x,v)-G(v)\big]\Big\|_{L^\infty_{x,v}}
\le\varepsilon_0.
\]

Applying Duhamel's formula to \eqref{zero-ODE-matrix-2}, we write
\[
U(t)=S_\alpha(t)U(0)+\int_0^t S_\alpha(t-s)R(s)\,ds.
\]
Define
\begin{equation}\label{def-U-infty}
U(\infty):=\Pi_0U(0)+\int_0^\infty \Pi_0R(s)\,ds.
\end{equation}
Then
\[
U(t)-U(\infty)
=
(S_\alpha(t)-\Pi_0)U(0)
+\int_0^t\big(S_\alpha(t-s)-\Pi_0\big)R(s)\,ds
-\int_t^\infty \Pi_0R(s)\,ds.
\]
Let
\[
0<\tilde\lambda<\min\{2b_0+3\beta,\lambda\}.
\]
Using \eqref{S-decay-thm12} and \eqref{R-decay-thm12}, we obtain
\[
\big|(S_\alpha(t)-\Pi_0)U(0)\big|\le Ce^{-\tilde\lambda t},
\]
\[
\left|\int_0^t\big(S_\alpha(t-s)-\Pi_0\big)R(s)\,ds\right|
\le
C\int_0^t e^{-(2b_0+3\beta)(t-s)}e^{-\lambda s}\,ds
\le
Ce^{-\tilde\lambda t},
\]
and
\[
\left|\int_t^\infty \Pi_0R(s)\,ds\right|
\le
C\int_t^\infty e^{-\lambda s}\,ds
\le
Ce^{-\tilde\lambda t}.
\]
Consequently,
\begin{equation}\label{U-limit-thm12}
|U(t)-U(\infty)|\le Ce^{-\tilde\lambda t},
\quad t\ge0.
\end{equation}
Setting
\[
\mathcal E_\alpha(\infty)=(U(\infty))_1,
\]
and recalling that $\mathcal E_\alpha(t)$ is the first component of $U(t)$, we infer from \eqref{U-limit-thm12} that
\[
|\mathcal E_\alpha(t)-\mathcal E_\alpha(\infty)|
\le
|U(t)-U(\infty)|
\le
Ce^{-\tilde\lambda t},
\quad t\ge0.
\]

We next give a more explicit representation of $\mathcal E_\alpha(\infty)$. By \eqref{alpha-beta-relation} and \eqref{def-Pi0}, a direct computation gives
\begin{equation}\label{Pi0-first-row}
(\Pi_0)_1
=
\bigg(
\frac{(b_0+\beta)^2}{b_0(b_0+3\beta)},\,
-\frac{\alpha}{b_0+3\beta},\,
\frac{\alpha^2}{2(b_0+\beta)(b_0+3\beta)}
\bigg).
\end{equation}
Since $0<\alpha\ll1$ and $\beta=O(\alpha^2)$, the coefficients in \eqref{Pi0-first-row} are uniformly bounded. Hence, by \eqref{R-decay-thm12},
\[
\int_0^\infty |(\Pi_0)_1R(s)|\,ds
\le
C\int_0^\infty |R(s)|\,ds
\le
C\mathfrak I_0^2.
\]
Moreover, the initial field energy satisfies
\[
\|\nabla_x\phi(0)\|_{L_x^2}^2\le C\mathfrak I_0^2.
\]
Taking the first component in \eqref{def-U-infty}, we obtain
\begin{equation}\label{Einf-exact}
\mathcal E_\alpha(\infty)
=
\frac{\sqrt6(b_0+\beta)^2}{b_0(b_0+3\beta)}
P_0^xc(0)
-\frac{\alpha}{b_0+3\beta}P_0^xd_{12}(0)
+
\frac{\alpha^2}{2(b_0+\beta)(b_0+3\beta)}
P_0^xd_{22}(0)
+
O(\mathfrak I_0^2).
\end{equation}

Finally, by the definitions of $c$ and $d_{ij}$ in \eqref{def-abc} and \eqref{def-dij-qi}, respectively, together with the rescaling of variables, we have
\[
\left\{
\begin{aligned}
\sqrt6\,P_0^x c(t)
&=
e^{-2\beta t}\int_{\mathbb T^3}\int_{\mathbb R^3}|v|^2F(t,x,v)\,dvdx
-\int_{\mathbb R^3}|v|^2G(v)\,dv,\\
P_0^xd_{12}(0)
&=
\int_{\mathbb T^3}\int_{\mathbb R^3}
\big[F_0(x,v)-G(v)\big]v_1v_2\,dvdx,\\
P_0^xd_{22}(0)
&=
\int_{\mathbb T^3}\int_{\mathbb R^3}
\big[F_0(x,v)-G(v)\big]\Big(v_2^2-\frac{1}{3}|v|^2\Big)\,dvdx.
\end{aligned}
\right.
\]
Combining the convergence estimate above with these identities yields the assertions of Theorem~\ref{thm1.2}. This completes the proof.
\end{proof}

\appendix

\section{Local Existence}\label{app-B}

\begin{lemma}[Local existence]\label{lem-local-existence}
Assume that
\[
\sum_{|m|+|n|\leq N}
\big\|w_l\partial_n^m\tilde f_0\big\|_{L^\infty_{x,v}}
<\infty.
\]
Then there exists $T^*>0$ such that the coupled system
\eqref{g1-eq}--\eqref{g2-eq} admits a unique solution
$[g_1,g_2]$ on $[0,T^*]$ satisfying
\[
\|[g_1,g_2]\|_{\widetilde{\mathcal Y}_{T^*}}
\leq
2\sum_{|m|+|n|\leq N}
\big\|w_l\partial_n^m\tilde f_0\big\|_{L^\infty_{x,v}}.
\]
\end{lemma}

\begin{proof}
We  define the closed subset of $\widetilde{\mathcal Y}_{T}$ as
\[
\mathcal Y_T
:=
\bigg\{
[\mathcal G_1,\mathcal G_2]\in \widetilde{\mathcal Y}_T\ \Big|\
\big\|[\mathcal G_1,\mathcal G_2]\big\|_{\widetilde{\mathcal Y}_T}
\le
2\sum_{|m|+|n|\le N}\big\|w_l\partial_n^m\tilde f_0\big\|_{L^\infty_{x,v}}
\bigg\}.
\]

For any $[g_1,g_2]\in\mathcal Y_T$, we define the nonlinear mapping
$\mathcal N$ through the Duhamel formulas
\eqref{hmn-g1}--\eqref{hmn-g2}, so that the coupled system is
reformulated as
\[
[g_1,g_2]=\mathcal N[g_1,g_2].
\]

\noindent
\underline{\textbf{Step 1. Self-mapping.}}
By the large-velocity estimate in Lemma~\ref{lem-DL-K}, Lemma~\ref{lemma-poisson-3D}, the bilinear estimate for $\widetilde H(g_1,g_2)$, and
the weighted bounds for $\mu^{\frac 12}G_1$ in
Lemma~\ref{lem-DL-G}, we obtain
\[
\begin{aligned}
&\sum_{|m|+|n|\leq N}
\sup_{0\leq t\leq T}
\big\|w_l\partial_n^m g_1(t)\big\|_{L^\infty_{x,v}}
\\
&\le
\sum_{|m|+|n|\leq N}
\big\|w_l\partial_n^m\tilde f_0\big\|_{L^\infty_{x,v}}
+
CT\Big(1+\alpha+\frac1l\Big)
\|[g_1,g_2]\|_{\widetilde{\mathcal Y}_T}
+
CT\|[g_1,g_2]\|_{\widetilde{\mathcal Y}_T}^2,
\end{aligned}
\]
and
\[
\sum_{|m|+|n|\leq N}
\sup_{0\leq t\leq T}
\big\|w_l\partial_n^m g_2(t)\big\|_{L^\infty_{x,v}}
\le
CT(1+\alpha)
\|[g_1,g_2]\|_{\widetilde{\mathcal Y}_T}
+
Cl^{-\frac 32}T
\|[g_1,g_2]\|_{\widetilde{\mathcal Y}_T}^2.
\]
Consequently,
\[
\|\mathcal N[g_1,g_2]\|_{\widetilde{\mathcal Y}_T}
\le
\sum_{|m|+|n|\leq N}
\big\|w_l\partial_n^m\tilde f_0\big\|_{L^\infty_{x,v}}
+
CT\Big(1+\alpha+\frac1l\Big)
\|[g_1,g_2]\|_{\widetilde{\mathcal Y}_T}
+
CT\|[g_1,g_2]\|_{\widetilde{\mathcal Y}_T}^2.
\]
Since $[g_1,g_2]\in\mathcal Y_T$, it follows that
\[
\begin{aligned}
\|\mathcal N[g_1,g_2]\|_{\widetilde{\mathcal Y}_T}
&\le
\sum_{|m|+|n|\leq N}
\big\|w_l\partial_n^m\tilde f_0\big\|_{L^\infty_{x,v}}
+
2CT\Big(1+\alpha+\frac1l\Big)
\sum_{|m|+|n|\leq N}
\big\|w_l\partial_n^m\tilde f_0\big\|_{L^\infty_{x,v}}
\\
&\quad+
4CT
\Big(
\sum_{|m|+|n|\leq N}
\big\|w_l\partial_n^m\tilde f_0\big\|_{L^\infty_{x,v}}
\Big)^2.
\end{aligned}
\]
Taking $l$ sufficiently large, then $\alpha>0$ sufficiently small,
and finally $T>0$ sufficiently small, we obtain
\begin{equation}
\|\mathcal N[g_1,g_2]\|_{\widetilde{\mathcal Y}_T}
\le
2\sum_{|m|+|n|\leq N}
\big\|w_l\partial_n^m\tilde f_0\big\|_{L^\infty_{x,v}}.
\end{equation}
The initial conditions follow directly from the Duhamel formulas, and
the moment constraints are preserved by the corresponding conservation
laws. Hence $\mathcal N$ maps $\mathcal Y_T$ into itself.

\noindent
\underline{\textbf{Step 2. Contraction.}}
Let
\[
[g_1,g_2],\ [\widehat g_1,\widehat g_2]\in\mathcal Y_T.
\]
Applying the preceding estimates to the difference system and using
the bilinear structure of $\widetilde H$, together with
\eqref{Poisson-g1g2} and Lemma~\ref{lemma-poisson-3D}, we obtain
\[
\begin{aligned}
\|\mathcal N[g_1,g_2]
-\mathcal N[\widehat g_1,\widehat g_2]\|_{\widetilde{\mathcal Y}_T}
&\le
CT\Big(
1+\alpha+\frac1l
+
\|[g_1,g_2]\|_{\widetilde{\mathcal Y}_T}
+
\|[\widehat g_1,\widehat g_2]\|_{\widetilde{\mathcal Y}_T}
\Big)
\|[g_1,g_2]
-[\widehat g_1,\widehat g_2]\|_{\widetilde{\mathcal Y}_T}
\\
&\le
CT\Big(
1+\alpha+\frac1l
+
\sum_{|m|+|n|\leq N}
\big\|w_l\partial_n^m\tilde f_0\big\|_{L^\infty_{x,v}}
\Big)
\|[g_1,g_2]
-[\widehat g_1,\widehat g_2]\|_{\widetilde{\mathcal Y}_T}.
\end{aligned}
\]
By decreasing $T$ further if necessary, we may assume that
\[
CT\Big(
1+\alpha+\frac1l
+
\sum_{|m|+|n|\leq N}
\big\|w_l\partial_n^m\tilde f_0\big\|_{L^\infty_{x,v}}
\Big)
\le\frac12.
\]
It follows that
\begin{equation}
\|\mathcal N[g_1,g_2]
-\mathcal N[\widehat g_1,\widehat g_2]\|_{\widetilde{\mathcal Y}_T}
\le
\frac12
\|[g_1,g_2]
-[\widehat g_1,\widehat g_2]\|_{\widetilde{\mathcal Y}_T}.
\end{equation}
Thus $\mathcal N$ is a contraction on $\mathcal Y_T$. By the Banach
fixed-point theorem, $\mathcal N$ admits a unique fixed point in
$\mathcal Y_T$, which is the desired solution. This completes the
proof.
\end{proof}

\section{Proof of a Moment Identity for \texorpdfstring{$L$}{L}}\label{app-A}

We recall the following auxiliary identity for Maxwell molecules.

\begin{lemma}[{\cite[Proposition 4.10]{JamesNotaVelazquez19a}}]\label{lem-Tij}
Let
\[
W_{ij}(v)=v_iv_j,\quad 1\le i,j\le3,
\]
and define
\begin{equation}\label{Tij-def}
T_{ij}
=
\frac12\int_{\mathbb S^2}B_0(\cos\theta)
\Big[
W_{ij}(v')+W_{ij}(v_*')-W_{ij}(v)-W_{ij}(v_*)
\Big]\,d\omega.
\end{equation}
Then
\begin{equation}\label{Tij-est}
T_{ij}
=
-b_0\Big[
(v-v_*)_i(v-v_*)_j-\frac{\delta_{ij}}{3}|v-v_*|^2
\Big].
\end{equation}
\end{lemma}

\begin{proof}[\itshape\bfseries Proof of Lemma~\ref{lem-Lmoments}]
The first identity in \eqref{L-moments} is precisely
\cite[Lemma 6.6]{DuanLiu21}. It remains to prove
\begin{equation}\label{L-viv2}
L\big(|v|^2v_i\mu^{\frac12}\big)
=
\frac43b_0(|v|^2-5)v_i\mu^{\frac12}.
\end{equation}
For $f=\mu^{\frac12}W$, the linearized operator can be written as
\[
Lf
=
-\mu^{\frac12}\int_{\mathbb R^3}\mu_*\,dv_*
\int_{\mathbb S^2}B_0(\cos\theta)
\big(W'+W_*'-W-W_*\big)\,d\omega.
\]
Taking $W(v)=|v|^2v_i$, define
\[
T_i
=
\frac12\int_{\mathbb S^2}B_0(\cos\theta)
\big(
|v'|^2v_i'+|v_*'|^2v_{*i}'-|v|^2v_i-|v_*|^2v_{*i}
\big)\,d\omega.
\]
Then
\begin{equation}\label{eq-L-Ti}
L\big(|v|^2v_i\mu^{\frac12}\big)
=
-2\mu^{\frac12}\int_{\mathbb R^3}\mu_*T_i\,dv_*.
\end{equation}

Let
\[
V=\frac{v+v_*}{2},
\quad
u=v-v_*.
\]
A direct computation gives
\begin{equation}\label{eq-Vu-cubic}
|v|^2v_i+|v_*|^2v_{*i}
=
2\Big(
|V|^2+\frac{|u|^2}{4}
\Big)V_i
+
(V\cdot u)u_i
\end{equation}
and
\begin{equation}\label{eq-Vu-quadratic}
v_iv_j+v_{*i}v_{*j}
=
2V_iV_j+\frac12u_iu_j.
\end{equation}
In view of \eqref{eq-Vu-cubic}--\eqref{eq-Vu-quadratic}, together with
$V'=V$ and $|u'|=|u|$, we obtain
\[
|v'|^2v_i'+|v_*'|^2v_{*i}'-|v|^2v_i-|v_*|^2v_{*i}=
2\sum_{j=1}^3V_j
\Big[
\big(v_i'v_j'+v_{*i}'v_{*j}'\big)
-
\big(v_iv_j+v_{*i}v_{*j}\big)
\Big].
\]
Hence
\begin{equation}\label{eq-Ti-Tij}
T_i
=
2\sum_{j=1}^3V_jT_{ij}.
\end{equation}
Combining \eqref{eq-Ti-Tij} with Lemma~\ref{lem-Tij}, we obtain
\begin{equation}\label{eq-Ti-explicit}
\begin{aligned}
T_i
&=
-2b_0\sum_{j=1}^3V_j
\Big(
u_iu_j-\frac{\delta_{ij}}{3}|u|^2
\Big)
\\
&=
-2b_0\Big(
(V\cdot u)u_i-\frac13|u|^2V_i
\Big).
\end{aligned}
\end{equation}
Finally, combining \eqref{eq-L-Ti} and \eqref{eq-Ti-explicit} and using
the standard Gaussian moments, we obtain
\[
\begin{aligned}
L\big(|v|^2v_i\mu^{\frac12}\big)
&=
4b_0\mu^{\frac12}
\int_{\mathbb R^3}\mu_*
\Big(
(V\cdot u)u_i-\frac13|u|^2V_i
\Big)\,dv_*
\\
&=
\frac43b_0(|v|^2-5)v_i\mu^{\frac12}.
\end{aligned}
\]
This completes the proof of Lemma~\ref{lem-Lmoments}.
\end{proof}

\noindent\textbf{Acknowledgements.}
The authors are supported by the special foundation for Guangxi Ba Gui Scholars and the National Natural Science Foundation of China (No.~12171104).

\medskip
\noindent\textbf{Data availability:}
This article does not have any external supporting data.

\section*{Declarations}

\noindent\textbf{Conflict of interest:}
The authors do not have any other competing interests to declare.

\end{document}